\documentclass[12pt]{amsart}

\usepackage{graphicx,amsthm}
\usepackage{enumerate}
\usepackage{empheq}
\usepackage{amssymb}
\usepackage[margin=1in]{geometry}
\usepackage{pst-all}
\usepackage{array}
\usepackage{amsmath,amscd,amsfonts,bbm,mathabx,graphicx}
\allowdisplaybreaks
\usepackage[onehalfspacing]{setspace}

\usepackage[all]{xy} 

\usepackage{caption}
\usepackage{subcaption}

\usepackage{enumitem}

\usepackage[mathscr]{euscript}
\usepackage{booktabs} 
\usepackage{amssymb}
\usepackage{hyperref}

\usepackage{tikz}
\usepackage{tikz-cd}
\usetikzlibrary{arrows.meta,positioning}

\definecolor{darkblue}{RGB}{0,65,120}
\definecolor{specialred}{RGB}{190,20,70}
\definecolor{grayedge}{RGB}{90,90,90}
\definecolor{lightpink}{RGB}{250,235,245}

\newtheorem{theorem}{Theorem}[section]
\newtheorem{lemma}[theorem]{Lemma}
\newtheorem{proposition}[theorem]{Proposition}
\newtheorem{corollary}[theorem]{Corollary}
\newtheorem{conjecture}[theorem]{Conjecture}

\theoremstyle{definition}
\newtheorem{definition}[theorem]{Definition}
\newtheorem{example}[theorem]{Example}

\newtheorem{remark}[theorem]{Remark}

\numberwithin{equation}{section}

\newtheorem{thmy}{Theorem}
\newenvironment{thmx}{\stepcounter{theorem}\begin{thmy}}{\end{thmy}}

\newcommand{\C}{\mathbb{C}}

\DeclareMathOperator{\GL}{GL}

\newcommand{\flag}[1]{\mathcal{F}\ell(\mathbb{C}^{#1})} 

\newcommand{\swh}[1]{\sigma_{#1,h}}
\DeclareMathOperator{\codim}{\codim}

\newcommand{\hpat}[1]{{#1}_h} 
\DeclareMathOperator{\diag}{diag} 
\DeclareMathOperator{\Stab}{Stab} 
\DeclareMathOperator{\supp}{supp} 
\DeclareMathOperator{\Hess}{Hess} 
\DeclareMathOperator{\asc}{asc} 
\DeclareMathOperator{\ch}{ch} 

\allowdisplaybreaks

\newcommand{\cblock}[1]{\colorbox{gray!20}{$#1$}}

\newcommand{\solidbar}{%
  \,\tikz[baseline=-0.5ex]{
    \draw[line width=0.4pt] (0,-1.5ex) -- (0,1.5ex);
  }\,%
}

\newcommand{\dashbar}{%
  \,\tikz[baseline=-0.5ex]{
    \draw[dashed, line width=0.4pt] (0,-1.5ex) -- (0,1.5ex);
  }\,%
}
\newcommand{\mystr}{\rule[-6pt]{0pt}{20pt}}

\newcommand{\bb}{\mathbf{b}}

\newcommand\set[1]{\{ #1 \}}

\newcommand{\Sn}[1]{\mathfrak{S}_{#1}} 
\newcommand{\Ahat}[1]{ \widehat{\mathcal A}_{s_{#1},w} }
\newcommand{\Atil}[1]{ \widetilde{\mathcal A}_{s_{#1},w} }
\newcommand{\Gh}{\mathcal{G}_{h}} 
\newcommand{\lh}[1]{\ell_{h}(#1)} 
\newcommand\dra{\dashrightarrow}
\newcommand\ra{\rightarrow}

\newcommand\da{\downarrow}

\newcommand{\bx}{{\mathbf x}}
\newcommand{\up}{\!\!\uparrow}

\definecolor{cadmiumgreen}{rgb}{0.0, 0.42, 0.24}

\newcommand{\x}{\mathbf{x}}

\begin{document}

\title[Hessenberg varieties associated with lollipop graphs]{Permutation module decomposition of the cohomology of Hessenberg varieties associated with lollipop graphs}

\author{Soojin Cho}
\address{Department of Mathematics, Ajou University, Suwon  16499, Republic of Korea}
\email{chosj@ajou.ac.kr}

\author{Seonjeong Park}
\address{Department of Mathematics Education, Jeonju University, Jeonju 55069, Republic of Korea}
\email{seonjeongpark@jj.ac.kr}

\thanks{This work was supported by the National Research Foundation of Korea [NRF-2020R1A2C1A01011045]. S.P. was also supported by the National Research Foundation of Korea(NRF) grant funded
by the Korea government(MSIT) (RS-2026-25490927).
}

\begin{abstract} 
We study the cohomology of regular semisimple Hessenberg varieties associated with lollipop graphs as a module under the dot action. Using the natural basis introduced by Cho, Hong, and Lee, which we call the CHL basis, we establish structural properties of the dot action, including a result for classes satisfying \(i\)-decomposability. We also obtain an explicit elementary symmetric function expansion of the chromatic quasisymmetric functions of lollipop graphs in terms of \(h\)-admissible permutations and their associated partitions. Combining these geometric and combinatorial results, we construct a permutation module decomposition of the cohomology of the corresponding Hessenberg varieties, thereby proving a conjecture of Cho, Hong, and Lee for lollipop graphs.
\end{abstract}


\maketitle

\setcounter{tocdepth}{1}

\section{Introduction } \label{sec:intro}

Chromatic quasisymmetric functions, introduced by Shareshian and Wachs in \cite{SW}, refine Stanley's chromatic symmetric functions (\cite{S1}):  For a graph $G=(V, E)$ with $V=[k]$,
\[ X_{G}({\bf x};q)= \sum_\kappa  \,q^{\asc(\kappa)}{\bf x}^\kappa\,, \]
where the sum is over all proper colorings $\kappa\colon \,V \rightarrow \mathbb{P}$ of $G$ with colors from positive integers,  ${\bf x}^\kappa=x_{\kappa(1)}\cdots x_{\kappa(k)}$, and $\asc(\kappa)=|\{ i<j \mid \{i,j\}\in E, \kappa(i)<\kappa(j)\}|$.

The Stanley--Stembridge conjecture asserts the \(e\)-positivity of the chromatic symmetric functions of incomparability graphs of \((3+1)\)-free posets. This conjecture has recently
been proved by Hikita (\cite{Hik}), while its refinement for chromatic quasisymmetric functions is still open. 

A \emph{Hessenberg function} is a nondecreasing function $h\colon [k] \rightarrow [k]$ satisfying $h(i)\geq i$ for all $i\in [k]$. The graph associated with $h$ is the graph $G_h=(V,E)$ defined by $V=[k]$ and $E=\{\{i,j\}\mid i<j\leq h(i)\}$. 
We remark that if the underlying graph $G$ is $G_h$ for a  Hessenberg function $h$, then  $X_{G}({\bf x};q)$ is a \emph{symmetric} function with coefficients in $\mathbb{Z}[q]$; see~\cite{SW}.  
The space of symmetric functions of degree $k$ has well known bases indexed by the integer partitions $\lambda$ of $k$ including the  \emph{elementary symmetric functions} $\{e_\lambda\}$ and the \emph{complete homogeneous symmetric functions} $\{h_\lambda\}$.
\begin{conjecture}[Refined Stanley--Stembridge Conjecture \cite{SW}]\label{conj:refinedSS}  Let $h\colon [k] \rightarrow [k]$ be a Hessenberg function, and let $G_h$ be the graph associated with $h$. Then  the chromatic quasisymmetric function $X_{G_h}({\bf x};q)$ of $G_h$ 
is $e$-positive; that is,
$$X_{G_h}({\bf x}; q)=\sum_{\lambda\vdash k} c_\lambda(q) e_\lambda({\bf x})\,,$$
where $c_\lambda(q)$ is a polynomial in $q$ with nonnegative coefficients for each $\lambda$.
\end{conjecture}

A regular semisimple Hessenberg variety $\Hess(S, h)$ is a subvariety of the full flag variety $\flag{k}$, determined by a Hessenberg function $h\colon [k] \rightarrow [k]$ and a regular semisimple operator~$S$. The cohomology space $H^{2d}(\Hess(S, h))$ carries a $\mathbb{C}\Sn{k}$-module structure due to Tymoczko's dot action, where $\Sn{k}$ is the \emph{symmetric group} on $[k]$. 

A remarkable relation between chromatic quasisymmetric functions and the cohomology of Hessenberg varieties is stated in the following theorem.
\begin{theorem}[\cite{BC, G-P, KL}]\label{thm:rel} Let $h\colon [k] \rightarrow [k]$ be a Hessenberg function and $S$ be a regular semisimple operator on $\mathbb{C}^k$. Then
$$\omega X_{G_h}({\bf x};q)= \sum_d \operatorname{ch}\left( H^{2d}(\Hess(S, h))\right) q^d\,,$$ where $\omega$ is the involution on the algebra of symmetric functions sending $e_\lambda$ to the complete homogeneous symmetric function $h_\lambda$ and $\operatorname{ch}$ is the Frobenius characteristic map. 
\end{theorem}

For a partition $\lambda=(\lambda_1, \dots, \lambda_\ell)$ of $k$, let  $\Sn{\lambda}=\Sn{\{1, \dots, \lambda_1\}}\times \cdots \times \Sn{\{\lambda_1+\cdots+\lambda_{\ell-1}+1, \dots, k \}}$ be the corresponding Young subgroup of $\Sn{k}$. 
The \emph{permutation module} $M^\lambda$ is the induced representation
\(
M^\lambda=\operatorname{Ind}_{\Sn{\lambda}}^{\Sn{k}}\mathbf{1},
\) of the trivial representation $\mathbf{1}$ of $\Sn{\lambda}$
and satisfies $\ch(M^\lambda)=h_\lambda({\bf x})$.
Hence, Theorem~\ref{thm:rel} allows us to restate Conjecture~\ref{conj:refinedSS} as follows.
\begin{conjecture}[\cite{SW}]\label{conj:refinedSS'}  Let $h\colon [k] \rightarrow [k]$ be a Hessenberg function. Then for each $d$, the $\mathbb{C}\Sn{k}$-module $H^{2d}(\Hess(S, h))$ is a direct sum of permutation modules $M^\lambda$.
\end{conjecture}

In \cite{CHL1,CHL2}, Cho, Hong, and Lee constructed a natural basis
$\{\sigma_{w,h}\mid w\in \Sn{k}\}$ for $H^*(\Hess(S,h))$ and studied its properties.
In particular, they identified a distinguished set
$\{\sigma_{w,h}\mid w\in \Gh^d\}$ of module generators for $H^{2d}(\Hess(S,h))$, 
where 
\[ \Gh^d\coloneqq \left\{ w\in \Sn{k}\,|\, \ell_h(w)=d, w^{-1}(w(j)+1)\leq h(j) \text{ for all }j\text{ with } w(j)\in [k-1]\right\}\,,\] 
and
\(\ell_h(w)\coloneqq |\{ (i, j)\,|\, i<j \leq h(i)  \text{ and } w(i)>w(j) \}|\,\) is the \emph{$h$-length} of $w$.
They formulated the following conjecture, which is related to the refined Stanley--Stembridge conjecture:
\begin{conjecture}[Conjecture 1.3 in \cite{CHL2}]\label{conj:CHL}
For each $w\in  \Gh^d$, there exists $\widehat{\sigma}_{w,h}\in \mathbb C\Sn{k}(\sigma_{w,h})$ such that $\mathbb C\Sn{k}(\widehat{\sigma}_{w,h})$ is a permutation module and 
$$ H^{2d}(\Hess(S, h))=\bigoplus_{w\in \Gh^d } \mathbb C\Sn{k}(\widehat{\sigma}_{w,h})\,.$$
\end{conjecture} 

Conjecture~\ref{conj:CHL} was proved for the permutohedral variety in \cite{CHL1}, and for $H^{2}(\Hess(S, h))$ with all Hessenberg functions $h$ in \cite{CHL2}. 
One of the main results of this paper is to prove Conjecture~\ref{conj:CHL} 
for the Hessenberg function 
$h=h_{m,n}\colon [k]\to [k]$, where $n=k-m$, defined by
\[
h_{m,n}(i)=
\begin{cases}
m, & 1\leq i\leq m-1,\\
\min(i+1,k), & m\leq i\leq k.
\end{cases}
\]
Then the associated graph $G_h$ is a \emph{lollipop graph}, obtained by joining a complete graph on $m$ vertices to a path graph by an edge. 
It is known that the chromatic quasisymmetric functions associated with lollipop graphs expand as a positive sum of elementary symmetric function basis, but no explicit expansion formula was known.

To prove Conjecture~\ref{conj:CHL}, it is essential to understand the $\Sn{k}$-action, or dot action, on the classes $\sigma_{w,h}$ for $w\in\Gh$; see Section~\ref{subsec:dot action}. 
When $h=h_{m,n}$ corresponds to a lollipop graph, many permutations $w\in \Gh$ have a special property; $w$ \emph{admits an $i$-block decomposition} for some $i$, that is, 
\[
\{w(j) \mid j=1, \dots, i\}=\{k-i+1,\ldots,k\}
\quad\text{and}\quad
\{w(j) \mid j=i+1\ldots, k\}=\{1,\ldots,k-i\}.
\]
This motivates us to study the dot actions on $\sigma_{w,h}$ when $w$ admits an $i$-block decomposition. Our first  main theorem is the following.

\begin{thmx}[Theorem~\ref{thm:separation}]\label{thmx:2} Suppose that $w\in\Sn{k}$ admits an $i$-block decomposition. Let 
\[^{\da i}w =(w(1)-(k-i))\,\, \cdots\,\, (w(i)-(k-i))\in \Sn{i} \text{ \,\,and \,\,  }
w^{\da i} = w(i+1)\, \cdots \, w(k)\in\Sn{k-i}\]
be two permutations given in one-line notation. For a Hessenberg function $h\colon [k] \to  [k]$, define Hessenberg functions $^{\da i}h$ on $[i]$ and $h^{\da i}$ on $[k-i]$ by
\[^{\da i}h(j)=\min(h(j), i) \,\,\,\text{ for \,\,} j=1, \dots, i \,, 
\text{ \,\,and \,\,  }
h^{\da i}(j)=h(i+j)-i \,\,\,\text{ for \,\, } j=1, \dots, k-i.  \]
Then $^{\da i}w\in \mathcal{G}_{^{\da i}h}$ and $w^{\da i}\in \mathcal{G}_{h^{\da i}}$. Moreover, if $j\neq k-i$, then   $s_j$ acts on $\sigma_{w,h}$ blockwise.
\end{thmx}

Our second main theorem gives an explicit \(e\)-positive expansion of
\(X_{L_{m,n}}({\bf x};q)\), described in terms of \(h_{m,n}\)-admissible permutations and their associated partitions.

\begin{thmx}[Theorem~\ref{thm:Acyclic}]\label{thmx:1} Let $h=h_{m,n}$. If we define a map from $\Gh$ to the set of partitions of $k$ as in Definition~\ref{def:lambda}, then 
\[  X_{G_h}({\bf x};q)=\sum_{w\in \Gh} q^{\ell_h(w)} e_{\lambda(w)}({\bf x})\,. \]
Equivalently, for each $d$, we have  
 \[H^{2d}(\Hess(S, h))=\bigoplus_{w\in \Gh^d}\ M^{\lambda(w)}\,.\] 
\end{thmx}

We then construct a permutation module decomposition of the cohomology space of the corresponding Hessenberg variety, thereby proving that  Conjecture~\ref{conj:CHL} holds for lollipop graphs. 
\begin{thmx}[Theorem~\ref{thm:main}]\label{thmx:3}
Let $h=h_{m,n}$ be a Hessenberg function corresponding to a lollipop graph on $k$ vertices.  Then for each $d$, we have
 \[H^{2d}(\Hess(S, h))=\bigoplus_{u\in \Gh^d}\mathbb C\Sn{k}(\widehat{\sigma}_{u,h}) 
 \,.\] 
In particular, $\Stab(\widehat{\sigma}_u)=\Sn{\lambda(u)}$ and $\mathbb C\Sn{k}(\widehat{\sigma}_u)\cong M^{\lambda(u)}$ for every  $u\in \Gh^d$.
\end{thmx}

The present paper is organized as follows. Necessary terminology and background material are introduced in Section~\ref{sec:prelim}.  Section~\ref{sec:dot action} is devoted to investigating the properties of the dot action especially on basis classes admitting block decompositions, and proving Theorem~\ref{thmx:2}. In Section~\ref{sec:A-set}, we continue investigating the properties of the dot action. We then prove our second main theorem on the expansion of the chromatic quasisymmetric functions in Section~\ref{sec:chromatic}.  We use the results from previous sections to construct $\widehat{\sigma}_{w,h}$ for $w\in \Gh^d$ and investigate their properties in Sections~\ref{sec:generators} and ~\ref{sec:new}. In Section~\ref{sec:decomposition} we provide a proof of Theorem~\ref{thmx:3}; that is, we show that $\mathbb{C}\Sn{k}(\widehat{\sigma}_{w,h})$ for $w\in\Gh^d$, form a permutation module decomposition of $H^{2d}(\Hess(S,h))$ when $h$ is the Hessenberg function for the lollipop graph.


\section{Preliminaries}\label{sec:prelim}
 \subsection{Basic terminology and Bruhat order}
For two integers $i<j$, we let $[i, j]\coloneqq \{i, i+1, \dots, j\}$ and $[i]\coloneqq [1, i]=\{1, 2,\dots, i\}$. 
For a set $A$ of integers we use $A\up$ to denote the ordered set $\{a_1<a_2< \cdots <a_m\}$ and for two sets $A$ and $B$ with the same cardinality $m$,  $A\up\leq B\up$ means $a_i\leq b_i$ for all $i=1, \dots, m$. 
  
The \emph{symmetric group on $[k]$} is denoted by $\Sn{k}$ and we usually use one-line notation $$w= w_1\, w_2\,\dots \, w_k= w(1)\, w(2)\,\dots \, w(k)$$ to represent the permutation $w$ where $w_i=w(i)$ is the value of $w$ at $i$. We also use $s_{i,j}$ to denote the transposition of $i$ and $j$, and $s_i$ for the adjacent transposition $s_{i, i+1}$. The \emph{length} $\ell(w)$ of a permutation $w\in \Sn{k}$ is the number of inversions in $w$:
\[\ell(w)\coloneqq |\{ (i, j)\,|\, i<j  \text{ and } w_i>w_j \}|\,.\] 
The \emph{Bruhat order} on $\Sn{k}$ is the transitive closure of the following relations:
\[u< v  \quad \text{ if and only if } \quad  v=us_{i,j} \text{ for some $i, j$  and  } \ell(u)<\ell(v)\,.\]
For $u, v\in \Sn{k}$ with $u\leq v$, we use $[u,v]$ to denote the set $\{w\in \Sn{k} \,|\, u\leq w\leq v\}$.
The unique maximal element in $\Sn{k}$ in the Bruhat order is denoted by $w_0$; it is given by $w_0=k\, (k-1)\, \cdots \, 2\, 1$. 
For a permutation $w\in\Sn{k}$, we define the \emph{left descent set} and the \emph{right descent set} of $w$ by
\[
D_L(w)\coloneqq \{\, i \mid \ell(s_i w)<\ell(w)\,\},
\qquad
D_R(w)\coloneqq \{\, i \mid \ell(w s_i)<\ell(w)\,\},
\]
respectively.
The following proposition is useful for determining the comparability of permutations in the Bruhat order. 

\begin{proposition}[{\cite[Theorem 2.6.4]{BB}}]\label{prop:Bruhat} For two permutations $u$ and $v$ in $\Sn{k}$, the following statements are equivalent:
\begin{enumerate}
\item $u\leq v$ in the Bruhat order.
\item $\{u_1, u_2,\dots, u_i\}\up \leq \{v_1, v_2,\dots, v_i\}\up$ for all $i=1, 2, \dots, k$.
\item $\{u_{k-i}, u_{k-i+1}, \dots , u_k\}\up \geq \{v_{k-i}, v_{k-i+1}, \dots , v_k\}\up$ for all $i=0, 1, \dots, k-1$.
\item  $\{u_1, u_2,\dots, u_i\}\up \leq \{v_1, v_2,\dots, v_i\}\up$ for all $i \in [k-1]-D_R(v)$.
\end{enumerate}
\end{proposition}

\begin{proposition}[Lifting Property, \cite{BB}]\label{prop:lifting}
Suppose that $u<v$ and $i\in D_L(v)-D_L(u)$. Then $u\leq s_i v$ and $s_iu \leq v$.
\end{proposition}

 \subsection{Hessenberg varieties} 
We fix a linear (regular semisimple) operator $S$ on $\mathbb{C}^k$ which is represented by a diagonal matrix with distinct diagonal entries. Let $h\colon [k] \rightarrow [k]$ be a \emph{Hessenberg function}, meaning that $h$ is an nondecreasing function satisfying $h(i)\geq i$ for all $i=1, \dots, k$.

Let $\mathfrak{gl}_k(\mathbb{C})$ be the Lie algebra of $\GL_k(\mathbb{C})$. The \emph{Hessenberg space} associated with $h$ is the subspace
\[
H_h
\coloneqq
\{\, X=(x_{ij})\in \mathfrak{gl}_k(\mathbb{C})
\mid
x_{ij}=0 \text{ whenever } i>h(j)\,\}.
\]
The full flag variety $\flag{k}$ can be identified with the homogeneous space $\GL_k(\mathbb{C})/B$, where $B$ is the Borel subgroup of upper triangular matrices in $\GL_k(\mathbb{C})$. Under this identification, the regular semisimple \emph{Hessenberg variety} determined by $S$ and $h$ is defined by
\[
\Hess(S,h)
\coloneqq
\{\, gB\in \GL_k(\mathbb{C})/B
\mid
g^{-1}Sg\in H_h\,\}.
\]
If $gB$ corresponds to the flag
\[
V_\bullet=(\{0\}=V_0\subset V_1\subset \cdots \subset V_k=\mathbb{C}^k),
\]
then the condition $g^{-1}Sg\in H_h$ is equivalent to
\[
S(V_i)\subset V_{h(i)}
\qquad\text{for } i=1,\dots,k.
\]
Note that  $\Hess(S, h)$ is a smooth projective variety of dimension $\displaystyle d_h=\sum_{i=1}^k (h(i)-i)$. 

The \emph{$h$-length} $\ell_h(w)$ of a permutation $w\in \Sn{k}$ is the number of $h$-inversions in $w$:
\[\ell_h(w)\coloneqq |\{ (i, j)\,|\, i<j \leq h(i)  \text{ and } w(i)>w(j) \}|\,.\] 
 The \emph{$h$-Bruhat order} on $\Sn{k}$ is the transitive closure of the following relations:
\[u<_h v  \quad \text{ if and only if } \quad  v=u s_{i,j} \text{ for some $i<j\leq h(i)$  and  } \ell(u)<\ell(v)\,.\]
For $u, v\in \Sn{k}$ with $u\leq_h v$, we use $[u,v]_h$ to denote the set $\{w\in \Sn{k} \,|\, u\leq_h w\leq_h v\}$.

\begin{remark}
\begin{itemize}
\item We usually represent a Hessenberg function $h\colon [k] \rightarrow [k]$ by listing its values:
\[
h=(h(1),h(2),\dots,h(k)).
\] 
\item If $h=(k, k, \dots, k)$, then  $\Hess(S, h)=\flag{k}$, $\ell_h=\ell$, and the $h$-Bruhat order coincides with the Bruhat order. 
\item For a Hessenberg function $h\colon [k] \rightarrow [k]$ and $u, v\in \Sn{k}$, if $u<_h v$ then $u< v$.
\item The relation $u <_h v$ does not imply $\ell_h (u)<\ell_h(v)$. For example, let $h=(3,3,4,4)$, $u=2\,3\,1\,4$ and  $v=2\,3\,4\,1$. Then $u <_h v$ since $v=us_{3, 4}$ and $\ell(u)=2<\ell(v)=3$. However, $\ell_h (u)=2$ and $\ell_h(v)=1$.
\end{itemize}
\end{remark}

\begin{proposition}[Chain Property of the $h$-Bruhat order, \cite{CHP}]\label{prop:h-chain}
    Let $h$ be a Hessenberg function on $[k]$. If $u<_h v$ in $\mathfrak S_k$, then 
    there exists a chain
\[
u=x^{(0)}<_h x^{(1)}<_h \cdots <_h x^{(j)}=v
\]
such that $\ell(x^{(i)})=\ell(u)+i$ for all $0\leq i\leq j$.
\end{proposition}

A useful feature of $\GL_k(\mathbb{C})/B$ is that it admits the
 \emph{Bruhat decomposition}: 
\[\GL_k(\mathbb{C})/B=\bigsqcup_{w \in \mathfrak{S}_k} B^-wB/B\,,\] where $B^-$ is the Borel subgroup of lower triangular matrices in $\GL_k(\mathbb{C})$. We let $\Omega_w^\circ\coloneqq B^-wB/B$ be the \emph{(opposite) Schubert cell}. 
 Then  $\Hess(S, h)$ is decomposed into their  \emph{minus cells} $\Omega_{w,h}^{\circ}=\Omega_w^{\circ}\cap\Hess(S, h)$:
\[\Hess(S, h) = \bigsqcup_{w \in \mathfrak{S}_k}\Omega_{w, h}^\circ\,. \]
Here, we note that each minus cell $\Omega_{w,h}^\circ$ is isomorphic to $\mathbb{C}^{d_h-\ell_h(w)}$.

We now recall the GKM description of the ordinary and equivariant cohomology of the Hessenberg variety $\Hess(S,h)$. Throughout this paper, all cohomology and equivariant cohomology groups are taken with coefficients in $\mathbb{C}$ unless otherwise stated.

\begin{definition}\label{def:GKM graph} For a Hessenberg function $h$ on $[k]$, the associated \emph{GKM graph}  is the edge-labeled graph $\Gamma_h=(V, E, \alpha)$, where $V=\Sn{k}$, $E=\{ \{u, v\}\,|\, v=us_{i,j} \text{ for } i<j\leq h(i)\}$, and the edge labeling $\alpha\colon E \rightarrow \mathbb C [t_1, t_2, \dots, t_k]$ is defined as $\alpha(\{u, v \})=t_{u_i}-t_{u_j}$ assuming that $\ell_h(u)<\ell_h(v)$.
\end{definition}

The graph $\Gamma_h$ encodes the one-dimensional $T$-orbits in $\Hess(S,h)$ and gives the following combinatorial model for the equivariant cohomology ring.

\begin{proposition}\label{prop:Hess} Let $h$ be a Hessenberg function on $[k]$. 
\begin{enumerate} 
\item The torus $T=(\mathbb{C}^*)^k$ acts on the Hessenberg variety $\Hess(S, h)$  by left multiplication. The fixed point set $\Hess(S, h)^T$ is naturally identified with $\Sn{k}$.\footnote{We write $T$ for the algebraic torus $(\mathbb{C}^*)^k$. When it is necessary to emphasize the dimension, we write $T^k$.}

\item $\Hess(S, h)$ is a GKM space with the GKM graph $\Gamma_h=(V, E, \alpha)$. Moreover, its $T$-equivariant cohomology ring is identified with
\[H_T^*(\Gamma_h)\coloneqq\left\{(p(v))_{v\in \Sn{k}} \in \bigoplus_{v \in \Sn{k}} \C[t_1,\dots,t_k] \,\middle|\, \alpha(\{u,v\}) \mid p(u) - p(v) \text{ for all } \{u,v\} \in E   \right\} \] 
as a $\C[t_1,\dots,t_k]$-algebra.
\end{enumerate}
\end{proposition}

The ordinary cohomology ring is obtained from the equivariant cohomology ring by setting the equivariant parameters equal to zero.

\begin{proposition} The ordinary cohomology ring $H^*(\Hess(S, h))$ of  $\Hess(S, h)$ is isomorphic to $$H^*(\Gamma_h)\coloneqq H_T^*(\Gamma_h)/(t_1, \dots, t_k),$$ 
where $t_i$ denotes the element of $H_T^*(\Gamma_h)$ whose value at each vertex $v\in\Sn{k}$ is $t_i$, and $(t_1,\dots,t_k)$ is the ideal of $H_T^*(\Gamma_h)$ generated by $t_1,\dots,t_k$.
\end{proposition}

We identify $H_T^*(\Hess(S, h))$ with $H^*_T(\Gamma_h)$ and $H^*(\Hess(S, h))$ with $H^*(\Gamma_h)$. Under these identifications, classes are often regarded as collections
of polynomials indexed by permutations
in $\Sn{k}$.
We define the \emph{support} of an element $\sigma\in H_T^*(\Hess(S, h))$ by
\[ \supp(\sigma)\coloneqq\{ v\in \Sn{k}\,|\, \sigma(v)\neq 0\}\,.\]

\subsection{Hessenberg Schubert varieties}

For each $w\in \Sn{k}$, let $\Omega_{w,h}$ denote the Zariski closure of the minus cell $\Omega_{w,h}^\circ$. We call $\Omega_{w,h}$ the \emph{Hessenberg Schubert variety} indexed by $w$. Since the Hessenberg variety $\Hess(S,h)$ admits a decomposition into minus cells, one obtains a natural basis of $H_T^*(\Hess(S,h))$ from the Hessenberg Schubert varieties, as introduced in~\cite{CHL1}. Each Hessenberg Schubert variety $\Omega_{w,h}$ determines a class $[\Omega_{w,h}]$ in the equivariant Chow ring $A_T^*(\Hess(S,h))$, graded by codimension. Moreover, the equivariant cycle map

\[
cl_{\Hess(S,h)}^T \colon A_T^*(\Hess(S,h))_{\mathbb{Q}} \stackrel{\cong}{\longrightarrow} H_T^{2*}(\Hess(S,h);\mathbb{Q})
\]
is an isomorphism. After extending scalars from $\mathbb{Q}$ to $\mathbb{C}$, we regard this as an isomorphism
\[
cl_{\Hess(S,h)}^T \colon
A_T^*(\Hess(S,h))_{\mathbb{C}}
\stackrel{\cong}{\longrightarrow}
H_T^{2*}(\Hess(S,h)).
\]

\begin{definition}[{\cite[Definition~2.9]{CHL1}}]\label{def:classes}
	Let $\Hess(S, h)$ be a regular semisimple Hessenberg variety.
	For $w \in \mathfrak{S}_k$, we define $\sigma_{w,h}^T \in H^{2 \ell_h(w)}_{T}(\Hess(S,h))$ to be the image of the class 
    \[
        [\Omega_{w,h}] \in A^{\ell_h(w)}_T(\Hess(S, h))_{\mathbb{C}}
    \] 
    under the complexified equivariant cycle map $cl_{\Hess(S,h)}^T$.
    We denote by $\sigma_{w, h}\in H^{2 \ell_h(w)}(\Hess(S, h))$ the image of 
    $\sigma_{w,h}^T$
   in ordinary cohomology.
\end{definition}

\begin{proposition}[\cite{CHL1}]\label{prop:basis}
Let $\Hess(S, h)$ be a regular semisimple Hessenberg variety. Then $\{ \swh{w}^T \,|\, w  \in \mathfrak{S}_k, \ell_h(w)=d\} $ is a basis of $H^{2d}_T(\Hess(S, h))$ and $\{ \swh{w}\,|\, w  \in \mathfrak{S}_k, \ell_h(w)=d\} $ is a basis  of $H^{2d}(\Hess(S, h))$.
\end{proposition}

Following~\cite{CHL1}, we introduce a directed graph $G_{w,h}$ associated with $w$ and $h$, which was used to explicitly describe the minus cell $\Omega_{w,h}^\circ$.

\begin{definition}
    Let $h$ be a Hessenberg function, and $w\in \Sn{k}$. We define a directed graph $G_{w,h}$ with the vertex set $[k]$ such that for each pair of indices $1\leq j<i\leq k$, there is an edge $j\to i$ in $G_{w,h}$ if and only if $j<i\leq h(j)$ and $w(j)<w(i)$.
\end{definition}

\begin{remark} If we disregard the direction of each (directed) edge in  $G_{w,h}$, then it is a subgraph of $G_h$. In particular, $G_{e,h}$ becomes $G_{h}$ when we replace each arrow with an edge in $G_{e,h}$.    
\end{remark}

\begin{example}
    Let $h=(3,4,5,5,6,7,7)$. When $w=e$, the set of directed edges of $G_{e,h}$ is:
    $$E(G_{e,h})=\{(1,2),(1,3),(2,3),(2,4),(3,4),(3,5),(4,5),(5,6),(6,7)\},$$ and hence for $w=4\ 6\ 5\ 7\ 1\ 2\ 3$, the set of directed edges of $G_{w,h}$ is:
    $$E(G_{w,h})=\{(1,2),(1,3),(2,4),(3,4),(5,6),(6,7)\}.$$
    See Figure~\ref{fig:Geh-Gwh}.
\end{example}
\begin{figure}[htbp]
\centering

\begin{subfigure}{0.48\textwidth}
\centering
\begin{tikzpicture}[>=stealth, thick, scale=0.8]
  \node[circle, draw, inner sep=1.5pt] (1) at (0,0) {$1$};
  \node[circle, draw, inner sep=1.5pt] (2) at (1.3,0) {$2$};
  \node[circle, draw, inner sep=1.5pt] (3) at (2.6,0) {$3$};
  \node[circle, draw, inner sep=1.5pt] (4) at (3.9,0) {$4$};
  \node[circle, draw, inner sep=1.5pt] (5) at (5.2,0) {$5$};
  \node[circle, draw, inner sep=1.5pt] (6) at (6.5,0) {$6$};
  \node[circle, draw, inner sep=1.5pt] (7) at (7.8,0) {$7$};

  \draw[->] (1) to (2);
  \draw[->] (1) to[bend left=35] (3);
  \draw[->] (2) to (3);
  \draw[->] (2) to[bend left=35] (4);
  \draw[->] (3) to (4);
  \draw[->] (3) to[bend left=35] (5);
  \draw[->] (4) to (5);
  \draw[->] (5) to (6);
  \draw[->] (6) to (7);
\end{tikzpicture}
\caption{$G_{e,h}$}
\end{subfigure}
\hfill
\begin{subfigure}{0.48\textwidth}
\centering
\begin{tikzpicture}[>=stealth, thick, scale=0.8]
  \node[circle, draw, inner sep=1.5pt] (1) at (0,0) {$1$};
  \node[circle, draw, inner sep=1.5pt] (2) at (1.3,0) {$2$};
  \node[circle, draw, inner sep=1.5pt] (3) at (2.6,0) {$3$};
  \node[circle, draw, inner sep=1.5pt] (4) at (3.9,0) {$4$};
  \node[circle, draw, inner sep=1.5pt] (5) at (5.2,0) {$5$};
  \node[circle, draw, inner sep=1.5pt] (6) at (6.5,0) {$6$};
  \node[circle, draw, inner sep=1.5pt] (7) at (7.8,0) {$7$};

  \draw[->] (1) to (2);
  \draw[->] (1) to[bend left=35] (3);
  \draw[->] (2) to[bend left=35] (4);
  \draw[->] (3) to (4);
  \draw[->] (5) to (6);
  \draw[->] (6) to (7);
\end{tikzpicture}
\caption{$G_{w,h}$ for $w=4657123$}
\end{subfigure}

\caption{Directed graphs associated with $h=(3,4,5,5,6,7,7)$.}
\label{fig:Geh-Gwh}
\end{figure}
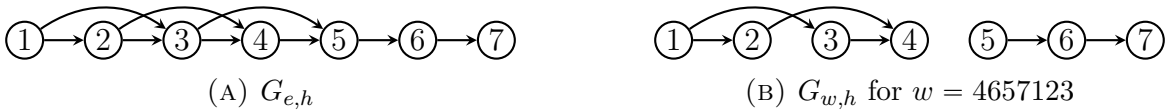

A vertex $i$ is \emph{reachable} from $j$ in $G_{w,h}$ if there exists a sequence of vertices
$v_0, v_1,\dots, v_k$ such that $j= v_0 \to v_1 \to \dots \to v_{k-1}\to v_k = i$. 
The following corollary is a streamlined version of~\cite[Corollary 3.9]{CHL1}; 
for reachable pairs $(j,i)$, the original statement gives an explicit recursive formula for $x_{i,j}$.

\begin{corollary}\label{cor:reachable_entry}
Let $x=(x_{i,j})$ be a lower triangular matrix with diagonal entries equal to~$1$. Then for each $w\in\Sn{k}$, $wx\in \Omega_{w,h}^\circ$ if and only if
\begin{enumerate}
\item $x_{i,j}=0$ whenever $i$ is not reachable from $j$ in $G_{w,h}$;
\item for each reachable pair $(j,i)$, the entry $x_{i,j}$ is determined recursively.
\end{enumerate}
\end{corollary}

In~\cite{CHL1}, the support of $\sigma_{w,h}$ was described in terms of the reachability condition in the directed graph $G_{w,h}$. We now introduce the notion of an $h$-admissible permutation, which allows one to express that description more simply.\footnote{Although the term ``$h$-admissible'' does not appear in~\cite{CHL1}, the notation $\mathcal{G}_h$ was already used there for the set of $h$-admissible permutations.}

\begin{definition}
Let \( h \) be a Hessenberg function on \( [k] \).
\begin{enumerate}
\item A permutation $w\in \Sn{k}$ is called \emph{$h$-admissible} if
\[
w^{-1}(w(j)+1)\le h(j)\qquad\text{for all }j\text{ with } w(j)\in [k-1].
\]
We denote by $\Gh$ the set of all $h$-admissible permutations. We decompose $\Gh$ according to the $h$-length:
\[
\Gh=\bigsqcup_d \Gh^d,
\qquad
\Gh^d\coloneqq \{\,w\in \Gh\mid \ell_h(w)=d\,\}.
\]
\item For an $h$-admissible permutation $w\in \Gh$, we define 
$P(w)$ to be the set of permutations whose relative order agrees with that of $w$ along every edge of $G_h$; that is,
\[ P(w)\coloneqq \{v\in \Sn{k}\,|\, v_i<v_j \iff w_i<w_j \text{ for all } i<j\leq h(i) \} \,. \]
\end{enumerate}
\end{definition}

\begin{proposition}\cite{CHL2}\label{prop:support}
Let $\Hess(S, h)$ be a regular semisimple Hessenberg variety. 
\begin{enumerate}
\item For each $v\in\Sn{k}$, there is a unique permutation $\widetilde{v}\coloneqq w\in  \Gh$ such that $v\in P(w)$; therefore, $\Sn{k}$ is partitioned into $P(w)$'s for $w\in \Gh$:
 $$\displaystyle \Sn{k}= \bigsqcup_{w \in \Gh} P(w).$$
\item For $w\in \Sn{k}$,  $\supp(\swh{w}^T)=(\Omega_{w,h})^T=\phi[\widetilde{w}, w_0]$, where $\phi=w\widetilde{w}^{-1}$. In particular, if  $w\in \Gh$, then $\supp(\swh{w}^T)=[w, w_0]$.
\end{enumerate}
\end{proposition}

\begin{remark}\label{rmk:generator_flag}
    If $h=(k, k, \dots, k)$, so that $\Hess(S, h)=\flag{k}$, then every permutation in $\Sn{k}$ is $h$-admissible. Hence,  
    for each $0\leq d\leq d_h= {\binom{k}{2}}$, we have  $$\Gh^d=\{u\in \Sn{k} \mid \ell_h(u)=\ell(u)=d\}.\,$$ 
\end{remark}

For given $h$ and $w\in \Sn{k}$, let $\Gamma_{w, h}=(V_w, E_w, \alpha_w)$ be the induced subgraph of $\Gamma_h$ by the set of vertices $ \supp(\swh{w}^T)=(\Omega_{w,h})^T$, which was considered in \cite{CHP}. Then, due to Proposition~\ref{prop:support}, $\Gamma_{w,h}$ and $\Gamma_{\widetilde{w},h}$ are isomorphic as edge-labeled graphs, where the isomorphism is given by the map $\phi$.

\begin{proposition}\cite[Proposition 2.11]{CHL1}\label{prop:computing the value}
Let $h$ be a Hessenberg function and $w$ be a permutation. 
\begin{enumerate}
\item If $v\in \supp (\sigma_{w,h}^T)$, then $w\leq_h v$. Moreover, if  $w\in\Gh$ and $w\leq v$, 
then $w\leq_h v$.
\item $\sigma_{w, h}^T(w)=\prod_{\{w, u\}\in E-E_w} \alpha(\{w, u\})$
\item  $\sigma_{w, h}^T(v)$ is homogeneous of degree $\ell_h(w)$ for each $v\in \supp(\swh{w}^T)$.
\item For $v\in \supp(\swh{w}^T)$, if the degree of $v$ in $\Gamma_{w,h}$ is equal to the degree of $w$  in $\Gamma_{w,h}$, that is $d_h-\ell_h(w)$, then  $$\sigma_{w, h}^T(v)=\prod_{\{v, u\}\in E-E_w} \alpha(\{v, u\})\,.$$
\end{enumerate}
\end{proposition}

 \subsection{Symmetric group action}\label{subsec:dot action} 
For a Hessenberg function $h$ on $[k]$, the $T$-equivariant cohomology space $H_T^{2d}(\Hess(S, h))$ 
becomes a $\mathbb{C}\Sn{k}$-module by the \emph{dot action} of Tymoczko defined as follows.
For $\sigma=(\sigma(v))_{v\in \Sn{k}} \in H_T^*(\Hess(S, h))$ and $u\in \Sn{k}$, 
\[(u\cdot \sigma)(v)(t_1,\dots,t_k)\coloneqq \sigma(u^{-1}v)(t_{u(1)},\dots,t_{u(k)}).\]

\begin{definition} Let $h$ be a Hessenberg function on $[k]$. 
\begin{enumerate}
\item For $u, v\in \Sn{k}$ with $u<v$ in the Bruhat order, we write 
\[v\rightarrow u\quad\text{ if }u<_h v,\qquad v \dashrightarrow u\quad\text{ if }u\not<_h v.\]
\item For $w\in \Sn{k}$ and $i\in[k-1]$ satisfying $w \rightarrow s_iw$, we define a set $\mathcal A_{s_i, w}$ as:
\[ \mathcal A_{s_i, w} \coloneqq \{ u \in \Omega_{s_{i}w,h}^T \cap \Omega_w^T \,|\, \dim_{\C}(\Omega_u^{\circ} \cap \Omega_{s_{i}w,h}) =\dim_{\C} \Omega_{w,h}, u \dashrightarrow s_iu \}\,. \]
\item For $w\in \Sn{k}$ and $i\in[k-1]$ satisfying $w \rightarrow s_iw$, if $u \in \mathcal A_{s_i, w}$, then we define two closed subvarieties $\mathcal T_u$ and $\mathcal T_{s_iu}$ as follows:
 \[\mathcal T_u\coloneqq \overline{\Omega_u^{\circ} \cap \Omega_{s_{i}w,h}} \quad \text{ and } \quad  \mathcal T_{s_iu}=\overline{\Omega_{s_iu}^{\circ} \cap \Omega_{s_{i}w,h}}\,.\]
We let  $\tau_u$ and $\tau_{s_iu}$ be the classes in $H^*_T(\Hess(S,h))$ induced by $\mathcal T_u$ and $\mathcal T_{s_iu}$, respectively.
\end{enumerate}
\end{definition}

Note that $w\not\in \mathcal{A}_{s_i,w}$ by definition. Furthermore, since $\mathcal{A}_{s_i,w}\subset [w,w_0]$  by Proposition~\ref{prop:computing the value}(1), if $u\in\mathcal{A}_{s_i,w}$ then $u>w$ in the Bruhat order.
The dot actions on the classes $\sigma_{w,h}\in H^*_T(\Hess(S,h))$ were calculated in \cite{CHL1}.

\begin{proposition}[{\cite[Theorem~B]{CHL1}}]\label{prop:simple action}
	Let $w$ be a permutation in $\mathfrak{S}_k$ and let $s_i$ be a simple reflection.
\begin{enumerate}
	\item If $ w  \dashrightarrow s_iw$ or $s_iw \dasharrow w$, then $s_i\cdot\swh{w}^T =\swh{s_i w}^T$.
	\item If $s_i w \rightarrow w$, then $s_i\cdot \swh{w}^T = \swh{w}^T$.
	\item If $w \rightarrow s_iw$, then
	\[
	\left(s_i \cdot\swh{w}^T + \sum_{u \in \mathcal A_{s_i, w} } \tau_{s_i u} \right) = \left(\swh{w}^T + \sum_{u   \in \mathcal A_{s_i, w}} \tau_u\right) + (t_{i+1}-t_i)\swh{s_i w}^T \,.
	\]
\end{enumerate}
\end{proposition}

Note that if $h=(k, k, \dots, k)$, then either $s_i w \rightarrow w$ or $w \rightarrow s_iw$ must hold for all $s_i$ and $w\in \Sn{k}$. Moreover $\mathcal{A}_{s_i,w}=\varnothing$ whenever  $w \rightarrow s_iw$. Therefore, we obtain the following well-known statement from the above proposition.
\begin{corollary}\label{cor:dot action on flag}
If $h=(k, k, \dots, k)$ so that  $\Hess(S, h)=\flag{k}$, then $s_i\cdot \sigma_w=\sigma_w$ for all $s_i$ and $w$ in $\Sn{k}$.    
\end{corollary}

\section{Block decompositions and the dot action}\label{sec:dot action}

In this section, we establish a factorization formula for equivariant Hessenberg Schubert classes associated with an $i$-block decomposition. This factorization will be used to study the dot action.

We first introduce two truncation operations on Hessenberg functions, obtained by restricting $h$ to the first $i$ indices and to the remaining $k-i$ indices.

\begin{definition}
Let \( h=(h(1), h(2), \dots, h(k)) \) be a Hessenberg function on \( [k] \). For an integer $0\leq i\leq k$, we define Hessenberg functions $^{\da i}h$ on $[i]$ and $h^{\da i}$ on $[k-i]$ as follows:
\[^{\da i}h(j)=\min(h(j), i) \quad \text{ for \,\,} j=1, \dots, i \,, \]
and
\[h^{\da i}(j)=h(i+j)-i \quad \text{ for \,\, } j=1, \dots, k-i.  \]
We use the convention that the Hessenberg function on \([0]\) is empty;
thus \({}^{\da 0}h\) and \(h^{\da k}\) are empty, while
\(h^{\da 0}=h\) and \({}^{\da k}h=h\).
\end{definition}

By construction, ${}^{\da i}h$ and $h^{\da i}$ are again Hessenberg functions, encoding the induced Hessenberg conditions in the two blocks.

\begin{example}\label{ex:h-i} Consider a Hessenberg function $h=(3, 3, 4, 5, 6, 7, 7)$ on $[k]=[7]$.
\begin{enumerate}
\item $^{\da 2}h=(2,2)$ and $h^{\da 2}=(2,3,4,5,5)$.
\item $^{\da 4}h=(3, 3, 4, 4)$ and $h^{\da 4}=(2, 3, 3)$.
\end{enumerate}
\end{example}
These truncations are most transparent at the level of diagrams, where they appear as the northwest and southeast blocks separated by the cut at $i$, see Figure~\ref{fig:hessenberg_block}.

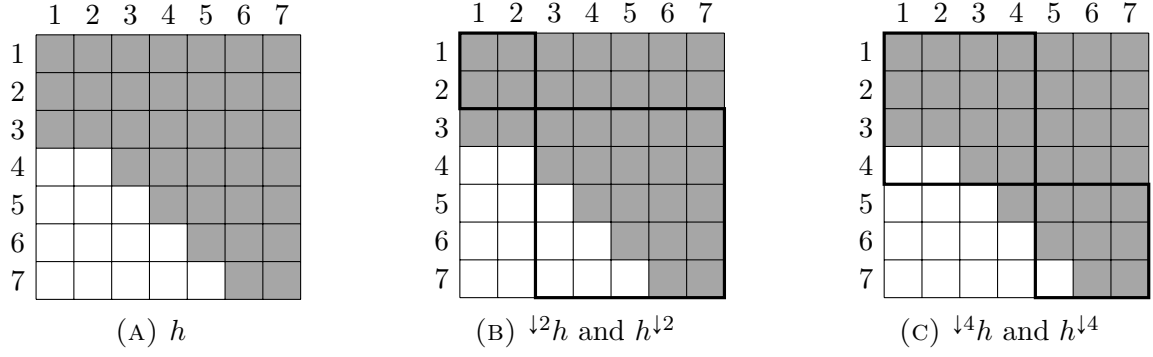
\begin{figure}[htbp]
\centering

\begin{subfigure}{0.32\textwidth}
  \centering
  \begin{tikzpicture}[scale=0.5]
  \def\n{7}
  \def\h{3,3,4,5,6,7,7}

  \foreach \val [count=\j from 1] in \h {%
    \foreach \i in {1,...,7} {%
      \ifnum\i>\val\relax\else
        \fill[black!35] ({\j-1},{\n-\i}) rectangle ({\j},{\n-\i+1});
      \fi
    }%
  }

  \draw[step=1, black] (0,0) grid (\n,\n);

  \foreach \k in {1,...,\n} {
    \node[above] at (\k-0.5,\n) {\small $\k$};
    \node[left]  at (0,\n-\k+0.5) {\small $\k$};
  }
  \end{tikzpicture}
  \caption{$h$}
\end{subfigure}
\hfill
\begin{subfigure}{0.32\textwidth}
  \centering
  \begin{tikzpicture}[scale=0.5]
  \def\n{7}
  \def\h{3,3,4,5,6,7,7}

  \foreach \val [count=\j from 1] in \h {%
    \foreach \i in {1,...,7} {%
      \ifnum\i>\val\relax\else
        \fill[black!35] ({\j-1},{\n-\i}) rectangle ({\j},{\n-\i+1});
      \fi
    }%
  }

  \draw[step=1, black] (0,0) grid (\n,\n);

  \foreach \k in {1,...,\n} {
    \node[above] at (\k-0.5,\n) {\small $\k$};
    \node[left]  at (0,\n-\k+0.5) {\small $\k$};
  }

  \draw[very thick, black] (0,5) rectangle (2,7);

  \draw[very thick, black] (2,0) rectangle (7,5);
  \end{tikzpicture}
  \caption{${}^{\da 2}h$ and $h^{\da 2}$}
\end{subfigure}
\hfill
\begin{subfigure}{0.32\textwidth}
  \centering
  \begin{tikzpicture}[scale=0.5]
  \def\n{7}
  \def\h{3,3,4,5,6,7,7}

  \foreach \val [count=\j from 1] in \h {%
    \foreach \i in {1,...,7} {%
      \ifnum\i>\val\relax\else
        \fill[black!35] ({\j-1},{\n-\i}) rectangle ({\j},{\n-\i+1});
      \fi
    }%
  }

  \draw[step=1, black] (0,0) grid (\n,\n);

  \foreach \k in {1,...,\n} {
    \node[above] at (\k-0.5,\n) {\small $\k$};
    \node[left]  at (0,\n-\k+0.5) {\small $\k$};
  }

  \draw[very thick, black] (0,3) rectangle (4,7);

  \draw[very thick, black] (4,0) rectangle (7,3);
  \end{tikzpicture}
  \caption{${}^{\da 4}h$ and $h^{\da 4}$}
\end{subfigure}

\caption{Hessenberg diagrams associated with block decompositions.}
\label{fig:hessenberg_block}
\end{figure}

Let $B_r$ denote the Borel subgroup of upper triangular matrices in $\GL_r(\C)$. 
Consider the block-anti-diagonal map
\begin{equation*}
\GL_i(\C)\times \GL_{k-i}(\C)\longrightarrow \GL_k(\C),
\qquad
(x_1,x_2)\longmapsto
\begin{pmatrix}
O & x_2\\
x_1 & O
\end{pmatrix}.
\end{equation*}
This map induces a well-defined embedding
\begin{equation}\label{eq:block_anti_diagonal}
\GL_i(\C)/B_i\times \GL_{k-i}(\C)/B_{k-i}
\hookrightarrow
\GL_k(\C)/B_k.
\end{equation}
Indeed, if $b_1\in B_i$ and $b_2\in B_{k-i}$, then
\[
\begin{pmatrix}
O & x_2b_2\\
x_1b_1 & O
\end{pmatrix}
=
\begin{pmatrix}
O & x_2\\
x_1 & O
\end{pmatrix}
\begin{pmatrix}
b_1 & O\\
O & b_2
\end{pmatrix},\quad\text{and}\quad \begin{pmatrix}
b_1 & O\\
O & b_2
\end{pmatrix}\in B_k.
\]
Hence, the image in $\GL_k(\C)/B_k$ is independent of the choice of representatives.
Moreover, if
\[
\begin{pmatrix}
O & x_2\\
x_1 & O
\end{pmatrix}B_k
=
\begin{pmatrix}
O & y_2\\
y_1 & O
\end{pmatrix}B_k,
\]
then
\[
\begin{pmatrix}
x_1^{-1}y_1 & O\\
O & x_2^{-1}y_2
\end{pmatrix}
\in B_k.
\]
Thus $x_1^{-1}y_1\in B_i$ and $x_2^{-1}y_2\in B_{k-i}$, so $x_1B_i=y_1B_i$ and $x_2B_{k-i}=y_2B_{k-i}$. Therefore, the induced map is injective.

For a regular semisimple element $S=\diag(\lambda_1,\dots,\lambda_k)$, we define
\[
S^{\da i}\coloneqq \diag(\lambda_1,\dots,\lambda_{k-i}),
\qquad
{}^{\da i}S\coloneqq \diag(\lambda_{k-i+1},\dots,\lambda_k)
\]
for $1\le i<k$; similarly, we let $T^{\da i}$ and ${}^{\da i}T$ be the subtori of $T$ determined by the first $k-i$ coordinates and the last $i$ coordinates, respectively. That is,
\[S=\begin{pmatrix}S^{\da i}&O\\O&{}^{\da i}S\end{pmatrix},\qquad T=\begin{pmatrix}
 T^{\da i}&O\\O&{}^{\da i}T   
\end{pmatrix}.\]
Then for $t=(t_1,\dots,t_k)\in T$, we write
\[{}^{\da i}t\coloneqq(t_{k-i+1},\dots,t_k),\qquad t^{\da i}\coloneqq (t_1,\dots,t_{k-i}).\]
We define the action of $T$ on  \(\Hess({}^{\da i}S,{}^{\da i}h)\times\Hess(S^{\da i},h^{\da i})\) by $$t\cdot(x_1,x_2)=({}^{\da i}t\cdot x_1,t^{\da i}\cdot x_2).$$ 

\begin{lemma}\label{lem:block_hess_embedding}
Let $h$ be a Hessenberg function on $[k]$ and let $1\le i<k$.
Then the restriction of the embedding in~\eqref{eq:block_anti_diagonal} to $\Hess({}^{\da i}S,{}^{\da i}h)\times\Hess(S^{\da i},h^{\da i})$ induces a natural $T$-equivariant closed embedding
\[
\iota\colon\Hess({}^{\da i}S,{}^{\da i}h)\times\Hess(S^{\da i},h^{\da i})
\hookrightarrow
\Hess(S,h).
\]
\end{lemma}
\begin{proof}
With our convention, the first factor corresponds to the bottom-left block in the block-anti-diagonal embedding, while the second factor corresponds to the top-right block. 
Note that for $(x_1,x_2)\in \GL_i(\C)\times \GL_{k-i}(\C)$, we have\[\begin{pmatrix}O&x_{2}\\x_{1}&O\end{pmatrix}^{-1}\begin{pmatrix}S^{\da i}&O\\O&{}^{\da i}S\end{pmatrix}\begin{pmatrix}O&x_{2}\\x_{1}&O\end{pmatrix}=\begin{pmatrix}x_1^{-1}{}^{\da i}S\ x_1&O\\O&x_2^{-1}S^{\da i}x_2\end{pmatrix}.\]
By the definitions of ${}^{\da i}h$ and $h^{\da i}$, the matrix on the right belongs to $H_h$ if and only if $x_1^{-1} {}^{\da i}S\ x_1\in H_{{}^{\da i}h}$ and $x_2^{-1} S^{\da i}x_2\in H_{h^{\da i}}$.
Hence the block-anti-diagonal embedding restricts to a well-defined closed embedding
\[
\iota\colon
\Hess({}^{\da i}S,{}^{\da i}h)\times \Hess(S^{\da i},h^{\da i})
\hookrightarrow
\Hess(S,h).
\]

From the definition of the action of $T$ on \(\Hess({}^{\da i}S,{}^{\da i}h)\times\Hess(S^{\da i},h^{\da i})\), we get
\[\iota(t\cdot(x_1,x_2))=\iota({}^{\da i}t\cdot x_1,t^{\da i}\cdot x_2)=\begin{pmatrix}O&t^{\da i}\cdot x_2\\{}^{\da i}t\cdot x_1&O\end{pmatrix}=t\cdot\begin{pmatrix}O& x_2\\x_1&O\end{pmatrix}=t\cdot\iota(x_1,x_2),\]
so $\iota$ is $T$-equivariant.
\end{proof}

Via the block-anti-diagonal embedding $\iota$, we can identify $$\left(\Hess({}^{\da i}S,{}^{\da i}h)\times\Hess(S^{\da i},h^{\da i})\right)^T$$ with the set of permutations $w\in \Sn{k}$ satisfying
\[
\{w_1,\ldots,w_i\}=\{k-i+1,\ldots,k\},
\qquad
\{w_{i+1},\ldots,w_k\}=\{1,\ldots,k-i\}.
\]
This observation motivates the following definition.

\begin{definition}
Let $w\in \Sn{k}$.
\begin{enumerate}
\item 
Let $1\le i\le k-1$.
We say that $w$ admits an \emph{$i$-block decomposition} if
\[
\{w_1,\ldots,w_i\}=\{k-i+1,\ldots,k\}
\quad\text{and}\quad
\{w_{i+1},\ldots,w_k\}=\{1,\ldots,k-i\}.
\]
We denote by
\(
\Sn{k}^{[i]}
\)
the set of permutations admitting an $i$-block decomposition.
\item If $w$ admits an $i$-block decomposition, we define the
\emph{left block permutation} and \emph{right block permutation} by
\begin{align*}
^{\da i}w 
  &\coloneqq (w_{1}-(k-i))\,\, \cdots\,\, (w_{i}-(k-i))\in \Sn{i}, \\[0.3em]
w^{\da i} 
  &\coloneqq w_{i+1}\, \cdots \, w_{k}\in\Sn{k-i}.
\end{align*}
\end{enumerate}
\end{definition}

Note that $\Sn{k}^{[i]}=\iota(\Sn{i}\times \Sn{k-i})$, and 
for each $(x,y)\in \Sn{i}\times\Sn{k-i}$, there exists a unique permutation $w\in\Sn{k}$ such that ${}^{\da i}w=x$ and $w^{\da i}=y$.

\begin{remark}\label{rem:block_upward_closed}
It is immediate from the definition that if $w\in \Sn{k}$ admits an
$i$-block decomposition, then every permutation $v\ge w$ in the Bruhat order
also admits an $i$-block decomposition.
\end{remark}

\begin{lemma}
    Let $h$ be a Hessenberg function on $[k]$. If $w\in\Sn{k}$ admits an $i$-block decomposition for some $1\leq i<k$, then the directed graph $G_{w,h}$ decomposes as a disjoint union
\begin{equation*}
G_{w,h}=G_{{}^{\da i}w,{}^{\da i}h}\bigsqcup G_{w^{\da i},h^{\da i}}.
\end{equation*}
\end{lemma}
\begin{proof}
    Recall that for $1\le j_1<j_2\le k$, there is a directed edge $j_1\to j_2$ in $G_{w,h}$ if and only if
\[
j_1<j_2\le h(j_1)
\quad\text{and}\quad
w(j_1)<w(j_2).
\]
Since $w$ admits an $i$-block decomposition, if $j_1\le i<j_2$, then $w(j_1)>w(j_2)$, and hence no edge
$j_1\to j_2$ occurs in $G_{w,h}$, which implies that no vertex $j_2>i$ is reachable from any vertex $j_1\le i$ in $G_{w,h}$.

If $1\le j_1<j_2\le i$, then by definition of ${}^{\da i}w$ and ${}^{\da i}h$,
\[
j_1<j_2\le h(j_1)\ \text{and}\ w(j_1)<w(j_2)
\quad\Longleftrightarrow\quad
j_1<j_2\le {}^{\da i}h(j_1)\ \text{and}\ {}^{\da i}w(j_1)<{}^{\da i}w(j_2).
\]
Thus, $j_1\to j_2$ is an edge of $G_{w,h}$ if and only if it is an edge of
$G_{{}^{\da i}w,{}^{\da i}h}$.
Similarly, if $i<j_1<j_2\le k$, then $j_1\to j_2$ is an edge of $G_{w,h}$ if and only if it is an edge of
$G_{w^{\da i},h^{\da i}}$.
This proves the graph decomposition.
\end{proof}
\begin{lemma}\label{lem:block_generator}
    Let $h$ be a Hessenberg function on $[k]$. For given $0\leq i\leq k$, an $h$-admissible permutation $w\in\mathcal{G}_h$ admits an $i$-block decomposition if and only if there exists $u\in\mathcal{G}_{{}^{\da i}h}$ and $v\in\mathcal{G}_{h^{\da i}}$ such that \(^{\da i}w=u\) and \(w^{\da i}=v\).
\end{lemma}
\begin{proof}
    It suffices to prove for $1\leq i<k$. Suppose that $w\in\Gh$ admits an $i$-block decomposition. For $1\leq j\leq i$, the $i$-block decomposition implies $k-i<w(j)\leq k$.
If $w(j)<k$, then the $h$-admissibility of $w$ gives
\[
w^{-1}(w(j)+1)\leq h(j),
\]
and the block condition forces $w^{-1}(w(j)+1)\leq i$.
Hence
\[
w^{-1}(w(j)+1)\leq \min\{h(j),i\}={}^{\da i}h(j),
\]
so $^{\da i}w$ is $^{\da i}h$-admissible.
For $i+1\leq j\leq k$, we have $w(j)\leq k-i$. When $w^{\da i}(j)<k-i$, again by $h$-admissibility, we obtain
\[
(w^{\da i})^{-1}(w^{\da i}(j)+1)=w^{-1}(w(j)+1)-i\leq h(j)-i=h^{\da i}(j),
\]
which shows that $w^{\da i}$ is $h^{\da i}$-admissible. 

Now we prove the other direction. Suppose that $u\in\mathcal{G}_{{}^{\da i}h}$ and $v\in\mathcal{G}_{h^{\da i}}$. Then there exists a permutation $w\in\Sn{k}$ such that ${}^{\da i}w=u$ and $w^{\da i}=v$. We claim that $w\in \Gh$. If $1\leq w_j<k-i$, then
\[
w^{-1}(w_j+1)
=
v^{-1}(v_j+1)+i
\leq {h^{\da i}}(j)+i
=
h(j).
\]
If $w_j=k-i$, then $j>i$ and $h(j)>i$. So we have 
\[
w^{-1}(k-i+1)=u^{-1}(1)\leq {i}<h(j)
\]
Finally, if $k-i+1\leq w_j<k$, then
\[
w^{-1}(w_j+1)
=
w^{-1}(u_j+k-i+1)
=
u^{-1}(u_j+1)
\leq {}^{\da i}h(j)
\leq h(j).
\]
Therefore, $w$ is an $h$-admissible permutation.
\end{proof}

We now turn to the relationship between Hessenberg Schubert varieties and block
decompositions.

\begin{proposition}\label{prop:block_prod}
	Let $h$ be a Hessenberg function on $[k]$ and let $\iota$ be as in Lemma~\ref{lem:block_hess_embedding}. Assume 
$w\in\Sn{k}$ admits an $i$-block decomposition for some $1\leq i<k$. Then $\iota$ restricts to a $T$-equivariant isomorphism \[\Omega_{{}^{\da i}w,{}^{\da i}h}\times \Omega_{w^{\da i},h^{\da i}}\xrightarrow{\;\sim\;} \Omega_{w,h}.\] 
\end{proposition}

\begin{proof}
Let $U_k$ denote the group of $k\times k$ lower triangular matrices with all diagonal entries equal to $1$.
Then, for a matrix $(x_{p,q})\in U_k$, $wx\in\Omega_{w,h}^\circ$ if and only if
\begin{enumerate}
\item $x_{p,q}=0$ whenever $p$ is not reachable from $q$ in $G_{w,h}$;
\item for each reachable pair $(q,p)$, the entry $x_{p,q}$ is determined recursively; an explicit recursive formula is given in~\cite[Corollary 3.9]{CHL1}.
\end{enumerate}
Hence, if $q\leq i  <p$, then the entry $x_{p,q}$ must vanish.
Since $x$ is lower triangular, this forces $x$ to be block diagonal of the form
\[
x=\begin{pmatrix}
x_1 & O\\
O & x_2
\end{pmatrix},
\qquad x_1\in U_i,\ x_2\in U_{k-i}.
\]
In particular, since $w$ admits an $i$-block decomposition, $wx$ is a block-anti-diagonal matrix of the form:
\[wx=\begin{pmatrix}O&w^{\da i}\\{}^{\da i}w&O\end{pmatrix}\begin{pmatrix}x_1&O\\O&x_2\end{pmatrix}=\begin{pmatrix}O&w^{\da i}x_2\\{}^{\da i}wx_1&O\end{pmatrix}.\]
Since the reachability conditions on the entries of $x_1$ and $x_2$ are precisely those determined by
the graphs $G_{{}^{\da i}w,{}^{\da i}h}$ and $G_{w^{\da i},h^{\da i}}$, respectively, we obtain
\[
\Omega_{w,h}^\circ
=
\iota(\Omega_{{}^{\da i}w,{}^{\da i}h}^\circ
\times
\Omega_{w^{\da i},h^{\da i}}^\circ).
\]
Since $\iota$ is the $T$-equivariant closed embedding coming from the block-anti-diagonal embedding~\eqref{eq:block_anti_diagonal}, we obtain
\[
\Omega_{w,h}
=\iota(
\Omega_{{}^{\da i}w,{}^{\da i}h}
\times
\Omega_{w^{\da i},h^{\da i}}).
\]
This completes the proof.
\end{proof}

The following is an immediate consequence of the proposition.
\begin{corollary}\label{cor:block_torus_action} 
If \(w\in\Sn{k}\) admits an \(i\)-block decomposition, then the following statements hold.
\begin{enumerate}
\item
The set of $T$-fixed points decomposes as
\[
\Omega_{w,h}^T=\Omega_{{}^{\da i}w,{}^{\da i}h}^{{}^{\da i}T}
\times
\Omega_{w^{\da i},h^{\da i}}^{T^{\da i}}
\] via $v=\iota({}^{\da i}v,v^{\da i})$.

\item
A one-dimensional $T$-orbit in $\Omega_{w,h}$ connects two fixed points
$u=\iota({}^{\da i}u,u^{\da i})$ and $v=\iota({}^{\da i}v,v^{\da i})$ if and only if it arises from a
one-dimensional orbit in exactly one factor, while the other factor is fixed.
\end{enumerate}
\end{corollary}

Let $N_\iota$ be the equivariant Euler class of the normal bundle of
the embedding $\iota$ in Lemma~\ref{lem:block_hess_embedding}.
Then at each $T$-fixed point, the $T$-equivariant Euler class $e^T(N_\iota)$ of $N_\iota$ is given as follows.

\begin{lemma}\label{lem:equiv_euler}
For $(x,y)\in \Sn{i}\times \Sn{k-i}$, let $v\in \Sn{k}$ be the unique
permutation satisfying ${}^{\da i}v=x$ and $v^{\da i}=y$.
Then we have
\[
\left.e^T(N_\iota)\right|_{(x,y)}
=
\prod_{1\le a\le i< b\le h(a)}
\bigl(t_{v(a)}-t_{v(b)}\bigr).
\]
\end{lemma}
\begin{proof}
For simplicity, let $X=\Hess({}^{\da i}S,{}^{\da i}h)$, $Y=\Hess(S^{\da i},h^{\da i})$, and $Z=\Hess(S,h)$. 
Recall that for a $T$-fixed point $v\in \Sn{k}$, the tangent space $T_vZ$ of $Z=\Hess(S,h)$ at $v$ is of the form
\[
T_vZ= 
\bigoplus_{1\le a<b\le h(a)} \C_{e_{v(a)}-e_{v(b)}}\quad\text{ as $T$-modules},
\]
where $\C_{\alpha}$ denotes the one-dimensional $T$-module determined by $\alpha$, see~\cite{DPS}. 

For $(x,y)\in\Sn{i}\times\Sn{k-i}$, let $v\in\Sn{k}^{[i]}$ be the unique permutation satisfying $x={}^{\da i}v$ and $y=v^{\da i}$. Then, by definition,
\[
x(a)=v(a)-(k-i)
\quad\text{for }1\le a\le i,
\qquad
y(c)=v(i+c)
\quad\text{for }1\le c\le k-i.
\]

Note that $T=T^{\da i}\times {}^{\da i}T$, where ${}^{\da i}T$ acts on $X$ and $T^{\da i}$ acts on $Y$. Hence for $1\leq r\leq i$, the character $e_{x(r)}$ of ${}^{\da i}T$ is identified with the character $e_{k-i+x(r)}$ of $T$. Similarly, for $1\leq s\leq k-i$, the character $e_{y(s)}$ of $T^{\da i}$ is identified with the character $e_{y(s)}$ of $T$. Therefore, the $T$-weights of $T_xX$ are identified as
\[T_{x}X\cong\bigoplus_{\substack{1\le a<b\le i\\ b\le {}^{\da i}h(a)}} \C_{e_{k-i+x(a)}-e_{k-i+x(b)}}= \bigoplus_{\substack{1\le a<b\le i\\ b\le h(a)}} \C_{e_{v(a)}-e_{v(b)}}.\] 
Here the condition $b\le {}^{\da i}h(a)$ is equivalent to $b\le h(a)$, since
\(
{}^{\da i}h(a)=\min(h(a),i)
\)
and $b\le i$.
Similarly, the $T$-weights of $T_yY$ are identified as
\[T_{y}Y=\bigoplus_{\substack{1\le c<d\le k-i\\ d\le h^{\da i}(c)}} \C_{e_{y(c)}-e_{y(d)}}\cong \bigoplus_{\substack{i+1\le a<b\le k\\ b\le h(a)}} \C_{e_{v(a)}-e_{v(b)}}\quad (\text{by taking } a=c+i\text{ and }b=d+i).
\]

For the embedding \(
\iota\colon X\times Y\hookrightarrow Z
\)
from Lemma~\ref{lem:block_hess_embedding}, the fixed point $(x,y)$ maps to $v$. Hence, we have an exact sequence of $T$-modules
\[
0
\longrightarrow
T_xX\oplus T_yY
\longrightarrow
T_vZ
\longrightarrow
(N_\iota)_{(x,y)}
\longrightarrow
0.
\]
Since $T$ is diagonalizable, this sequence splits as a sequence of $T$-modules. Comparing the above weight decompositions, the weights appearing in $T_xX\oplus T_yY$ are precisely those corresponding to pairs
\[
1\le a<b\le i
\qquad\text{or}\qquad
i+1\le a<b\le k
\]
with $b\le h(a)$. Thus the remaining weights in $T_vZ$ are exactly those corresponding to pairs
\[
1\le a\le i< b\le h(a).
\]
Therefore
\[
(N_\iota)_{(x,y)}
\cong
\bigoplus_{1\le a\le i< b\le h(a)}
\C_{e_{v(a)}-e_{v(b)}}.
\]
Thus the restriction of the $T$-equivariant Euler class to the fixed point $(x,y)$ is
\[
\left.e^T(N_\iota)\right|_{(x,y)}
=
\prod_{1\le a\le i< b\le h(a)}
\bigl(t_{v(a)}-t_{v(b)}\bigr).
\]
This proves the lemma.
\end{proof}

We are now ready to prove the blockwise description of the dot action of \(s_j\), $j\neq k-i$. We use the index-shift homomorphism
\[
{\operatorname{sh}}_{k-i}:\C[t_1,\dots,t_i]\longrightarrow \C[t_1,\dots,t_k]
\]
defined by $\operatorname{sh}_{k-i}(t_j)=t_{k-i+j}$.

\begin{theorem}\label{thm:separation}
Let $h$ be a Hessenberg function on $[k]$, and let $w\in \Sn{k}$ admit an
$i$-block decomposition.
Then, after $T$-equivariant localization, the dot action of $s_j$, $j\neq k-i$, on
$\sigma^T_{w,h}$ is determined blockwise. More precisely, for $j\neq k-i$ and $v\in\Sn{k}$, we have
\[
(s_j\cdot \sigma_{w,h}^T)(v)=0
\qquad
\text{if } v\notin \Sn{k}^{[i]}.
\]
If $v\in \Sn{k}^{[i]}$, then
\[
(s_j\cdot \sigma_{w,h}^T)(v)
=
\begin{cases}
\displaystyle
{\operatorname{sh}}_{k-i}\bigl((s_{j-(k-i)}\cdot \sigma_{{}^{\da i}w,{}^{\da i}h}^T)({}^{\da i}v)\bigr)\,
\sigma_{w^{\da i},h^{\da i}}^T(v^{\da i})\,
\left.e^T(N_\iota)\right|_{({}^{\da i}v,v^{\da i})},
& \text{if } j>k-i,\\[1.2em]
\displaystyle
{\operatorname{sh}}_{k-i}\bigl(\sigma_{{}^{\da i}w,{}^{\da i}h}^T({}^{\da i}v)\bigr)\,
\bigl((s_j\cdot \sigma_{w^{\da i},h^{\da i}}^T)(v^{\da i})\bigr)\,
\left.e^T(N_\iota)\right|_{({}^{\da i}v,v^{\da i})},
& \text{if } j<k-i.
\end{cases}
\]
\end{theorem}

\begin{proof}
Let $X=\Hess({}^{\da i}S,{}^{\da i}h)$,
$Y=\Hess(S^{\da i},h^{\da i})$, and
$Z=\Hess(S,h)$. Let 
$\pi_X\colon X\times Y\to X$ and $\pi_Y\colon X\times Y\to Y$ be the natural projections, and consider the embedding
\(
\iota\colon X\times Y \hookrightarrow Z
\) in Lemma~\ref{lem:block_hess_embedding}.
Under the identification of fixed points with permutations, we have $\iota((X\times Y)^T)=\Sn{k}^{[i]}$. Moreover,
\[
\begin{aligned}
\dim_\C Z-\dim_\C (X\times Y)&=d_h-(d_{{}^{\da i}h}+d_{h^{\da i}})\\
&=|\{(a,b)\mid 1\leq a\leq i <b \leq h(a)\}|\\
&=\ell_h(w)-(\ell_{{}^{\da i}h}({}^{\da i}w)+\ell_{h^{\da i}}(w^{\da i})).
\end{aligned}
\]

We use the identification $T=T^{\da i}\times {}^{\da i}T$ to regard the external products below as $T$-equivariant classes on $X\times Y$. Consider the following commutative diagram:
\[
\begin{tikzcd}[column sep=1.5em, row sep=large]
A_{{}^{\da i}T}^{\ell_{{}^{\da i}h}({}^{\da i}w)}(X)_\C \otimes_\C A_{T^{\da i}}^{\ell_{h^{\da i}}(w^{\da i})}(Y)_\C
  \arrow[r, "\boxtimes"]
  \arrow[d, "cl_X^{{}^{\da i}T}\otimes cl_Y^{T^{\da i}}"']
&
A_T^{\ell_{{}^{\da i}h}({}^{\da i}w)+\ell_{h^{\da i}}(w^{\da i})}(X\times Y)_\C
  \arrow[r, "\iota_*"]
  \arrow[d, "cl_{X\times Y}^{T}"']
&
A_T^{\ell_{h}(w)}(Z)_\C
  \arrow[d, "cl_Z^{T}"']
\\
H_{{}^{\da i}T}^{2{\ell_{{}^{\da i}h}({}^{\da i}w)}}(X)\otimes_\C H_{T^{\da i}}^{2{\ell_{h^{\da i}}(w^{\da i})}}(Y)
  \arrow[r, "\times"]
&
H_T^{2(\ell_{{}^{\da i}h}({}^{\da i}w)+\ell_{h^{\da i}}(w^{\da i}))}(X\times Y)
  \arrow[r, "\iota_*"]
&
H_T^{2\ell_{h}(w)}(Z),
\end{tikzcd}
\]
where the cohomology external product $\times$ (resp. Chow external product $\boxtimes$) is defined by
\[
\alpha\times\beta
\coloneqq  \pi_X^*(\alpha)\,\cup\,\pi_Y^*(\beta)
\quad
(\text{resp.\ }\alpha\boxtimes\beta\coloneqq \pi_X^*(\alpha)\cdot \pi_Y^*(\beta)).
\]
Here, $\iota_*$ denotes the proper pushforward induced by $\iota$. 

By Proposition~\ref{prop:block_prod}, the embedding $\iota$ restricts to a $T$-equivariant isomorphism \[\Omega_{{}^{\da i}w,{}^{\da i}h}\times \Omega_{w^{\da i},h^{\da i}}\xrightarrow{\;\sim\;} \Omega_{w,h}.\] Therefore, in equivariant
Chow groups, we have
\[
[\Omega_{w,h}]
=
\iota_*\Big(
\pi_X^*\big([\Omega_{{}^{\da i}w,{}^{\da i}h}]\big)
\cdot
\pi_Y^*\big([\Omega_{w^{\da i},h^{\da i}}]\big)
\Big)
\quad\text{in } A_T^*(Z)_\C.
\]

Applying the equivariant cycle class isomorphism
$cl_Z^T\colon A_T^*(Z)_\C\to H_T^{2*}(Z)$ and using its
compatibility with external products and proper pushforwards, we obtain
\[
\sigma_{w,h}^T
=
\iota_*\Big(
\pi_X^*\big(\sigma_{{}^{\da i}w,{}^{\da i}h}^T\big)
\;\cup\;
\pi_Y^*\big(\sigma_{w^{\da i},h^{\da i}}^T\big)
\Big)
\quad\text{in } H_T^{2*}(Z).
\]

By the self-intersection formula, we further have
\[
\iota^*(\sigma_{w,h}^T)
=
\pi_X^*\big(\sigma_{{}^{\da i}w,{}^{\da i}h}^T\big)
\;\cup\;
\pi_Y^*\big(\sigma_{w^{\da i},h^{\da i}}^T\big)
\;\cup\;
e^T(N_\iota) \quad\text{in }H_T^{2*}(X\times Y),
\]
see \cite[Appendix~A.6]{AF}.

Note that under the identification \(T\cong T^{\da i}\times {}^{\da i}T\), pulling back a class from \(X\) to \(X\times Y\) and regarding it as a \(T\)-equivariant class shifts the equivariant parameters by
\[
\operatorname{sh}_{k-i}\colon \C[t_1,\dots,t_i]\longrightarrow \C[t_1,\dots,t_k],
\qquad
t_j\longmapsto t_{j+k-i}.
\]
By contrast, pulling back a class from \(Y\) to \(X\times Y\) leaves the equivariant parameters unchanged, via the inclusion
\[
\C[t_1,\dots,t_{k-i}]\hookrightarrow \C[t_1,\dots,t_k],
\qquad
t_j\longmapsto t_j.
\]

After applying \(T\)-equivariant localization, the equality
\[
\sigma_{w,h}^T
=
\iota_*\Big(
\pi_X^*\big(\sigma_{{}^{\da i}w,{}^{\da i}h}^T\big)
\cup
\pi_Y^*\big(\sigma_{w^{\da i},h^{\da i}}^T\big)
\Big)
\]
can be evaluated at each fixed point \(v\in Z^T\). Since
\(
\iota((X\times Y)^T)=\Sn{k}^{[i]},
\)
the localization formula for the proper pushforward \(\iota_*\) gives
\[\sigma_{w,h}^T(v)=\begin{cases}
 {\operatorname{sh}}_{k-i}(\sigma_{^{\da i}w,{}^{\da i}h}^T({}^{\da i}v))\,\sigma_{w^{\da i},h^{\da i}}^T(v^{\da i})\,\left.e^T(N_\iota)\right|_{({}^{\da i}v,v^{\da i})},&\text{ if }v\in \Sn{k}^{[i]},\\
 0,&\text{ otherwise}.
\end{cases}
 \]

Recall that the dot action of \(\Sn{k}\) on \(H_T^*(\Hess(S,h))\) is given by
\[
(u\cdot \sigma)(v)
\coloneqq
\sigma(u^{-1}v)\bigl(t_{u(1)},\dots,t_{u(k)}\bigr),
\qquad u\in \Sn{k}.
\]
If \(j\neq k-i\), then \(s_j\) preserves the two blocks of values
\[
\{1,\dots,k-i\}
\quad\text{and}\quad
\{k-i+1,\dots,k\}.
\]
Hence \(s_j\) preserves the subset \(\Sn{k}^{[i]}\subset\Sn{k}\). Therefore, if \(v\notin\Sn{k}^{[i]}\), then \(s_jv\notin\Sn{k}^{[i]}\), and consequently
\[
(s_j\cdot \sigma_{w,h}^T)(v)=0.
\]

Now assume that \(v\in\Sn{k}^{[i]}\). The $T$-equivariant Euler class $e^T(N_\iota)$ is invariant under the dot action of \(s_j\) for \(j\neq k-i\), in the sense that
\[
\left.
(s_j\cdot e^T(N_\iota))
\right|_{({}^{\da i}v,v^{\da i})}
=
\left.e^T(N_\iota)\right|_{({}^{\da i}v,v^{\da i})}.
\]
Indeed, \(s_j\) only permutes weights within one of the two blocks, and hence preserves the product of the normal weights
\[
\prod_{1\le a\le i<b\le h(a)}
\bigl(t_{v(a)}-t_{v(b)}\bigr).
\]
Thus the above localized formula yields the following blockwise formulas. If \(j>k-i\), then \(s_j\) acts on the first factor as \(s_{j-(k-i)}\), and we get
\[
(s_j\cdot \sigma_{w,h}^T)(v)
=
\operatorname{sh}_{k-i}
\Bigl(
(s_{j-(k-i)}\cdot
\sigma_{{}^{\da i}w,{}^{\da i}h}^T)({}^{\da i}v)
\Bigr)\,
\sigma_{w^{\da i},h^{\da i}}^T(v^{\da i})\,
\left.e^T(N_\iota)\right|_{({}^{\da i}v,v^{\da i})}.
\]
If \(j<k-i\), then \(s_j\) acts on the second factor, and we get
\[
(s_j\cdot \sigma_{w,h}^T)(v)
=
\operatorname{sh}_{k-i}
\bigl(
\sigma_{{}^{\da i}w,{}^{\da i}h}^T({}^{\da i}v)
\bigr)\,
(s_j\cdot \sigma_{w^{\da i},h^{\da i}}^T)(v^{\da i})\,
\left.e^T(N_\iota)\right|_{({}^{\da i}v,v^{\da i})}.
\]
This proves the stated blockwise formulas and completes the proof.
\end{proof}

\begin{definition} Let  $h_1$ and $h_2$ be two  Hessenberg functions on $[k_1]$ and $[k_2]$, respectively, and $u\in \Sn{k_1}$, $v\in\Sn{k_2}$ be two permutations. Then we define another Hessenberg function on $[k_1+k_2]$ and a permutation in $\Sn{k_1+k_2}$ as follows:
\[(h_1 : h_2)\coloneqq (h_1(1),\dots,h_1(k_1-1),h_1(k_1)+1,h_2(1)+k_1,\dots,h_2(k_2)+k_1)\,,\]
\[(u : v)\coloneqq (u_1+k_2)\ \dots \ (u_{k_1}+k_2)\ v_1 \ \dots\ v_{k_2}.\]

We also use the notation $(\sigma_u: \sigma_v)$ to denote the class $\sigma_{(u:v), h}$ where $h=(h_1:h_2)$. We then extend the definition of $(\sigma_u: \sigma_v)$ linearly; 
\[ \left(\big(\sum_{u\in\Sn{k_1}} c_u\sigma_u\big) : \big(\sum_{v\in\Sn{k_2}} d_v\sigma_v \big)\right)\coloneqq \sum_{u\in\Sn{k_1}}\sum_{v\in\Sn{k_2}} c_u d_v (\sigma_u: \sigma_v)\,\,\,\,\text{for }\,\,\, c_u, d_v\in \mathbb{C}\,.\]
\end{definition}

Applying Theorem~\ref{thm:separation} to $h=(h_1:h_2)$, we obtain the following.
\begin{corollary}
    Let  $h_1$ and $h_2$ be two  Hessenberg functions on $[k_1]$ and $[k_2]$, respectively, and $u\in \Sn{k_1}$, $v\in\Sn{k_2}$ be two permutations. If $h=(h_1:h_2)$, then
    \[
    s_i\cdot \sigma_{(u:v)} 
    =\begin{cases}
        ((s_{i-k_2}\cdot \sigma_u) :\sigma_v) & \text{ for }k_2<i<k_1+k_2,\\
        (\sigma_u:(s_i\cdot\sigma_v))&\text{ for }i<k_2.
    \end{cases}
    \]
\end{corollary}

\begin{example}
Let $h=(3, 3, 4, 5, 6, 7, 7)$ be  a Hessenberg function on $[7]$. Then $h=(^{\da 3}h : h^{\da 3})$ and $h = (^{\da 4}h : h^{\da 4})$, but $h\neq (^{\da 2}h : h^{\da 2})=(2,3,4,5,6,7,7)$.
\end{example}


\section{Set \texorpdfstring{$\mathcal{A}_{s_i, w}$}{Asi,w}}\label{sec:A-set}

For arbitrary permutations $u,v\in\Sn{k}$, little is known about the intersection
\(
\Omega_u^\circ\cap \Omega_{v,h}.
\)
This makes it difficult to determine the set $\mathcal{A}_{s_i,w}$. Moreover,
even if this set is known, there is no general method for explicitly computing
the classes $\tau_u$'s. These difficulties present a major
obstacle to applying Proposition~\ref{prop:simple action} to Conjecture~\ref{conj:CHL}.
In this section, we develop several tools to overcome these difficulties.

Let \( h \) be a Hessenberg function on \( [k] \). 
Recall that \[\mathcal{A}_{s_i,w}\coloneqq \{u\in\Omega_{s_iw,h}^T\cap \Omega_w^T \mid \dim_\C(\Omega_u^\circ \cap \Omega_{s_iw,h})=\dim_\C\Omega_{w,h}\text{ and }u\dasharrow s_i u\}.\]
For simplicity, we denote by $\mathcal{A}_i$ the set $\mathcal{A}_{s_i,s_i}$. Then the following holds.

\begin{proposition}[Corollary 4.9 in \cite{CHL2}]\label{prop:s_i}
Let $h$ be a Hessenberg function. Then 
\[s_i\cdot \sigma_{s_i, h}=\sigma_{s_i, h}+\sum_{u\in \mathcal{A}_i}\left(\sigma_{u, h}-\sigma_{s_iu, h} \right)= \sigma_{s_i, h}+\sum_{u\in \mathcal{A}_i}\left(\sigma_{u, h}-s_i\cdot\sigma_{u, h} \right)\,,   \]
where $\mathcal{A}_i=\{ u\,|\, s_i\leq u, \ell_h(u)=1, u \dasharrow s_iu \}$.
\end{proposition}

The following lemma is an easy consequence, but it is useful.
\begin{lemma}\label{lem:Hess_Sch_decomp}
    If $u\geq v$ in the Bruhat order, then $\Omega_u^\circ\cap \Omega_{v,h}\subset \Omega_{u,h}^\circ.$
\end{lemma}
\begin{proof}
    Since $\Omega_{v,h}\subset \Omega_v\cap \Hess(S,h)$, if $u\geq v$, then we get
    \begin{align*}
        \Omega_u^\circ\cap \Omega_{v,h}
        &\subset \Omega_u^\circ \cap (\Omega_v\cap\Hess(S,h))\\
        &=\Omega_u^\circ\cap \left(\Bigl(\bigsqcup_{z\geq v}\Omega_{z}^\circ\Bigr) \cap\Hess(S,h)\right)\\
        &= \bigsqcup_{z\geq v}(\Omega_u^\circ\cap\Omega_z^\circ\cap \Hess(S,h))\\
        &=\Omega_{u}^\circ \cap \Hess(S,h),
    \end{align*}
    which proves the lemma.
\end{proof}

For a permutation \( w\in \Sn{k} \) and a simple reflection $s_i$ satisfying \( w \ra s_{i}w \), we define another set:
\[ 
\Atil{i} \coloneqq \set{ u\in\Omega_{s_i w,h}^{T} \cap \Omega_{w}^{T} \mid u \dra s_i u, \lh{u}\leq \lh{w} }\,.
\]

\begin{lemma}\label{lem:Atil} For a Hessenberg function \( h \) on \( [k] \), if \( w \ra s_{i}w \) and \( s_{i}w\in \Gh \), then 
\[ \Atil{i}=\{u\in [w, w_0] \,|\, u \dra s_i u,\, \lh{u}\leq \lh{w} \}, \] and \(\mathcal{A}_{s_i, w}\) is a subset of \(\Atil{i}\).
\end{lemma}
\begin{proof}
Let $w$ be a permutation $w$ satisfying $w\rightarrow s_iw$ and $s_iw\in \Gh$. Then $$\Omega_{s_iw,h}^T\cap \Omega_w^T=[s_iw,w_0]\cap[w,w_0]= [w,w_0].$$

Now we assume $u\in [w,w_0]$. Then $u\geq s_iw$. By Lemma~\ref{lem:Hess_Sch_decomp}, we get $\Omega_u^\circ\cap \Omega_{s_iw,h}\subset \Omega_{u,h}^\circ$.
Therefore, if $u\in \mathcal{A}_{s_i, w}$, then we get $$\dim(\Omega_{w,h})=\dim(\Omega_u^\circ \cap \Omega_{s_iw,h})\leq \dim (\Omega_{u,h}^\circ),$$ so $\ell_h(w)\geq \ell_h(u)$.
\end{proof}

Now we introduce the notion of $h$-associated patterns, and we will show that $\Atil{i}$ becomes $\mathcal{A}_{s_i,w}$ under a suitable condition related to the $h$-associated patterns.

\begin{definition}\cite[Definition~3.9]{CHP}
Let $h$ be a Hessenberg function on $[k]$. An $h$-admissible permutation $w\in \Gh$ \emph{contains the $h$-associated pattern $\hpat{\pi}$} if there exist indices $i<j<m<\ell$ satisfying the following conditions:
\begin{enumerate}
    \item \textbf{$\hpat{2143}$}: $w(j)<w(i)<w(\ell)<w(m)$, with $i<j<m<\ell\leq h(i)$.

    \item \textbf{$\hpat{1324}$}: $w(i)<w(m)<w(j)<w(\ell)$, with $i<j<m<\ell\leq h(j)$ and $m \leq h(i)$.

    \item \textbf{$\hpat{1243}$}: $w(i)<w(j)<w(\ell)<w(m)$, with $i<j<m<\ell\leq h(j)$ and $j \leq h(i)<\ell$.

    \item \textbf{$\hpat{2134}$}: $w(j)<w(i)<w(m)<w(\ell)$, with $i<j<m<\ell\leq h(m)$ and $m \leq h(i)<\ell$.

    \item \textbf{$\hpat{1423}$}: $w(i)<w(m)<w(\ell)<w(j)$, with $i<j<m<\ell\leq h(j)$ and $m \leq h(i)<\ell$.

    \item \textbf{$\hpat{2314}$}: $w(m)<w(i)<w(j)<w(\ell)$, with $i<j<m<\ell\leq h(j)$ and $m \leq h(i)<\ell$.

    \item \textbf{$\hpat{2413}$}: $w(m)<w(i)<w(\ell)<w(j)$, with $i<j\leq h(i)<m\leq h(j)<\ell\leq h(m)$.
\end{enumerate}
We say that $w$ avoids $\pi_h$ if $w$ does not contain the pattern $\pi_h$.
\end{definition}

\begin{example}
When $h=(4,4,4,5,6,7,7)$ and $w={1\, 4\,2 \,5}7\,3\,6$, the permutation $w$ contains the $h$-associated pattern \( \hpat{{213}4} \)  
since $w(3)<w(2)<w(4)<w(5)$ where $2<3<4<5\leq h(4)=5$ and $4\leq h(2)=4\leq 5$. One can check that $w$ contains the $h$-associated pattern  \( \hpat{{1324}} \), but avoids the $h$-associated pattern  \( \hpat{{2143}} \).
\end{example}

\begin{corollary}\cite[Corollary 4.9]{HLP}\label{cor:smooth}
    Let $h$ be a Hessenberg function on $[k]$. For an $h$-admissible permutation $w\in\Sn{k}$, the intersection $\Omega_{w}\cap\Hess(S,h)$ is smooth if and only if $w$ avoids all the $h$-associated patterns $\pi_h$. Furthermore, if $w$ avoids all the $h$-associated patterns, then $\Omega_{w,h}=\Omega_w\cap \Hess(S,h)$ and it is smooth.
\end{corollary}

\begin{example}
    \begin{enumerate}
        \item Let $h=(k,\dots,k)$. Then $\mathcal{G}_h=\Sn{k}$, and every $w\in\Sn{k}$ avoids the $h$-associated patterns except possibly $\hpat{2143}$ and $\hpat{1324}$. In this case, $\Hess(S,h)=\flag{k}$, so $\Omega_{w}$ is smooth if and only if $w$ avoids the patterns $2143$ and $1324$.
        \item Let $h=(2,3,\dots,k,k)$. Then every $w\in\Gh$ avoids all the $h$-associated patterns. Therefore, every Hessenberg Schubert variety is smooth in this case.
        \item Let $h=(m,\dots,m,m+1,\dots,k,k)$ with $m\geq 1$. Then every $w\in\Gh$ avoids the patterns except possibly $\hpat{2143}$, $\hpat{1324}$, and $\hpat{2134}$. In this case, if $w$ avoids $\hpat{{2143}}$, $\hpat{{1324}}$, and $\hpat{{213}4}$, then $\Omega_{w,h}=\Omega_w\cap\Hess(S,h)$ and it is smooth.
    \end{enumerate}
\end{example}

Although computing $\tau_u$ for $u\in \mathcal A_{s_i, w}$ is difficult in general, we can fortunately compute it when $s_iw$ avoids the $h$-associated patterns.

\begin{lemma}\label{lem:key}Let $h$ be a Hessenberg function on $[k]$. When  \( w \ra s_{i}w \)\,, \( s_{i}w\in \Gh \), and $s_i w$ avoids all the $h$-associated patterns, we have 
\[ \mathcal A_{s_i, w}=\Atil{i}=\{u\in [w, w_0] \,|\, u \dra s_i u,\, \lh{u} = \lh{w} \}. \]
Moreover, $\tau_u=\sigma_{u,h}^T$ and $\tau_{s_iu}=\sigma_{s_iu,h}^T$ for all $u\in \mathcal A_{s_i, w}$, that implies
\[ s_{i}\cdot \sigma^T_{w,h}=\sigma^T_{w,h}+(t_{i+1}-t_{i})\cdot \sigma^T_{s_{i}w,h}+\sum_{u\in \Atil{i}}(\sigma^T_{u,h}-\sigma^T_{s_iu,h}).\] 
\end{lemma}
\begin{proof}
Suppose that  \( w \ra s_{i}w \)\,, \( s_{i}w\in \Gh \), and $s_i w$ avoids all the $h$-associated patterns. Then
$\Omega_{s_iw,h}=\Omega_{s_iw}\cap \Hess(S,h)$ by Corollary~\ref{cor:smooth}. Therefore, for each $u\in [w,w_0]$, we have
\begin{align*}
\Omega_u^\circ \cap \Omega_{s_iw,h}&=\Omega_u^\circ \cap \Omega_{s_iw}\cap\Hess(S,h)\\
&=\Omega_u^\circ\cap\left(\bigsqcup_{v\geq s_iw}\Omega_v^\circ \right)\cap\Hess(S,h)\\
&=\Omega_u^\circ\cap\Hess(S,h).
\end{align*}
Hence we get
\[\mathcal{T}_u=\overline{\Omega_u^\circ \cap\Omega_{s_iw,h}}=\overline{\Omega_u^\circ\cap\Hess(S,h)}=\Omega_{u,h}.\]
Using the same argument, we also obtain $\mathcal{T}_{s_iu}=\Omega_{s_iu,h}$. Therefore, we get $\tau_u=\sigma_{u,h}$ and $\tau_{s_iu}=\sigma_{s_iu,h}$. Since \(\tau_u\) and \(\tau_{s_i u}\) appear in the formula for
\(s_i\cdot \sigma_{w,h}^T\), they have the same cohomological degree as
\(\sigma_{w,h}^T\). Hence
\(
\ell_h(u)=\ell_h(s_i u)=\ell_h(w).
\)
Therefore, the set $\mathcal{A}_{s_i,w}$ becomes the set
\[\Atil{i}=\{u\in [w, w_0] \,|\, u \dra s_i u,\, \lh{u} = \lh{w} \}. \]
Applying this to Proposition~\ref{prop:simple action}, the latter part immediately follows.
\end{proof}

For a permutation \( w\in \Sn{k} \), if \( w \ra s_{i}w \), then we define a set 
\[ 
\Ahat{i} \coloneqq \set{ u\in \Omega_{s_i w,h}^{T} \cap \Omega_{w}^{T} \mid u \dra s_i u }\,.
\]

Note that we have $\mathcal{A}_{s_i,w}\subset \Ahat{i}$  and $\Atil{i}\subset \Ahat{i}$.

\begin{lemma}\label{lem:Ahat0}
For a Hessenberg function \( h=(m,\dots,m,m+1,\dots,k,k) \) with $m\geq 1$, let \( w \in \Sn{m+n} \) such that \( w \ra s_i w \).  If \( w \) is an $h$-admissible permutation that avoids the $h$-associated pattern \( \hpat{{213}4} \), then \( \Ahat{i} \) is empty. Therefore, \[ s_{i}\cdot \sigma^T_{w,h}=\sigma^T_{w,h}+(t_{i+1}-t_{i})\cdot \sigma^T_{s_{i}w,h}.\]
Consequently, we have $s_i\cdot\sigma_{w,h}=\sigma_{w,h}$.
\end{lemma}

\begin{proof}
For \(u\in [w,w_0]\) with \(u>s_i u\), we claim that \(u\to s_i u\). This
implies that \(\widehat{\mathcal{A}}_{s_i,w}\) is empty.
We divide the proof into two cases according to whether $w^{-1}(i)>m$ or $w^{-1}(i)\le m$.

\medskip

\noindent\underline{Case 1: $w^{-1}(i)>m$.}
Since $w\to s_i w$, we have $w^{-1}(i+1)=w^{-1}(i)-1$.
We first show that
\[
w^{-1}(j)>w^{-1}(i)\quad\text{for all } j<i.
\]
By $h$-admissibility, we have
\[
w^{-1}(i)\le h\bigl(w^{-1}(i-1)\bigr).
\]
Since $w^{-1}(i)>m$, it follows that
\[
h\bigl(w^{-1}(i-1)\bigr)=w^{-1}(i-1)+1,
\]
and hence $w^{-1}(i-1)>w^{-1}(i)$.
By repeated applications of $h$-admissibility, the same argument yields
\begin{equation}\label{eq:less}
w^{-1}(j)>w^{-1}(i)\quad\text{for all } j<i.
\end{equation}

Similarly, since
\[
w^{-1}(i+2)\le h\bigl(w^{-1}(i+1)\bigr)=w^{-1}(i),
\]
we obtain $w^{-1}(i+2)<w^{-1}(i)$.
Again, by repeated applications of $h$-admissibility, we obtain
\begin{equation}\label{eq:more}
w^{-1}(j)<w^{-1}(i)\quad\text{for all } j\ge i+1.
\end{equation}

Combining~\eqref{eq:less} and~\eqref{eq:more}, we conclude that
\[
w^{-1}(i+1)=k-i
\quad\text{and}\quad
w^{-1}(i)=k-i+1,
\]
and hence $w[k-i+1,k]=[i]$.
Since $u>w$, it follows that
\[
u^{-1}(i+1)=k-i
\quad\text{and}\quad
u^{-1}(i)=k-i+1,
\]
and therefore $u\to s_i u$.

\medskip

\noindent\underline{Case 2: $w^{-1}(i)\le m$.}
In this case, we have
\[
w^{-1}(i+1)\le w^{-1}(i)-1<m.
\]
Note that if $[i,k]\subset w[m]$, then $[i,k]\subset u[m]$ since $u>w$, which implies that $u\to s_i u$.

If $w(m)>w(m+1)$, then by repeated applications of $h$-admissibility we obtain
$[i,k]\subset w[m]$, as desired.

Now suppose that $w(m)<w(m+1)$.
Let $a=\min w[m]$.
Since $w$ avoids $\hpat{{213}4}$, we have
\[
w^{-1}(a)<w^{-1}(a+1)<\cdots<w^{-1}(w(m))=m.
\]
Since $w^{-1}(i+1)<w^{-1}(i)\le m$, it follows that $i\geq w(m)$.
Moreover,
\[
w^{-1}(i+2)\le h\bigl(w^{-1}(i+1)\bigr)=m,
\]
and hence $w^{-1}(i+2)<m$.
By repeated applications of $h$-admissibility, we conclude that
\[
[i,k]\subset w[m].
\]

This completes the proof.
\end{proof}

The following lemma shows which Hessenberg Schubert classes 
can appear in the
expansion of \(\tau_u\).

\begin{lemma}\label{lem:tau2}  Suppose that $w\rightarrow s_iw$.
For $u\in\mathcal A_{s_i, w}$, if
\[
\tau_u=\sum_v c_v\,\sigma^T_{v,h}
\quad\text{and}\quad
\tau_{s_i u}=\sum_v d_v\,\sigma^T_{v,h},
\]
then
\[
\{v\mid c_v\neq 0\}\subset [u,w_0]_h
\quad\text{and}\quad
\{v\mid d_v\neq 0\}\subset [s_i u,w_0]_h.
\]
\end{lemma}

\begin{proof}  Since $u\in\mathcal A_{s_i, w}$, we have $u\geq w> s_iw$ and $s_iu\geq s_iw$ by Proposition~\ref{prop:lifting}. 
It follows from Lemma~\ref{lem:Hess_Sch_decomp} that 
$$\mathcal{T}_u=\overline{\Omega_u^\circ\cap\Omega_{s_iw,h}}\subset \overline{\Omega_{u,h}^\circ}=\Omega_{u,h}.$$ Hence we get
\[
\operatorname{supp}(\tau_u)=\operatorname{supp}([\mathcal{T}_u]) \subset \mathcal{T}_u^T
\subset
\Omega^T_{u,h}
\subset
[u,w_0]_h.
\]
Similarly, $\mathcal{T}_{s_iu}\subset \Omega_{s_iu,h}$, so \(\operatorname{supp}(\tau_{s_iu})\subset [s_iu,w_0]_h.\)

Let $v_0$ be a minimal element with respect to the $h$-Bruhat order
such that $c_{v_0}\neq 0$.
Then
\[
\tau_u(v_0)
=
\sum_v c_v\,\sigma^T_{v,h}(v_0)
=
c_{v_0}\,\sigma^T_{v_0,h}(v_0)
\neq 0
\]
by Proposition~\ref{prop:computing the value}.
Hence $v_0\in \operatorname{supp}(\tau_u)$, so $v_0\geq_h u$.

{Now let \(v\) be any permutation such that \(c_v\neq 0\). Choose a minimal
element \(v_0\), with respect to the \(h\)-Bruhat order, among the permutations
\(z\) satisfying
\[
z\leq_h v
\quad\text{and}\quad
c_z\neq 0.
\]
By the preceding argument applied to \(v_0\), we have \(u\leq_h v_0\). Therefore
\[
u\leq_h v_0\leq_h v,
\]}
which proves the desired inclusion
$\{v\mid c_v\neq 0\}\subset [u,w_0]_h$.

The same argument applied to \(\tau_{s_i u}\) gives
\(
\{\,v\mid d_v\neq 0\,\}\subset [s_i u,w_0]_h.
\)
\end{proof}


\section{Chromatic quaisymmetric functions of lollipop graphs}\label{sec:chromatic}

In this section, we study the Hessenberg varieties associated with the Hessenberg functions \(h=h_{m,n}\), whose incomparability graphs are lollipop graphs. By defining a composition
\(\bb(w)\) for each \(h_{m,n}\)-admissible permutation \(w\), we obtain an
elementary symmetric function expansion of \(X_{L_{m,n}}(\bx;q)\), which
yields the corresponding permutation module decomposition.

Let $m\geq 1$, $n\geq 0$, and $k=m+n$. We consider the Hessenberg function 
\[
h=h_{m,n} \coloneqq (\underbrace{m,\dots,m}_{m-1},m+1,\dots,k-1,k,k),
\]
with the convention that \(h_{m,0}=(m,\dots,m)\).
Then the incomparability graph associated with $h_{m,n}$ is the $(m,n)$-lollipop graph. More precisely, it consists of a complete graph on the vertices $\{1,2,\ldots,m\}$ and a path graph on the vertices $\{m,m+1,\ldots,m+n\}$, sharing the vertex $m$. We denote this graph by $L_{m,n}$.

If $n=0$, then $L_{m,n}=K_m$, the complete graph. If $m=1$ or $m=2$, then $L_{m,n}=P_{m+n}$, the path graph.

We first record a simple consequence of \(h_{m,n}\)-admissibility.

\begin{lemma}\label{lem:admissible}
    Let $h=h_{m,n}$ with $n>0$. For $w\in\Gh$, if $w_m<w_{m+1}$, then $w_m=w_{m+1}-1$.
\end{lemma}
\begin{proof}
    Let $w\in \Gh$ with $w_m<w_{m+1}$. Set $$x\coloneqq w_{m+1}\quad\text{ and }\quad A\coloneqq \{w_i\mid 1\leq i\leq m,\,w_i<x\}.$$ Since $w_m<w_{m+1}$, we have $A\neq\emptyset$. Let $y\coloneqq \max A$. We claim that $y=x-1$ and $y=w_m$. Suppose $y<x-1$. Then $y+1<x$ and $y+1\not\in A$. Since $w^{-1}(y)\leq m$, we have $w^{-1}(y+1)\leq h(w^{-1}(y))\leq m+1$. Since $w^{-1}(y+1)\neq m+1$, we obtain $w^{-1}(y+1)\leq m$, which contradicts to $y+1\not\in A$. Therefore, $y=x-1$. 
    
    On the other hand, the $h$-admissibility of $w$ implies that $$m+1=w^{-1}((x-1)+1)\leq h(w^{-1}(y)).$$ Since $y\in A$, we have $h(w^{-1}(y))\leq m+1$, so $h(w^{-1}(y))=m+1$. Since \(h=h_{m,n}\) with \(n>0\), this forces \(w^{-1}(y)=m\). Therefore, $w^{-1}(y)=m$, and consequently, $w_m=y=x-1=w_{m+1}-1.$
\end{proof}

We now define a composition and a partition associated with a permutation \(w\), when the Hessenberg function is \(h=h_{m,n}\).

\begin{definition}\label{def:lambda}  
Let $w=w_1\cdots w_k\in\Sn{k}$, where $k=m+n$.
We define a composition $\bb(w)$ of $k$ by the following procedure.

\begin{enumerate}
    \item We place a dashed bar immediately before $w_m$ if $w_m=\min\{w_i\mid 1\leq i\leq m\}.$
   \item For \(i=m+1, m+2, \dots, m+n\), if \(w_{i-1}>w_i\) then
    \begin{itemize}
       \item[-]  place a dashed bar immediately before \(w_i\) if there is a bar immediately before \(w_{i-1}\); and
       \item[-]  place a solid bar immediately before \(w_i\), otherwise.
    \end{itemize}
\end{enumerate}

After this procedure, we decompose \(w\) into blocks using only the solid bars. Let
    \[
    \bb(w)=(b_1,b_2,\ldots,b_r)
    \]
    be the sequence of the lengths of these blocks, read from left to right. We let $\lambda(w)$ be the partition of $k$ that we obtain by rearranging the parts of $\bb(w)$ in nonincreasing  order.
\end{definition}

\begin{remark} If $h=(2, 3, 4, \dots, k, k)$ then there are two $(m,n)$ such that $h=h_{m,n}$; either $(m,n)=(1, k-1)$ or $(m,n)=(2,k-2)$; while one can check that $\bb(w)$ is independent of the choice of $(m,n)$.\footnote{
In this case, the composition $\bb(w)$ coincides with the erasure $\widehat{\mathbf{a}}(w)$ in  \cite[Definition~5.23]{CHL1}.}
\end{remark}

\begin{remark} The set $\Gh$ of $h$-admissible permutations and the set of acyclic orientations on $G_h$ are in bijection; see \cite{CHL2}. Thus, our  $\lambda$ from $\Gh$ to the set of partitions for $h=h_{m,n}$ defines a map from the set of acyclic orientations on $L_{m,n}$. Moreover, the map respects the well-known theorem on the relation of the coefficients of the $e$-expansion and the acyclic orientations due to Stanley \cite[Theorem~3.3]{S1}, Shareshian and Wachs \cite[Theorem~5.3]{SW}; in the sense that an acyclic orientation with $j$ sinks is mapped to a partition of length $j$.  
\end{remark}

The following lemma immediately follows from the definition of $\bb(w)$.

\begin{lemma}\label{lem:lambda_1}
   Let $w\in \Gh$ for $h=h_{m, n}$ and suppose that
\(\bb(w)=(b_1,\dots,b_\ell)\). Then $b_1\geq m$. Furthermore, if \(b_1=m\), then
\[
w_m>w_{m+1}
\quad\text{and}\quad
w_m\neq \min\{w_1,\dots,w_m\}.
\]
If \(b_1>m\), then either \(w_m<w_{m+1}\), or
\[
w_m>w_{m+1}
\quad\text{and}\quad
w_m=\min\{w_1,\dots,w_m\}.
\]
\end{lemma}

From now on, we shade the first $m$ entries of $w\in\Sn{k}$, where $k=m+n$, to distinguish them from the remaining entries. Table~\ref{tab:lambda-example} illustrates the construction of $\bb(w)$
from several permutations.

\begin{table}[ht]
\centering
\begin{tabular}{|c|c|c|c|}
\hline
$(m,n)$ & $w\in\Sn{m+n}$ & Construction & $\bb(w)$  \\
\hline
\mystr $(2,2)$ & $w=\cblock{4\ 2} 3\ 1$ &  $ \cblock{4\dashbar 2} 3 \solidbar 1$ & $(3,1)$\\[6pt]
\hline
\mystr $(3,3)$ & $w=\cblock{5\ 6\ 4} 3\ 1\ 2$   &$\cblock{5\ 6 \dashbar 4 }\dashbar 3 \dashbar 1\ 2$  & $(6)$\\[6pt]
\hline
\mystr $(3,3)$ & $w=\cblock{4\ 6\ 5} 3\ 1\ 2$  & $\cblock{4\ 6\ 5}\solidbar 3\dashbar 1\ 2$ &$(3,3)$\\[6pt]
\hline
\end{tabular}
\caption{Examples of the construction of $\bb(w)$ from permutations.}
\label{tab:lambda-example}
\end{table}

\begin{example}
Let $h=h_{3, 4}$ and $w= \cblock{7\, 5\, 6} 4\, 1\, 2\, 3 \in \Gh$. Then $\bb(w)=(3,4)$, $^{\da 3}w=3\, 1\, 2\in\mathcal{G}_{^{\da 3}h}$ and $w^{\da 3}=4\,1\,2\,3 \in \mathcal{G}_{h^{\da 3}}$. Furthermore, $w=(^{\da 3}w:w^{\da 3})$.
\end{example}

\begin{lemma}\label{lem:separation} Let $h=h_{m, n}$ and $w\in \Gh$. Assume $\bb(w)=(b_1, \dots, b_\ell)\neq (k)$  where $k=m+n$. For $j=1, \dots, \ell-1$, let $i_j\coloneqq b_1+\cdots+b_{j}$. Then $\{w_{i_j+1}, \dots, w_k\}=[k-i_j]$. In particular, $(w^{-1})_1>m$ and $\bb(w^{\da i_j})=(b_{j+1}, \dots,b_\ell)$ where $w^{\da i_j}=w_{i_j+1}\,\cdots w_{k}$.
\end{lemma}
\begin{proof} We use an induction on $n$. If $n=0$ then $\bb(w)=(k)$ for all $w\in\Gh$ and there is nothing to prove. Let $n>0$ and suppose that the assertion is true for all $0\leq n'<n$. Let $w\in \Gh$ have $\bb(w)=(b_1, \dots, b_\ell)\neq (k)$. We know from the definition of $\bb(w)$ that $i_{\ell-1}\geq m$,  $w_{i_{\ell-1}} > w_{i_{\ell-1}+1}$ and $w_{i_{\ell-1}+1}$ is the minimum of the set $\{w_{i_{\ell-1}+1},\dots,  w_k\}$. Thus, if $x\coloneqq w_{i_{\ell-1}+1}>1$ then $x-1\in \{w_1, \dots, w_{i_{\ell-1}-1} \}$ while $h(1)\leq \cdots \leq h(i_{\ell-1}-1)\leq i_{\ell-1}<(w^{-1})_x$. This is a contradiction to the $h$-admissibility of $w$. Therefore, we obtain $w_{i_{\ell-1}+1}=1$ and a similar argument shows that $(w_{i_{\ell-1}+1},\dots,  w_k)=(1, 2, \dots, k-i_{\ell-1})$. Hence $w$ admits an $i_{\ell-1}$-block decomposition and $^{\da {i_{\ell-1}}}w$ is $^{\da {i_{\ell-1}}}h$-admissible where $^{\da {i_{\ell-1}}}h=h_{m,n'}$ with $n'<n$ and we can apply the induction hypothesis to finish the proof.
\end{proof}

From Lemma~\ref{lem:separation} we have the following corollary.
\begin{corollary}\label{cor:separation}
Let $w\in \Gh$ for $h=h_{m, n}$ and $\bb(w)=(b_1, \dots, b_\ell)\neq (k)$  where $k=m+n$. For $j=1, \dots, \ell-1$ let $i_j\coloneqq b_1+\cdots+b_{j}$. Then $w$ admits an $i_j$-block decomposition
and $h=(^{\da i_j}h : h^{\da i_j})$. 
\end{corollary}

For a positive integer $m$, we define $$[m]_q\coloneqq \sum_{i=0}^{m-1} \, q^i\quad\text{ and }\quad [m]_q!\coloneqq [m]_q[m-1]_q\cdots [1]_q.$$ In the following lemma we state basic properties of chromatic quasisymmetric functions $X_G(\bx ; q)$. 

\begin{lemma}[\cite{SW}, Section 2]\label{lem:X_G}
\begin{enumerate}
   \item If a graph $G$ is a disjoint union of two graphs $G^1$ and $G^2$, that is $G=G^1\bigsqcup G^2$, then 
      $$X_{G}(\bx ; q)=X_{G^1}(\bx ; q)X_{G^2}(\bx ; q)\,.$$
   \item If $h=h_{m,0}$ then $\Gh=\Sn{m}$ and $\bb(w)=\lambda(w)=(m)$ for all $w\in \Gh$. Moreover  $$X_{G_h}(\bx ; q)=X_{K_m}(\bx ; q)=\sum_{w\in \Sn{m}} q^{\ell(w)}e_{\lambda(w)}=\sum_{w\in \Sn{m}} q^{\ell(w)}e_{m}=[m]_q! e_{m}\,. $$    
\end{enumerate}
\end{lemma}

\begin{lemma}[\cite{CHL1}, Theorem 5.27]\label{lem:X_path} If $h=h_{2, n}$ then 
    $$X_{G_h}(\bx ; q)=X_{P_{n+2}}(\bx ; q)=\sum_{w\in \Gh} q^{\ell_h(w)}e_{\lambda(w)}\,.$$ 
\end{lemma}

 \begin{proposition}[\cite{HNY}, Proposition~4.3]\label{prop:HNY}
Let $m\geq 2$ and $n\geq 1$. The chromatic quasisymmetric functions $X_{L_{m,n}}(\bx;q)$ of lollipop graphs are uniquely determined by the following recurrence relation:
 \begin{equation}\label{eq:recurrence}
     [m]_q X_{L_{m,n}}(\bx;q)= X_{L_{m+1,n-1}}(\bx;q)+ q[m-1]_q X_{K_m \bigsqcup P_n}(\bx;q)\,. 
 \end{equation} 
\end{proposition}

\begin{theorem}\label{thm:Acyclic} Let $h=h_{m,n}$. Then the chromatic quasisymmetric function $X_{G_h}$ expands into the elementary symmetric functions in the following way:
\[  X_{L_{m,n}}(\bx;q)=\sum_{w\in \Gh} q^{\ell_h(w)} e_{\lambda(w)}(\bx)\,. \]
Equivalently, the cohomology of the Hessenberg variety associated with $h$ decomposes into permutation modules as follows: For all $0\leq d \leq d_h$,
 \[H^{2d}(\Hess(S,h))=\bigoplus_{w\in \Gh^d}\ M^{\lambda(w)}\,.\] 
\end{theorem}
\begin{proof}
Let $Y_{L_{m,n}}(\bx;q)\coloneqq\sum_{w\in \Gh} q^{\ell_h(w)} e_{\lambda(w)}(\bx)$. Then by Lemma~\ref{lem:X_G}(2) and Lemma~\ref{lem:X_path}, we have $$Y_{K_m}(\bx;q)=X_{K_m}(\bx;q)\,\,\, \text{ and } \,\,\,Y_{P_n}(\bx;q)=X_{P_n}(\bx;q).$$ 
Thus it remains to show that $Y_{L_{m,n}}(\bx;q)$ satisfies the same recurrence relation (\ref{eq:recurrence})
as $X_{L_{m,n}}(\bx;q)$.
We show that the following recurrence relation is satisfied in what follows, where we note that $X_{K_m \bigsqcup P_n} =X_{K_m}X_{P_n}$:
 \begin{equation*}
     [m]_q Y_{L_{m,n}}(\bx;q)= Y_{L_{m+1,n-1}}(\bx;q)+ q[m-1]_q Y_{K_m}(\bx;q)Y_{P_n}(\bx;q)\,. 
 \end{equation*} 
 
Since \(e_{\bb(w)}=e_{\lambda(w)}\), we shall use \(e_{\bb(w)}\) in the proof. Let $h\coloneqq h_{m,n}$, and 
\[A\coloneqq \{w\in\Gh \mid \bb(w)_1>m \},\qquad B\coloneqq \{w\in\Gh \mid \bb(w)_1=m \}.\]
Then $\Gh= A \bigsqcup B$, and we get
$$ [m]_q Y_{L_{m,n}} = [m]_q  \sum_{w\in A} q^{\ell_h(w)}e_{\bb(w)}+ [m]_q \sum_{w\in B} q^{\ell_h(w)}e_{\bb(w)}\,.$$ 
We claim that 
$$[m]_q \sum_{w\in A} q^{\ell_h(w)}e_{\bb(w)}= Y_{L_{m+1,n-1}}\,\,\,\text{ and}$$
$$[m]_q \sum_{w\in B} q^{\ell_h(w)}e_{\bb(w)}= q[m-1]_q Y_{K_m}Y_{P_n}\,.$$

Let us consider 
$[m]_q  \sum_{w\in A} q^{\ell_h(w)}e_{\bb(w)}=(1+q+\cdots+q^{m-1})\sum_{w\in A} q^{\ell_h(w)}e_{\bb(w)}$ first.
Let $h'\coloneqq h_{m+1, n-1}$, and we find a bijective map $\varphi\colon\{0,1,\dots,m-1\}\times A\to \mathcal{G}_{h'}$, $\varphi(i,w)\coloneqq w'$, satisfying $\ell_{h'}(w')=\ell_h(w)+i$ and $\bb(w)=\bb(w')$ in the following. 

Note that if $w\in A$, then either $w_m<w_{m+1}$, or $w_m>w_{m+1} $ and $w_m=\min\{w_1,\dots,w_{m}\}$ by Lemma~\ref{lem:lambda_1}. In particular, if $w_m<w_{m+1}$, then $w_m=w_{m+1}-1$ by Lemma~\ref{lem:admissible}.
For given $w\in A$, let
\[U(w)\coloneqq \{1\leq j<m\mid w_j>w_{m+1}\}\quad \text{and}\quad L(w)\coloneqq \{1\leq j<m\mid w_j<w_{m+1}\}.\] Letting $i_w\coloneqq |U(w)|$, we can write $U(w)=\{p_{i_w}<\cdots<p_2<p_1\}$ and $L(w)=\{q_{m-1-i_w}<\cdots<q_2<q_1\}.$
Note that if $w_m>w_{m+1}$, then $U(w)=[m-1]$ and $L(w)=\emptyset$. Setting $j_w=|i_w-i|$, we then define
\[
\varphi(i,w)=w'\coloneqq\begin{cases}
    w&\text{if }i_w=i,\\
    ws_{p_1,m}s_{p_2,p_1}\dots s_{p_{j_w},p_{{j_w}-1}} &\text{if }i_w>i,\\
    ws_{q_1,m}s_{q_2,q_1}\dots s_{q_{j_w},q_{{j_w}-1}} &\text{if }i_w<i.
\end{cases}
\]
Since $w^{\da m}=(w')^{\da m}$ and $\bb(w)_1>m$, it is clear that $w'\in\mathcal{G}_{h'}$ and $\bb(w)=\bb(w')$. Moreover,
\[
\ell_{h'}(w')-\ell_h(w)=\begin{cases}
    i_w&\text{if }i=i_w,\\
    i_w-(i_w-i)&\text{if }i_w>i,\\
    i_w+(i-i_w)&\text{if }i_w<i.
\end{cases}
\]
Thus $\ell_{h'}(w')=\ell_h(w)+i$ in any case.
In the following, we show that for each $w'\in\mathcal{G}_{h'}$, there exists a unique pair $(i,w)\in \{0,1,\dots,m-1\}\times A$ such that $\varphi(i,w)=w'$. This proves that $\varphi$ is bijective. We divide the proof into two cases according to whether $w'_{m+1}=\min\{w'_1,\dots,w'_{m+1}\}$ or not.

If $w'_{m+1}=\min\{w'_1,\dots,w'_{m+1}\}$, then we find $j$ satisfying $w'_{j}= \min\{w'_1,\dots,w'_m\}$. Define $w\coloneqq w's_{j}s_{j+1}\dots s_{m-1}$. 
Then $w_m>w_{m+1}$ and $w_m=\min\{w_1,\dots,w_m\}$, so $w\in A$. Moreover, we have $w'=\varphi(j-1,w)$ and $\ell_{h'}(w')-\ell_h(w)=(m-1)-(m-j)=j-1$.

If $w'_{m+1}\neq\min\{w'_1,\dots,w'_{m+1}\}$, we can find $j$ satisfying $1\leq j\leq m$ and $w'_j=w'_{m+1}-1$ since $w'$ is $h'$-admissible. We let
\[U'(w')\coloneqq\{j<j'\leq m\mid w'_{j'}>w'_{m+1}\}\text{ and }L'(w')\coloneqq\{j<j'\leq m\mid w'_{j'}<w'_{m+1}\}.\]
Set
\[
i_{w'}\coloneqq 
\left|\{1\leq j'\leq m\mid w'_{j'}>w'_{m+1}\}\right|.
\]
Write
\[
U'(w')=\{p'_1<\cdots<p'_{|U'(w')|}\}
\quad\text{and}\quad
L'(w')=\{q'_1<\cdots<q'_{|L'(w')|}\}.
\]
We define \((i,w)\) by
\[
(i,w)\coloneqq
\begin{cases}
\left(i_{w'},\,w'\right),
&\text{if } w'_m=w'_{m+1}-1,\\[4pt]
\left(
i_{w'}-|U'(w')|,\,
w's_{j,p'_1}s_{p'_1,p'_2}\cdots s_{p'_{|U'(w')|},m}
\right),
&\text{if } w'_m>w'_{m+1}-1,\\[4pt]
\left(
i_{w'}+|L'(w')|,\,
w's_{j,q'_1}s_{q'_1,q'_2}\cdots s_{q'_{|L'(w')|},m}
\right),
&\text{if } w'_m<w'_{m+1}-1.
\end{cases}
\]
Then $w_m=w_{m+1}-1$, so $w\in A$. Moreover, we have $w'=\varphi(i,w)$.
Therefore, $\varphi$ is bijective.

\smallskip

We now consider $[m]_q  \sum_{w\in B} q^{\ell_h(w)}e_{\bb(w)}$. 
By Lemma~\ref{lem:lambda_1}, for $w\in B$, we have $w_m>w_{m+1}$ and $w_m\neq\min\{w_1,\dots,w_m\}$. Moreover, by Lemma~\ref{lem:separation},
we have $\{w_1, \dots, w_m\}=\{n+1, \dots, n+m\}$, which implies $w_{m}\neq n+1$.
Let $$H\coloneqq\{u\in\Sn{[n+1,n+m]}\mid u_{m+n}\neq n+1\},\quad H'=\Sn{[n+1,n+m]}\setminus H, \quad h''=h_{1,n-1}.$$ Then $w\in B$ if and only if there exists $(u,v)\in H\times \mathcal{G}_{h''}$ such that $$ u_{n+1}\ u_{n+2}\ \dots \ u_{m+n} =w_1\ w_2\ \dots\ w_m\quad\text{ and }\quad v=w_{m+1}\ w_{m+2}\ \dots \ w_{m+n}$$ by Lemma~\ref{lem:block_generator}. Then

\begin{align*}
[m]_q  \sum_{w\in B} q^{\ell_h(w)}e_{\bb(w)}&=[m]_q  \sum_{(u,v)\in H\times \mathcal{G}_{h''}} q^{\ell(u)+\ell_{h''}(v)+1} e_{\bb(u)}e_{\bb(v)}\\
&= q[m]_q \left(\sum_{u\in H} q^{\ell(u)} e_{(m)}\right)\left(\sum_{v\in \mathcal{G}_{h''}} q^{\ell_{h''}(v)} e_{\bb(v)}\right)\\
&=q[m]_q \left(\sum_{u\in \Sn{[n+1, n+m]}} q^{\ell(u)} e_{(m)}-\sum_{u\in H'} q^{\ell(u)} e_{(m)}\right)Y_{P_n}\\
&= q[m]_q\left([m]_q!-q^{m-1}[m-1]_q!\right)e_{(m)}Y_{P_n}\\
&=q[m-1]_q [m]_q! e_{(m)}Y_{P_n}\\
&=q[m-1]_q Y_{K_m}Y_{P_n}\,.
\end{align*}
This completes the proof.
\end{proof}


\section{Module generators}\label{sec:generators}

This section is devoted to the construction and study of the module generators 
\(\widehat{\sigma}_{w,h}\) for \(w\in\Gh^d\). We first define these elements
and restate Theorem~C. We then prove several properties
that \(\widehat{\sigma}_{w,h}\) satisfy, that will be used in Section~\ref{sec:decomposition}.

We only consider the Hessenberg functions  \( h=h_{m, n} \) on $[k]=[m+n]$ associated with the lollipop graph $L_{m,n}$ in this and all subsequent sections. We use $\sigma_{u}^T$ and $\sigma_{u}$ instead of $\sigma_{u,h}^T$  and $\sigma_{u,h}$, respectively, when there is no possible confusion.

\begin{definition}\label{def:module generators}
For $w\in \Gh$, let $\bb(w)=(b_1, \dots, b_\ell)$. Then, we define
\begin{enumerate}
\item  a subgroup $\mathcal{S}_w$, called the \emph{symmetrizer group of $\sigma_w$}, of $\Sn{k}$ as
\[\mathcal{S}_w=\Sn{\{w_1, \dots, w_{b_1}\}}\times \cdots \times \Sn{\{w_{b_1+\cdots+b_{\ell-1}+1}, \dots, w_{k}\} }, \text{ and }\]
\item an element $\widehat{\sigma}_w$ in $H^*(\Hess(S,h))$ as 
\[ \widehat{\sigma}_w= \sum_{v\in \mathcal{S}_w} v\cdot \sigma_w\]
\end{enumerate}
\end{definition}

\begin{example}\label{ex:sigma_hat}
For the permutations appearing in Table~\ref{tab:lambda-example}, we compute the symmetrizer group $\mathcal{S}_w$.
\begin{itemize}
\item[-] If $h=(2, 3, 4, 4)$ and  $w=\cblock{4\ 2} 3\ 1\in \Gh$,
then $\mathcal{S}_w = \Sn{\{4, 2, 3\}}\times \Sn{\{1\}}$.

\medskip
\item[-]
If $h=(3, 3, 4, 5, 6, 6)$ and $w=\cblock{5\ 6\ 4} 3\ 1\ 2\in \Gh$,
then $\mathcal{S}_w=\Sn{6}$.

\item[-]
If $h=(3, 3, 4, 5, 6, 6)$ and $w=\cblock{4\ 6\ 5} 3\ 1\ 2\in \Gh$,
then $\mathcal{S}_w=\Sn{\{4, 6, 5\}}\times\Sn{\{3, 1, 2\}}.$
\end{itemize}
\end{example}

Note that $b_1\geq m$ for any $w\in \Sn{k}$ when $\bb(w)=(b_1, \dots, b_\ell)$, by Lemma~\ref{lem:lambda_1}.

We are ready to state the main theorem, whose proof is deferred to Section~\ref{sec:decomposition}. 

\begin{theorem}\label{thm:main} Let $h=h_{m, n}$ and $k=m+n$, and let $d$ be an integer such that $0\leq d\leq d_h$. If we let $M=\sum_{u\in \Gh^d}\mathbb C\Sn{k}(\widehat{\sigma}_u)$, then $H^{2d}(\Hess(S,h))=M$ and it is decomposed as a direct sum of permutation modules:  
 \[H^{2d}(\Hess(S,h))=M=\bigoplus_{u\in \Gh^d}\mathbb C\Sn{k}(\widehat{\sigma}_u)\cong\bigoplus_{u\in \Gh^d} M^{\lambda(u)} \,.\] 
In particular,  $\Stab(\widehat{\sigma}_u)=\mathcal{S}_u$ and $\mathbb C\Sn{k}(\widehat{\sigma}_u)$ is isomorphic to the permutation module  $M^{\lambda(u)}$ for each $h$-admissible permutation $u\in \Gh^d$.
\end{theorem}

\begin{remark}
Since $\mathcal{S}_u=u\Sn{\bb(u)}u^{-1}$ for $u\in \Gh$, we know $\mathbb C\Sn{k}(\widehat{\sigma}_u)\cong \operatorname{Ind}_{\Sn{\bb(u)}}^{\Sn{k}} \mathbf{1}$, which is isomorphic to $\operatorname{Ind}_{\Sn{\lambda(u)}}^{\Sn{k}} \mathbf{1}=M^{\lambda(u)}$ in a natural way.
\end{remark} 
\begin{example}
Let $h=h_{3,2}$ and $0\leq d=3\leq 5=d_h$, then 
$$\Gh^3=\{\cblock{3\ 5\ 4} 2\ 1, \cblock{4\ 3\ 5} 2\ 1, \cblock{5\ 3\ 4} 1\ 2, \cblock{4\ 5\ 2} 3\ 1, \cblock{5\ 2\ 3} 4\ 1, \cblock{4\ 5\ 3} 1\ 2, \cblock{5\ 4\ 1} 2\ 3 \}\,,$$
and the corresponding compositions are given as follows:
\[\begin{array}{c|ccccccc}
\hline
         u\in \Gh^3           & \cblock{3\ 5\ 4} 2\ 1 & \cblock{4\ 3\ 5} 2\ 1 & \cblock{5\ 3\ 4} 1\ 2 & \cblock{4\ 5\ 2} 3\ 1 & \cblock{5\ 2\ 3} 4\ 1 & \cblock{4\ 5\ 3} 1\ 2, & \cblock{5\ 4\ 1} 2\ 3 \\
\hline
  \bb(u)                 & (3,2) & (3,2)  &(3,2) & (4,1) & (4,1) &(5) & (5)\\
\hline
\end{array}\]   
Therefore, by Theorem~\ref{thm:main} we have 
\[H^{6}(\Hess(S,h))\cong 3 M^{(3,2)}\oplus 2 M^{(4,1)}\oplus 2 M^{(5)} \,.\] 
\end{example}

In the rest of this section, we investigate the properties that the $h$-admissible permutations $u\in \Gh$ and the corresponding cohomology class $\sigma_u$ satisfy, and develop background theory that we will need in the proof of Theorem~\ref{thm:main}.

\begin{lemma} \label{lem:trivial}
Let $h=h_{m,n}$ with $n>0$. For $w\in \Gh$, $\bb(w)=(m+n)$ if and only if $w$ satisfies one of the following conditions:
\begin{enumerate}
    \item[i)] $w_m w_{m+1} \cdots w_{m+n}= (n+1)\,n\,\cdots 2 \, 1$,
    \item[ii)] $w_m<w_i$ for all $i<m$, and there is $0< j <n$ satisfying 
    \begin{align*}
        &w_m\, w_{m+1} \cdots w_{m+j-1} = (n+1)\,n\, \cdots (n-j+2)\quad \text{ and }\\
        &w_{m+j}\, w_{m+j+1}\,\cdots \, w_{m+n}= 1\,\,2\,\cdots \, (n-j+1),
    \end{align*}
    \item[iii)] $w_m\, w_{m+1}\, \cdots \, w_{m+n} = x\,(x+1) \, \cdots \, (x+n)$ for some $1\leq x<m$, and
    \item[iv)] $w_m w_{m+1} \cdots w_{m+n}= m\, (m+1)\,\cdots \, (m+n)$.
\end{enumerate}
\end{lemma} 

\begin{proof} It is clear that the permutations satisfying i), ii), iii), or iv) are $h$-admissible permutations with $\bb(w)=(k)$. 

Suppose that $w\in \Gh$ is an $h$-admissible permutation with $\bb(w)=(k)$. We divide the proof into two cases: $w_m<w_{m+1}$ and $w_m>w_{m+1}$. Let $x\coloneqq w_m$ for convenience.

    If $w_m<w_{m+1}$, then Lemma~\ref{lem:admissible} implies that $w_{m+1}=x+1$. Moreover, by the same argument as in the proof of Lemma~\ref{lem:admissible}, we have \[w_{m+2}\, \cdots w_{m+n} = (x+2)\, \cdots\, (x+n).\] Hence $w$ satisfies condition iii) or iv). 
    
    Now assume $w_m>w_{m+1}$. Since $\bb(w)=(k)$, it follows from Lemma~\ref{lem:lambda_1} that $w_m=\min\{w_1, \dots, w_m\}$, which implies that $x\leq n+1$.
We first show that $x=n+1$. Suppose that $x<n+1$. Then there are more than $m-1$ numbers in $[m+n]$ that are greater than $x$, and there exists $z>x$ such that $w_p=z$ with $p>m+1$. Assume that  $z$ is the smallest among such numbers. Then $w^{-1}(z-1)\leq m$. Since $w\in \Gh$, we get $w^{-1}(z)\leq m$, which is a contradiction.
This shows that $\{w_1, \dots, w_m\}=[n+1, k]$ and $\{w_{m+1}, \dots, w_{m+n}\}=[n]$.

Let $j$ be the integer such that $$w_{m}=n+1>w_{m+1}>\cdots> w_{m+j}<w_{m+j+1}$$ and $0<j\leq n$. If $j=n$, then $w$ satisfies condition i). If $j<n$, then $w_{m+j}< w_{m+j+1}<\cdots< w_{m+n}$ since $\bb(w)=(k)$, which implies that $w_{m+j}$ is the smallest among $w_{m+1}, \dots, w_{m+n}$; that is  $w_{m+j}=1$. We finally show that  $w_{m+j}\, w_{m+j+1}\, \cdots \, w_{m+n}= 1\,2\, \cdots\, (n-j+1)$. Suppose not. Then there exists $y\in S\coloneqq \{w_{m+j},w_{m+j+1}, \cdots, w_{m+n}\}$ such that $y> 1$ and $y-1$ is not in $S$. This does not occur since $y-1\in [n]$ and $w^{-1}(y)\leq h(w^{-1}(y-1))=w^{-1}(y-1)+1$.
This shows that $w$ satisfies condition ii).
\end{proof}

\begin{definition} 
We say that $w\in \Gh$ is \emph{of type i), ii), iii) or iv)} if $\bb(w)=(k)$ and $w$ satisfies the corresponding condition in Lemma~\ref{lem:trivial}.
\end{definition}

\begin{remark} Permutations of type i) can be realized as the ones of type ii) with $j=0$, and permutations of type iv) can be understood as the ones of type iii) with $x=m$. However, we distinguish permutations of type i) and iv) from the ones of type ii) and iii), respectively, since they have different features in our proof of the main theorem; Theorem~\ref{thm:main}.
\end{remark}

\begin{example}\label{ex:(k)}
Let $h=h_{4, 3}=(4,4,4, 5, 6, 7,7)$.
\begin{itemize} 
\item[-] If  $w=\cblock{7\ 5\ 6\ 4}3\ 2\ 1\in \Gh$, then $\bb(w)=(7)$ and $w$ is of type i).

\item[-]
If  $w=\cblock{6\ 7\ 5\ 4}1\ 2\ 3\in \Gh$, then $\bb(w)=(7)$ and  $w$ is of type ii) with $j=1$. 

\item[-]
If  $w=\cblock{5\ 7\ 6\ 1}2\ 3\ 4\in \Gh$, then $\bb(w)=(7)$ and  $w$ is of type iii) with $x=1$.

\item[-]
If  $w=\cblock{2\ 7\ 1\ 3}4\ 5\ 6\in \Gh$, then $\bb(w)=(7)$ and  $w$ is of type iii) with $x=3$.  

\item[-]
If $w=\cblock{2\ 3\ 1\ 4}5\ 6\ 7\in \Gh$, then $\bb(w)=(7)$ and  $w$ is of type iv).

\end{itemize}
\end{example}

\begin{definition}\label{def:index}
    Let $h=h_{m,n}$ and $k=m+n$. For each $u\in\Sn{k}$, we assign a pair $(j_u, x_u)$ of nonnegative integers as follows:
    \begin{enumerate}
        \item If $w\in\Gh$, define $x_w\coloneqq \min\{ w_j\mid m\leq j\leq m+n\}$ and $j_w\coloneqq   (w^{-1})_{x_w}-m$.
        \item If $u\in\Sn{k}$, define $(j_u, x_u)\coloneqq (j_w, x_w)$ where $w=\widetilde{u}$ is the unique $h$-admissible permutation such that $u\in P(w)$.
    \end{enumerate}
\end{definition}

\begin{remark}
One can easily check the following.
\begin{enumerate}
\item For  $w\in\Gh$ with $\bb(w)\neq (k)$, $(j_w, x_w)=( (w^{-1})_1-m,1)$, so $j_w>0$ by Lemma~\ref{lem:separation}.
\item For $w\in\Gh$ with $\bb(w)=(k)$, 
$$
(j_w, x_w)=\begin{cases}
			( (w^{-1})_1-m,1) & \text{if $w$ satisfies the condition i) or ii)} \\
			(0,w_m) & \text{if $w$ satisfies the condition iii) or iv)}
		 \end{cases}
$$ 
    \item If  $w\in\Gh$ is of type i) or ii), then $(w^{-1})_1 >m$, so $j_w>0$. 
    \item For each $u\in\Sn{k}$ if we let $w\coloneqq \widetilde{u}$ then $w_{m+j_u}=x_u$. In particular, for each $w\in \Gh$, $w_{m+j_w}=x_w$. 
\end{enumerate}

\end{remark}

\begin{definition}[Partial order]\label{def:order}  Let $h=h_{m,n}$, and $k=m+n$. 
We define a partial order on the set $\Sn{k}$ of permutations as follows: For two permutations $u$ and $v$ in $\Sn{k}$, 
\[u \ll v \text{ if and only if } (j_u>j_v) \text{ or } (j_u=j_v=0 \text{ and } x_u<x_v).\]
\end{definition}

For a permutation $w$ in Examples~\ref{ex:sigma_hat} and~\ref{ex:(k)}, the pair $(j_w,x_w)$ are computed in Table~\ref{tab:new_order}. 
\begin{enumerate}
    \item The permutations $\cblock{5\ 6\ 4}3\ 1\ 2$ and $\cblock{4\ 6\ 5}3\ 1\ 2$ are not comparable  in the partial order associated with $h=h_{3,3}$ .
    \item For five permutations in the table associated with $h_{4,3}$, we get
    $$\cblock{7\ 5\ 6\ 4}3\ 2\ 1\ll \cblock{6\ 7\ 5\ 4}1\ 2\ 3\ll  \cblock{5\ 7\ 6\ 1}2\ 3\ 4\ll \cblock{2\ 7\ 1\ 3}4\ 5\ 6\ll \cblock{2\ 3\ 1\ 4}5\ 6\ 7.$$
\end{enumerate}

\begin{table}[h]
\centering
\begin{tabular}{|c||c|c|c|c|c|c|c|c|}
\hline
$h$ & $h_{2,2}$ & $h_{3,3}$ & $h_{3,3}$  & $h_{4,3}$  & $h_{4,3}$ & $h_{4,3}$ & $h_{4,3}$ & $h_{4,3}$ \\
\hline
$w$ & $\cblock{42}31$ & $\cblock{564}312$ & $\cblock{465}312$ & $\cblock{7564}321$  & $\cblock{6754}123$ & $\cblock{5761}234$ & $\cblock{2713}456$ & $\cblock{2314}567$ \\
\hline
$(j_w,x_w)$ & $(2,1)$  & $(2,1)$  & $(2,1)$  & $(3,1)$ & $(1,1)$ & $(0,1)$ & $(0,3)$ & $(0,4)$ \\
\hline
\end{tabular}
\caption{Examples of computations of $(j_w,x_w)$}
\label{tab:new_order}
\end{table}

\begin{lemma}\label{lem:u}
   Let $h=h_{m,n}$ and $w\in\Gh$ be an $h$-admissible permutation. For $u\in P(w)$, let $j=j_u=j_w$ and $x=x_u=x_w$.
   \begin{enumerate} 
   \item If $\bb(u)=\bb(w)=(k)$, then $u_{m+j}=x$.    
   \item If $\bb(u)=\bb(w) \neq (k)$, then $(u^{-1})_1\leq m+j$
   \end{enumerate}
\end{lemma}
\begin{proof}
(1) Suppose that $w$ is of type i) or ii) so that $j_w>0$ and $x_w=1$. For $u\in P(w)$, since $u_m>\cdots>u_{m+j}<\cdots<u_{k}$ and $u_m<u_i$ for all $i\in [m-1]$, $u_{m+j}$ must be the smallest; that is $u_{m+j}=1$.\\ Suppose that  $w$ is of type iii) or iv) so that $j_w=0$. Then $w_m=x<\cdots<w_k=x+n$ and there are exactly $x-1$ elements that are smaller than $x$ in $\{ w_1, \dots, w_{m-1}\}$. Since $u$ and $w$ have the same relative order along every edge of $G_h$, 
we get 
$u_m=x$. 

(2) Let $\bb(u)=(b_1,\dots, b_\ell)$, $\ell>1$. By Corollary~\ref{cor:separation}, we have $m+j\geq k-b_\ell+1$ and $w_{m+j}=1<\cdots <w_k$, which implies that $u_{m+j}<\cdots <u_k$. Therefore, either $u_{m+j}=1$, or $u_{m+j}> 1$ and $(u^{-1})_1< m+j$.
\end{proof}

\begin{example}
Recall that when $h=h_{3, 3}$ and $w=\cblock{4\,6\,5}3\,1\,2\in \Gh$, we get $\bb(w)=(3,3)$ and $j_w=2$. For $u= \cblock{ 1\,6\,5} 3\,2\,4\in P(w)$, we get $(u^{-1})_1=1<m+j_w=5$.
\end{example}

\begin{proposition}[Lemma 3.7 and Proposition 3.8 in \cite{CHL2}]\label{prop:representative}
Let $h$ be a Hessenberg function on $[k]$.
\begin{enumerate}
\item If $u\dasharrow s_iu$ for a simple reflrection $s_i$, then there is a unique $h$-admissible permutation $w$ such that both $u$ and $s_iu$ belong to $P(w)$.
\item For $w\in \Gh$  and $u\in P(w)$, there is a sequence of simple reflections $s_{a_1}, \dots, s_{a_l}$ satisfying $u=s_{a_l}\cdots s_{a_1} w$ and $w\dasharrow s_{a_1}w \dasharrow \cdots \dasharrow s_{a_l} \cdots s_{a_1} w=u$, that implies $\sigma_u^T=s_{a_l}\cdots s_{a_1}\cdot\sigma_w^T$, and hence $\sigma_u=s_{a_l}\cdots s_{a_1}\cdot\sigma_w$.
\end{enumerate}
\end{proposition}

\begin{proposition}[Theorem 2 in \cite{Stumbo}] \label{prop:coset representative}
The minimal length (left) coset representatives of the subgroup $\Sn{\{1, 2, \dots, i\}}\times \Sn{\{i+1, i+2, \dots, k\}}$ of $\Sn{k}$ is 
\[ \{w_{k-1}(l_{k-1}) w_{k-2}(l_{k-2})\cdots w_i(l_i) \,|\, i\geq l_i\geq l_{i+1}\geq \cdots \geq l_{k-1} \geq 0  \}\,,  \]
where $w_j(l_j)=s_{j-l_j+1}\cdots s_{j-1}s_j$.
\end{proposition} 

In the following lemma, we compute the set $P(w)$ for $w\in \Gh$ with $\bb(w)=(k)$.
\begin{lemma}\label{lem:P(w)}
    Let $h=h_{m,n}$ with $n>0$ and $w\in \Gh$.  If $w$ is of type i) or iv), then $P(w)=\{w\}$; if $w$ is of type ii) or iii), then $P(w)$ is the set \[\{w_{k-1}(l_{k-1})w_{k-2}(l_{k-2})\cdots w_{n-j+x}(l_{n-j+x})w\,|\,n-j\geq l_{n-j+x}\geq l_{n-j+x+1}\geq \cdots \geq l_{k-1}\geq 0 \}.\] 
\end{lemma}
\begin{proof}
    We begin with the case where $w$ is of type i), and then treat types ii)--iv) uniformly.

    Suppose that $w$ is of type i). Then, for $u\in P(w)$, we have $u_m>u_{m+1}>\cdots>u_{m+n}$ and $u_i>u_m$ for $i=1, 2, \dots, m-1$. Thus $u_m=n+1$, and the other $u_j$'s are uniquely determined. This shows that $P(w)=\{w\}$.

    Now we assume that $w$ is of type ii), iii), or iv). Let $(j,x)\coloneqq (j_w,x_w)$. Then $w_{m+j}=x<w_{m+j+1}<\dots<w_{k}$. If $u\in P(w)$ then $u_{m+j}=x$ and $u$ can be identified with a sequence $x=u_{m+j}<u_{m+j+1}<u_{m+j+2}< \cdots < u_{k}$. Therefore there are $\binom{k-x}{n-j}$ permutations in $P(w)$. 
    Letting $$S\coloneqq \{w_{k-1}(l_{k-1}) w_{k-2}(l_{k-2})\cdots w_{n-j+x}(l_{n-j+x})\mid  n-j\geq l_{n-j+x}\geq l_{n-j+x+1}\geq \cdots \geq l_{k-1}\},$$ we can check that  $w \dasharrow s_{a_1}w \dasharrow s_{a_2}s_{a_1}w \dasharrow s_{a_l}\cdots s_{a_1}w$ when $s_{a_l}\cdots s_{a_1}$ is a permutation in $S$. Therefore, by Proposition~\ref{prop:representative} we have $P(w)\subset \{uw\,|\, u\in S\}$. We then finish the proof by computing $|S|$, the number of decreasing functions from $[n-j+x, k-1]$ to $[0,n-j]$, which is  $\binom{k-x}{n-j}$.
\end{proof}

In the following lemma, we investigate the stabilizers of $\sigma_w$ that we need for the subsequent argument. 
 
\begin{lemma}\label{lem:stabilizer}
    Let $h=h_{m,n}$ with $n>0$ and let $w\in \Gh$ with $(j_w,x_w)=(j,w)$.
    \begin{enumerate}
        \item If $w$ is of type i), then $\Stab(\sigma_w)=\Sn{k}$, so we have $\widehat{\sigma}_w=k!\sigma_w$.
        \item If $w$ is of type ii) or iii) with $x=1$, then $\Sn{\{ 1,2 \dots, n-j+1 \}}\times \Sn{\{n-j+2, \dots, k-1, k \}}$ is a subgroup of $\Stab(\sigma_w)$.
        \item If $w$ is of type iii) with $x>1$, then $\Sn{\{ x-1,x, \dots, n+x \}}\times \Sn{\{n+x+1, \dots, k-1, k \}}$ is a subgroup of $\Stab(\sigma_w)$.
        \item If $w$ is of type iv), then $\Sn{\{m-1,m,\dots,m+n\}}$ is a subgroup of $\Stab(\sigma_w)$.
    \end{enumerate}
\end{lemma}
\begin{proof}
    We first consider when $w$ is of type i). Then $w$ admits an $m$-block decomposition. By Theorem~\ref{thm:separation}, for $n<j<k$, $s_j$ acts trivially on $\sigma_w$ since the dot action is trivial for the complete graph case. For  $j\leq n$, since $ w\rightarrow s_jw$ and  \( w \) is a \( \hpat{2134} \)-avoiding $h$-admissible permutation, $s_j$ also acts trivially on $\sigma_w $ by Lemma~\ref{lem:Ahat0}. Therefore, we obtain $\widehat{\sigma}_w=k!\, \sigma_w$. 

    Next, suppose that $w$ is of type ii), or $w$ is of type iii) with $x=1$. 
    Since $w^{-1}(1)<w^{-1}(2)<\cdots<w^{-1}(n-j+1)$, we get $s_iw\to w$ for $1\leq i\leq n-j$, so $s_1, \dots, s_{n-j}$ stabilize $\sigma_w$. On the other hand, since $w(m)=\min\{w(i)\mid 1\leq i\leq m\}$, $w$ avoids the pattern $\hpat{2134}$. For $n-j+2\leq i\leq k$, since $w\to s_i w$, it follows from Lemma~\ref{lem:Ahat0} that $\Sn{\{n-j+2, \dots, k-1, k\}}$ is a subgroup of $\Stab(\sigma_w)$.

    Finally, we assume that $w$ is of type iii) with $x>1$, or $w$ is of type iv). Since $s_iw\to w$ for $i=x-1,x,\dots,x+n-1$, $\Sn{\{x-1,x,\dots,x+n\}}$ is a subgroup of $\Stab(\sigma_w)$. Thus the case of type iv) is settled. We now continue with the case of type iii) with $x>1$. For $i=x+n+1, \dots, k-1$, if  $s_iw \rightarrow w$ then $s_i$ stabilizes $\sigma_w$, thus let us assume that $w \rightarrow s_iw$.  Note that $s_iw$ is an $h$-admissible permutation, and hence $\mathcal A_{s_i, w}\subset \Atil{i}=\{u\in [w, w_0] \,|\, u \dra s_i u,\, \lh{u}\leq \lh{w} \}$ by Lemma~\ref{lem:Atil}. If $u\in  \Atil{i}$, then $u \dra s_i u$, which implies that $(u^{-1})_{i}>m$. Hence $u\notin [w, w_0]$ because $w_j, j\in [m, m+n]$, are all smaller than or equal to $x+n<i$. This is a contradiction. Therefore, $\mathcal A_{s_i, w}=\Atil{i}=\varnothing$ and $s_i\cdot \sigma_w=\sigma_w$, which proves the desired assertion for type iii) with $x>1$.
\end{proof}

\begin{lemma}\label{lem: iii)-1} Let $h=h_{m,n}$ with $n>0$ and let $w\in \Gh$ be an $h$-admissible permutation of type iii) with $x\coloneqq x_w>1$ or of type iv). 
Then, for $i\in [x-2]$,
\[s_i\cdot\sigma_w=\sigma_w+\sum_{v\ll w} c_v \sigma_v \text{\,\,\, for some constant \,\,\, $c_v$'s}\,.\] 
\end{lemma}
\begin{proof} 
We first note that since we assume $x>1$, for each $i\in [x-2]$ the following hold:
\begin{itemize}
    \item both $i$ and $i+1$ are in $\{w_j\mid j\in [m-1]\}$, and
    \item $w_j=(s_iw)_j=x+j-m\geq x$ for $j=m,\dots,k$.
\end{itemize}
If $s_iw\to w$, then $s_i$ stabilizes $\sigma_w$, so the claim holds. Now we assume that $w\to s_iw$. Since $s_iw\in\Gh$, we have $$s_i\cdot \sigma_w=\sigma_w+\sum_{v\in \mathcal{A}_{s_i,w}}(\tau_v-\tau_{s_iv}),$$ where $\mathcal{A}_{s_i,w}$ is a subset of $\Atil{i}=\{v\in [w, w_0] \,|\, v \dra s_i v,\, \lh{v}\leq \lh{w} \}$ by Lemma~\ref{lem:Atil}.

We claim that if $v'\not\ll w$, then $\sigma_{v'}$ does not appear in the sum $\sum_{v\in \mathcal{A}_{s_i,w}}(\tau_v-\tau_{s_iv}).$ Suppose not, i.e., there is a permutation $v'$ such that $\widetilde{v'}\in \Gh$ is of type iii) or iv) with $x_{v'}=x_{\widetilde{v'}}\geq x$ and $\sigma_{v'}$ appears in the sum $\sum_{v\in \mathcal{A}_{s_i,w}}(\tau_v-\tau_{s_iv}),$ which implies that $v'\geq w$ or $v'\geq s_iw$ by Lemma~\ref{lem:tau2}. However, since $v'\dra s_iv'$, $i+1$ belongs to $\{v'_{m+1},\dots,v'_{m+n}\}$, so $$\{v'_m,\dots,v'_{m+n}\}\up \not\geq \{w_m,\dots,w_{m+n}\}\up =\{(s_iw)_m,\dots,(s_iw)_{m+n}\}\up,$$ which is a contradiction.
\end{proof}

The following simple lemma will be used later in Section~\ref{sec:decomposition}.

\begin{lemma}\label{lem:wl} In $\Sn{k}$, for $1< x\leq l<k-1$, we have 
\[ (s_x s_{x+1} \cdots s_{l+1})(s_x s_{x+1} \cdots s_l)=(s_{x+1}\cdots s_{l+1})(s_x\cdots s_{l+1}).\]
\end{lemma} 
\begin{proof} If $l=2$ then $(s_2 s_3)s_2=(s_3)(s_2 s_3)$, so the lemma is true. Let $l>2$ and suppose that the lemma is true for all $2\leq l'<l$. Then 
\begin{align*}
(s_x s_{x+1} \cdots s_{l+1})(s_x s_{x+1} \cdots s_l)&=(s_x s_{x+1} \cdots s_{l})(s_x  \cdots  s_{l-1})(s_{l+1} s_l)\\
&=(s_{x+1}\cdots s_{l})(s_x\cdots s_{l})(s_{l+1} s_l)=(s_{x+1}\cdots s_{l})(s_x\cdots s_{l-1}s_{l+1}s_ls_{l+1})\\
& =(s_{x+1}\cdots s_{l}s_{l+1})(s_x\cdots s_{l-1}s_ls_{l+1})\,,
\end{align*}
which completes the proof.
\end{proof}

\begin{definition}\label{def:M}
Let $d$ be an integer such that $0\leq d\leq d_h$.
\begin{enumerate}
    \item We define a  $\mathbb C\Sn{k}$ submodule  $M$ of $H^{2d}(\Hess(S,h))$ as
\[M\coloneqq \sum_{u\in \Gh^d}\mathbb C\Sn{k}(\widehat{\sigma}_u)\subset H^{2d}(\Hess(S,h))\,.\]
\item For each $w\in\Gh^d$, we define a submodule $M_{\ll w}$ of $M$ as 
\[M_{\ll w}\coloneqq \sum_{u\in \Gh^d, u\ll w}\mathbb{C} \Sn{k}(\widehat{\sigma}_u).\]
\item For $\alpha,\beta \in H^{2d}(\Hess(S,h))$, we use the notation $\alpha \equiv\beta \pmod{M'}$ when $\alpha-\beta \in M'$ for a $\mathbb{C} \Sn{k}$ submodule $M'$ of $H^{2d}(\Hess(S,h))$.
\end{enumerate}
\end{definition}

We prove a proposition that will play a crucial role in the proof of  Theorem~\ref{thm:main}. 

\begin{proposition}\label{lem2:main}
    Let $w\in \Gh^d$ be an $h$-admissible permutation with $\bb(w)=(k)$ satisfying $0\leq j=j_w<n$ and $x_w=1$.
    Let $$\widetilde{S}\coloneqq\{w_{k-1}(\ell_{k-1})w_{k-2}(\ell_{k-2})\cdots w_{n-j+1}(\ell_{n-j+1})\mid n-j+1\geq \ell_{n-j+1}\geq \cdots\ell_{k-1}\geq 0\},$$ and
    $$ S\coloneqq\{w_{k-1}(\ell_{k-1})w_{k-2}(\ell_{k-2})\cdots w_{n-j+1}(\ell_{n-j+1})\mid n-j+1> \ell_{n-j+1}\geq \cdots\ell_{k-1}\geq 0\}.$$
    For $w^*\coloneqq (w_{n-j+1}(n-j)) w$, if $s_1\cdot\sigma_{w^*}\equiv \sigma_{w^*}\pmod{M_{\ll w}}$, then we have
    $$\sum_{u\in\widetilde{S}-S}u\cdot\sigma_w\equiv \sum_{v\in P(w)}d_v\sigma_v\pmod{M_{\ll w}}$$
    for nonnegative integers $d_v$.
\end{proposition}
\begin{proof}
    For convenience, we make a table of the values of $w$ and $w^*$ at positions $p\coloneqq(w^{-1})_{n-j+2}$, and $m+j,m+j+1,\dots,m+n$.
    \[\begin{array}{c|ccccccccc}
    \hline
        &  p & m+j & m+j+1 & m+j+2 & \cdots & m+n\\
    \hline
    w & n-j+2 & 1 & 2&3&\cdots & n-j+1\\
    w^* &   2  & 1 & 3 & 4 &  \cdots & n-j+2\\
    \hline
    \end{array}\]
    
    Note that $P(w)=\{uw\mid u\in S\}$ by Lemma~\ref{lem:P(w)}, and $$\widetilde{S}-S=\{w_{k-1}(\ell_{k-1})w_{k-2}(\ell_{k-2})\cdots w_{n-j+1}(\ell_{n-j+1})\mid n-j+1 = \ell_{n-j+1}\geq \cdots\ell_{k-1}\geq 0\}.$$ Hence $w^*\in P(w)$, and we have $(w_{n-j+1}(n-j))\cdot\sigma_w=\sigma_{(w_{n-j+1}(n-j))w}=\sigma_{w^*}$ by Proposition~\ref{prop:simple action}(1). 

    Since we assume $s_1\cdot\sigma_{w^*}\equiv \sigma_{w^*}\pmod{M_{\ll w}}$, we get $$w_{n-j+1}(n-j+1)\cdot\sigma_w=s_1w_{n-j+1}(n-j)\cdot\sigma_w=s_1\cdot\sigma_{w^*}\equiv\sigma_{w^*} \pmod{M_{\ll w}},$$ and we also get
    \begin{align*}
        & (w_{k-1}(\ell_{k-1})\cdots w_{n-j+2}(\ell_{n-j+2})w_{n-j+1}{ (n-j+1)})\cdot\sigma_w\\
        &\equiv (w_{k-1}(\ell_{k-1})\cdots w_{n-j+2}{ (\ell_{n-j+2})})\cdot\sigma_{w^*}\pmod{M_{\ll w}}\\
        &=(w_{k-1}(\ell_{k-1})\cdots w_{n-j+2}{ (\ell_{n-j+2})}w_{n-j+1}(n-j))\cdot\sigma_{w}.
    \end{align*}
    Since $0\leq \ell_{k-1}\leq\cdots\leq \ell_{n-j+2}\leq \ell_{n-j+1}\leq n-j+1$, if none of $\ell_{k-1},\dots,\ell_{n-j+2}$ is equal to $n-j+1$, then by Lemma~\ref{lem:P(w)}, we have
    $$(w_{k-1}(\ell_{k-1})\cdots w_{n-j+2}(\ell_{n-j+2})w_{n-j+1}({ n-j+1}))\cdot\sigma_w\equiv \sigma_u \pmod{M_{\ll w}}$$ for some $u\in P(w)$.  Thus, we consider $u=w_{n-j+2}(n-j+1)w_{n-j+1}(n-j+1)$ next. Since  $w_{n-j+2}(n-j+1) = s_{2}\cdots s_{n-j+1}s_{n-j+2}$, we have 
    \begin{align*}
    &(w_{n-j+2}(n-j+1))(w_{n-j+1}(n-j+1))\cdot\sigma_w\\
    &\equiv s_2\cdots s_{n-j+1}s_{n-j+2}\cdot  \sigma_{w^*} \pmod{M_{\ll w}}\\
    &=(s_2\cdots s_{n-j+1}s_{n-j+2})(s_2\cdots s_{n-j}s_{n-j+1})\cdot \sigma_w \\
    &=(s_3\cdots s_{n-j+1}s_{n-j+2})(s_2\cdots s_{n-j+1}s_{n-j+2})\cdot \sigma_w\\
    &=(s_3\cdots s_{n-j+1}s_{n-j+2})(s_2\cdots s_{n-j+1})\cdot \sigma_w\\
    &= \sigma_{w_{n-j+2}(n-j)w_{n-j+1}(n-j)w}\,.
    \end{align*}
    In the above, the third equality holds due to Lemma~\ref{lem:wl}, the fourth equality holds because $s_{n-j+2}$ stabilizes $\sigma_w$ by Lemma~\ref{lem:stabilizer}, and the last equality holds  by Proposition~\ref{prop:representative}. Therefore, if none of the $l_{k-1}, \dots, l_{n-j+3}$ is equal to $n-j+1$, then 
    \begin{equation*}
    (w_{k-1}(l_{k-1}) \cdots w_{n-j+3}(\ell_{n-j+3}) w_{n-j+2}(n-j+1)w_{n-j+1}(n-j+1))\cdot \sigma_w\equiv\sigma_{u} \pmod{M_{\ll w}}  
    \end{equation*}
    for some  $u\in P(w)$.
    We can follow the same argument to show that if $n-j+1=l_{n-j+2}=l_{n-j+3}=\cdots= l_a$ for $a\geq n-j+3$, then for any $n-j+1> l_{a+1}\geq \cdots \geq l_{k-1}$ we have
    \begin{equation*}
    (w_{k-1}(l_{k-1}) \cdots w_{n-j+2}(l_{n-j+2})w_{n-j+1}(n-j+1))\cdot \sigma_w\equiv\sigma_{u}  \pmod{M_{\ll w}}  
    \end{equation*}
    for some $u\in P(w)$.
    This let us conclude that for nonnegative integers $d_u$'s,
    $$\sum\left(w_{k-1}(l_{k-1}) \cdots w_{n-j+2}(l_{n-j+2})w_{n-j+1}(n-j+1)\right)\cdot\sigma_w \equiv \sum_{u\in P(w)} d_u\sigma_u \pmod{M_{\ll w}}
    \,,$$
     where the sum is over the sequences ${n-j+1\geq l_{n-j+2}\geq\cdots\geq l_{k-1}}$.
\end{proof}

\section{Partner Permutations}\label{sec:new}

In understanding $M_{\ll w}$ for $w\in \Gh$ of type ii) or iii), it is very useful to consider the behavior of a certain permutation, called the partner of $w$. In this section we define the partner permutations and investigate the properties of them in the relations to the class $\sigma_w$.  
\begin{definition}[Partner permutation]\label{def:partner} 
Let $w$ be an $h$-admissible permutation of type ii) or iii). Let $j=j_w$ and $x=x_w$. Then we define a permutation  $\underline{w}=s_{x, x+2}w'$ for $w'\coloneqq u^x w$, which we call  the \emph{partner permutation of $w$}, where 
\[
    u^x=\begin{cases}
        s_{x+2}\cdots s_{x+n-j} &\text{ if }n-j>1, \text{ and}\\
        id &\text{ if }n-j=1.
    \end{cases}
\]
\end{definition}

\begin{remark}\label{rmk:partner}
Let $w$ be an $h$-admissible permutation $w$ of type ii) or iii).
\begin{enumerate}
    \item When $n-j>1$, $u^{x}$ is the cycle permutation $(x+2, x+3, \dots , x+n-j,  x+n-j+1)$.
    \item Since $j<n$, we have $(w_{m+j}, w_{m+j+1}, \cdots ,w_{m+n}) = (x, x+1, \cdots, x+n-j)$. Then the permutation $w'=u^xw$ satisfies that $((w')^{-1})_{x+2}=(w^{-1})_{x+n-j+1}<m+j$ and
    $$(w'_{m+j}, w'_{m+j+1},w'_{m+j+2}, \cdots ,w'_{m+n}) = (x, x+1, x+3, \cdots, x+n-j+1).$$
    \item The partner permutation $\underline{w}$ is not necessarily  $h$-admissible, and $\bb(\underline{w})\neq (k)$. 
\end{enumerate}
\end{remark}

\begin{example}\label{ex:underbar}
\begin{enumerate}

\item For $w=\cblock{6\, 7\, 5\, 4} 1\, 2\, 3$ in the second example of Example~\ref{ex:(k)}, 
$\bb(w)=(7)$ and  $w$ is of type ii) with $j=1$. Here, $x=1$, $u^x=(3,4)=s_3$, $w'=s_3 w= \cblock{6\,7\,5\,3} 1\,2\,4\in P(w)$ and $\underline{w}=s_{1, 3}w'= \cblock{6\,7\,5\,1} 3\,2\,4 < w'$ in the Bruhat order. \\
Note that $\underline{w}\not\in \Gh$; $\widetilde{\underline{w}}=\cblock{ 6\,7\,5\,3} 4\,1\,2$ and $j_{\underline{w}}= j_{\widetilde{\underline{w}}}=2=j_w+1$. Moreover $\bb(\underline{w})=(5,2)$.

\item For $w=\cblock{5\, 7\, 6\, 1} 2\, 3\, 4$ in the third example of Example~\ref{ex:(k)}, 
$\bb(w)=(7)$ and  $w$ is of type  iii) with $x=1$. Here $u^x=(3, 4, 5)=s_3s_4$, $w'=(3, 4, 5) w= \cblock{3\,7\,6\,1} 2\,4\,5\in P(w)$ and $\underline{w}=s_{1, 3}w'=\cblock{1\,7\,6\,3}2\,4\,5<w'$ in the Bruhat order. Note that $\bb(\underline{w})=(4,3)$.

\item If $h=h_{3,4}$ and $w=\cblock{1\,7\,2}3\,4\,5\,6\,\in\Gh$,
 then $\bb(w)=(7)$ and $w$ is of type iii).
In this case $x=2$, $u^x=(4, 5, 6,7)=s_4s_5s_6$, $w'=u^x w=\cblock{1\,4\,2} 3\, 5\,6\,7\in P(w)$ and $\underline{w}=s_{2, 4}w'= \cblock{1\,2\,4}3\,5\,6\,7<w'$ in the Bruhat order. 
Note also that $w'\in P(w)$, $\bb(\underline{w})=(3, 4)$ and $s_{3} \underline{w}\in\Gh$.
\end{enumerate} 
\end{example}

Let $w\in\Gh$ be a permutation of type ii) or iii). 
Setting $(j,x)\coloneqq (j_w,x_w)$, we make a table of values of permutations $w$, $w'=u^x w$, $s_{x+1}w'$, $\underline{w}$, $s_{x+1}\underline{w}$, and $\widetilde{\underline{w}}$ at positions $p\coloneqq (\underline{w}^{-1})_x$, $m+j,m+j+1,\dots,m+n$. Note that if $w$ is of type ii), then $x=1$ and $p=m+j-1\geq m$; if $w$ is of type iii), then $j=0$ and $p<m$. This table will be used repeatedly when we verify properties of the partner permutation.
    \begin{table}[ht]
    \centering
    \[\begin{array}{c|ccccccccc}
    \hline
        &  p & m+j & m+j+1 & m+j+2 & \cdots & m+n\\
    \hline
    w & { x+n-j+1} & x & x+1&x+2&\cdots & x+n-j\\
    w' &   x+2  & x & x+1 & x+3 &  \cdots & x+n-j+1\\
     s_{x+1}w'              &   x+1  & x & x+2 & x+3 &  \cdots & x+n-j+1\\
     \underline{w} & x & x+2 & x+1 & x+3 & \cdots & x+n-j+1\\
     s_{x+1}\underline{w} & x & x+1 & x+2& x+3& \cdots & x+n-j+1\\
    \widetilde{\underline{w}} & y-1 & y & 1 & 2 & \cdots & x+n-j\\
    \hline
    \end{array}\]
    \caption{Values of the relevant permutations at selected positions.}
\label{tab:permutation-values}
\end{table}

\begin{lemma} \label{lem:partner} Let $h=h_{m,n}$ with $n>0$,  
and  $w\in \Gh$ be an $h$-admissible permutation of type ii) or iii).
\begin{enumerate}
\item $w' \dasharrow s_{x+1}w'$, and both $w'$ and $s_{x+1}w'$ are in $P(w)$.
\item $\ell_h(w)=\ell_h(w')=\ell_h(\underline{w})$,  $\ell(w')=\ell(\underline{w})+1$, and $\underline{w}<w'$ in the Bruhat order.
\item $\underline{w} \rightarrow s_{x+1} \underline{w}$\, and $s_{x+1} \underline{w}\in\Gh$.
\item If $u\in \mathcal A_{s_{x+1}, \underline{w}}$, then $\underline{w}< u$ and $s_{x+1}\underline{w}< s_{x+1}u$.
\item $\bb(\underline{w})=(m+j_w, n-j_w)$.
\item $\underline{w}\not\in \Gh$ and $j_{\underline{w}}=j_{\widetilde{\underline{w}}}=j+1$.
\end{enumerate}
\end{lemma}
\begin{proof} 
\begin{description}[leftmargin=0pt, labelsep=0.5em, itemindent=0pt, itemsep=0.5em]
    \item[(1)] It is clear that $w' \dasharrow s_{x+1}w'$. Note that $u^x$ keeps the relative orders among the last $n-j+1$ numbers and the relative orders among the first $m+j$ numbers in $w$, see Remark~\ref{rmk:partner}(2). Accordingly, $w'\in P(w)$. Moreover, since $w' \dasharrow s_{x+1}w'$, Proposition~\ref{prop:representative}(1) implies that  $w'$ and $s_{x+1}w'$ are in $P(w)$.
\item[(2)] From (1), we have $\ell_h(w)=\ell_h(w')$. Note that for $v\in \Sn{k}$, if we let $L(v)\coloneqq \{(i,i')\mid i<i'\leq h(i)\text{ and }\,v(i)>v(i')\}$, then $\ell_h(v)=|L(v)|$. Now let $w'_a=x+2$ and $w'_b=x$. Then $w'_{b+1}=x+1$, $(a,b)\in L(w')$, and $(b,b+1)\not\in L(w')$. On the other hand, $(a,b)\not\in L(\underline{w})$, $(b,b+1)\in L(\underline{w})$, and $L(w')\setminus\{(a,b)\}=L(\underline{w})\setminus\{(b,b+1)\}$. Therefore, $\ell_h(\underline{w})=\ell_h(w')$ and $\underline{w}<w'$ in the Bruhat order.

\item[(3)] It is clear that $\underline{w}\to s_{x+1}\underline{w}$. For simplicity, let $v=s_{x+1}\underline{w}$. Then $(v^{-1})_x=(w^{-1})_{x+n-j+1}$ and $$(v_{m+j},v_{m+j+1},\dots,v_{m+n})=(x+1,x+2,\dots,x+n-j+1),$$ which implies that $v\in\Gh$. 

\item[(4)] Since $\underline{w}\to s_{x+1}\underline{w}$, if $u\in \mathcal{A}_{s_i,w}$, then $u>\underline{w}$, so $u>s_{x+1}\underline{w}$. Since $x+1\in D_L(u)-D_L(s_{x+1}w)$, by the lifting property of the Bruhat order (Proposition~\ref{prop:lifting}), we get $s_{x+1}\underline{w}<s_{x+1}u$.
\item[(5)] This is clear from the definition of $\underline{w}$.
\item[(6)] From Table~\ref{tab:permutation-values}, we have
    \[(\underline{w}^{-1})_{\underline{w}_p+1}=(\underline{w}^{-1})_{x+1}=m+j+1>m+j=h(p),\]
    so $\underline{w}\not\in \Gh$. Furthermore, $j_{\underline{w}}=j_{\widetilde{\underline{w}}}=j+1$.
\end{description}
\end{proof}

Note that Lemma~\ref{lem:partner} allows us to apply several useful results, including Theorem~\ref{thm:separation} and Lemma~\ref{lem:Atil}, to the partner permutation $\underline{w}$ when $w\in\Gh$ of type ii) or iii). Moreover, $\underline{w}\ll w$.

\begin{proposition}\label{prop:P(w)_classes}
    Let $h=h_{m,n}$ with $n>0$, and  $w\in \Gh$ be an $h$-admissible permutation of type ii) or iii). Let $x\coloneqq x_w$. Then,  
    for any $u\in P(w)$, we have $$\sigma_w-\sigma_u \in \mathbb{C}\Sn{k}(\sigma_{w'}-\sigma_{s_{x+1} w'}).$$
\end{proposition}
\begin{proof}

From Table~\ref{tab:permutation-values}, we obtain 
    $$w\dasharrow s_{x+n-j}w \dasharrow \cdots \dasharrow (s_{x+2}\cdots s_{x+n-j})w=w'$$
    and $s_{x+2},s_{x+3},\dots,s_{x+n-j}\in \Stab(\sigma_{s_{x+1}w'}).$ Then we have 
    \begin{equation*}
        \begin{split}
            (s_{x+n-j-1}\cdots s_{x+2})\cdot(\sigma_{w'}-\sigma_{s_{x+1} w'})&=\sigma_{(s_{x+n-j-1}\cdots s_{x+2})w'}-\sigma_{s_{x+1} w'}\\
    &=\sigma_{s_{x+n-j}w}-\sigma_{s_{x+1}w'}
    \in \mathbb{C}\Sn{k}(\sigma_{w'}-\sigma_{s_{x+1} w'})
        \end{split}
    \end{equation*}
    and
    \begin{equation*}
        \begin{split}
            (s_{x+n-j}\cdots s_{x+2})\cdot(\sigma_{w'}-\sigma_{s_{x+1} w'})&=\sigma_{(s_{x+n-j}\cdots s_{x+2})w'}-\sigma_{s_{x+1} w'}\\
    &=\sigma_w-\sigma_{s_{x+1}w'}\in \mathbb{C}\Sn{k}(\sigma_{w'}-\sigma_{s_{x+1} w'})\,,
        \end{split}
    \end{equation*}
    which implies that 
    \begin{equation}\label{eq1}
        \sigma_w-\sigma_{s_{x+n-j}w}\in \mathbb{C}\Sn{k}(\sigma_{w'}-\sigma_{s_{x+1} w'}).
    \end{equation}
    
    By Lemma~\ref{lem:P(w)}, 
    we have 
  \[P(w)=\{w_{k-1}(l_{k-1})w_{k-2}(l_{k-2})\cdots w_{n-j+x}(l_{n-j+x})w\,|\,n-j\geq l_{n-j+x}\geq l_{n-j+x+1}\geq \cdots \geq l_{k-1}\geq 0 \},\]
    so each $u\in P(w)$ corresponds to a decreasing sequence $n-j\geq l_{n-j+x}\geq l_{n-j+x+1}\geq \cdots \geq l_{k-1}\geq 0$. We use an induction on the sum $l_u\coloneqq \sum_{i=n-j+x}^{k-1} l_i$ of $u$. If $l_u=0$ then $u=w$ and the assertion is trivial.

    If $l_u=1$ then $u=s_{x+n-j}w$, so it follows from~\eqref{eq1} that $\sigma_w-\sigma_u\in\mathbb{C}\Sn{k}(\sigma_{w'}-\sigma_{s_{x+1} w'}).$

    Now we assume that the assertion holds for all $v$ with $l_v<l$ where $l>0$. Let $u\in P(w)$ be a permutation with $l_u=l$. Then there is a simple transposition $s_y$, { $y\geq x+1$,} such that $s_yu\in P(w)$,  { $s_y u\dasharrow  u$}, and $l_{s_yu}=l_u-1$ by Lemma~\ref{lem:P(w)}.  
    By the induction hypothesis, $\sigma_{w}-\sigma_{s_y u}$ is in $\mathbb{C}\Sn{k}(\sigma_{w'}-\sigma_{s_{x+1} w'})$. Since $s_y\in \Stab(\sigma_w)$ when $y\neq x+n-j$, by Lemma~\ref{lem:stabilizer},
    whereas
$w\dasharrow s_yw$ when $y=x+n-j$, we obtain 
    $$s_y\cdot(\sigma_{w}-\sigma_{s_y u})=\begin{cases}
    			\sigma_{w}-\sigma_{u} & \text{if $y\neq x+n-j$,}\\
                 \sigma_{s_{x+n-j}w}-\sigma_{u} & \text{if $y=x+ n-j$\,}.
    		 \end{cases}$$		
    The case $y\neq x+n-j$ follows immediately. If $y=x+n-j$, then the result follows by subtracting \eqref{eq1} from the second term above. Therefore, $\sigma_{w}-\sigma_{u}\in \mathbb{C}\Sn{k}(\sigma_{w'}-\sigma_{s_{x+1} w'})$.
\end{proof}

In the following, we prepare two lemmas that will be used in the proof
of Proposition~\ref{prop:partner1} for the case where $w$ is of type iii).

\begin{lemma}\label{lem:type3_step1_2}
    Let $h=h_{m,n}$ with $n>0$, and  $w\in \Gh$ be of type iii) with $x\coloneqq x_w$. For $v$ with $\ell_h(v)=\ell_h(\underline{w})$, if $v>\underline{w}$ or $v>s_{x+1}\underline{w}$, then $v\ll w$ or $v\in \{w',s_{x+1}w'\}$.
\end{lemma}
\begin{proof}
    Suppose that $v\not\ll w$. Then $j_v=j_w=0$ and $x_v\geq x_w=x
    $. If $x_v>x$, then by Lemma~\ref{lem:u}(1) we have
            $$[x]\subset v[m-1]\quad\text{ and }\quad x+1\leq x_v=v_m<v_{m+1}<v_{m+2}<\dots<v_{m+n}.$$
    On the other hand, the values of $\underline{w}$ and $s_{x+1}\underline{w}$ are given in the following table:
       \[\begin{array}{c|cccccccc}
        \hline
                               & m & m+1 & m+2 & \cdots & m+n\\
        \hline
         \underline{w}                & x+2 & x+1 & x+3 &  \cdots & x+n+1\\
        s_{x+1}\underline{w}                & x+1 & x+2 & x+3 &  \cdots & x+n+1\\
        \hline
        \end{array}\] 
    Note that $\underline{w}[m-1]=(s_{x+1}\underline{w})[m-1]=\{1,\dots,x\}\cup\{x+n+2,\dots,m+n\}$. However,
    $$v[m-1]\cap\{x+1,\dots,x+n+1\}\neq \emptyset$$ because $v\neq s_{x+1}\underline{w}$. Hence, $v[m-1]\up <\underline{w}[m-1]\up=(s_{x+1}\underline{w})[m-1]\up$, which implies that neither $v>\underline{w}$ nor $v>s_{x+1}\underline{w}$ in the Bruhat order. 
    Therefore, we only need to check when $x_v=x$.
    
    Now we assume $x_v=x$. Then, by Lemma~\ref{lem:u}(1) we have
    \begin{equation}\label{eq:v}
        [x-1]\subset v[m-1]\quad\text{ and }\quad x=v_m<v_{m+1}<v_{m+2}<\dots<v_{m+n}.
    \end{equation}
    Since both $\underline{w}_{m+n}$ and { $(s_{x+1}\underline{w})_{m+n}$} are not greater than $x+n+1$, if $v>\underline{w}$ or $v>s_{x+1}\underline{w}$, then $v_{m+n}\leq x+n+1$. By~\eqref{eq:v}, there exists a unique element $y$ such that $$\{v_{m+1},\dots,v_{m+n}\}=[x+1,x+n+1]-\{y\}.$$ 
    
    Now we divide the proof into two cases; $v>\underline{w}$ or $v>s_{x+1}\underline{w}$. In each case, we will show the following:
    \begin{itemize}[topsep=2pt, itemsep=0pt, parsep=0pt]
        \item If $v>\underline{w}$, then $y=x+2$ and $v=w'$.
        \item If $v>s_{x+1}\underline{w}$, then either $y=x+1$ and $v=s_{x+1}w'$, or $y=x+2$ and $v=w'$.
    \end{itemize}
    In each case, we prove the statement by dividing it into three subcases; $y=x+1$, $y=x+2$, or $y\geq x+3$. 

    We first consider the case $v>\underline{w}$.
    \textbf{(a)} If $y=x+1$, then $v[m]\up < \underline{w}[m]\up$, which is a contradiction.
    \textbf{(b)} If $y=x+2$, then $v_m\,v_{m+1}\,\dots\,v_{m+n}=w'_m\,w'_{m+1}\,\dots\,w'_{m+n}$. Since $\ell_h(v)= \ell_h(\underline{w})=\ell_h(w')$ by Lemma~\ref{lem:partner}(2), the number of inversions in $v_1\,v_2\,\dots\,v_m$ equals that of inversions in $w'_1\,w'_2\,\dots\,w'_m$. Accordingly, $\ell({ v})=\ell(w')$. By Lemma~\ref{lem:partner}(2), we have $\ell(w')=\ell(\underline{w})+1$. Hence $\ell(v)=\ell(\underline{w})+1$ while $\underline{w} < v$. We can conclude that $v=s_{i,j}\underline{w}$ for a transposition $s_{i,j}$, and the choice $(i,j)$ is uniquely determined; $(i, j)=(x, x+2)$ by comparing $v_m$ and $\underline{w}_m$.  This shows that if $y=x+2$, then $v=w'$.
    \textbf{(c)} If $y\geq x+3$, then the values of $v$ and $\underline{w}$ at positions $m, m+1,\dots,m+n$ are given as follows.
        \[\begin{array}{c|cccccccc}
\hline
                    & m & m+1 & m+2 & \cdots & i & i+1 & \cdots & m+n\\
\hline
  v                 & x & x+1 & x+2 & \cdots & y-1 & y+1 & \cdots & x+n+1\\
 \underline{w} & x+2 & x+1 & x+3 & \cdots & y & y+1 & \cdots & x+n+1\\
\hline
\end{array}\]
Since $\ell_h(v)=\ell_h(w')$, $v_m<v_{m+1}<\cdots<v_{m+n}$, and $w'_m<w'_{m+1}<\cdots<w'_{m+n}$, the number of inversions in $v_1\, v_2\, \dots\, v_m$ equals that of inversions in $w'_1\, w'_2\, \dots \, w'_m$. Hence we consequently have $\ell(v)=\ell(w')$. This implies that $\ell(v)=\ell(w')=\ell(\underline{w})+1$ while $\underline{w}< v$. Therefore there exists a transposition $s_{i,j}$ satisfying $v=s_{i,j}\underline{w}$, which however can not hold as we can check from the table of values of $v$ and $\underline{w}$.

Now we consider the case $v>s_{x+1}\underline{w}$.
    \textbf{(a)} If $y=x+1$, then the values for $v$ and $s_{x+1}w'$ at positions ${ m,} m+1, \dotsm m+n$ are the same. Note that $s_{x+1}\underline{w} < s_{x+1}w'$ and $s_{x+1}\underline{w} < v$. Moreover by Lemma~\ref{lem:partner}(2), $\ell(s_{x+1}\underline{w})=\ell(\underline{w})-1=\ell(w')-2= \ell(s_{x+1}w')-1$, while $\ell_h(v)=\ell_h(\underline{w}) = \ell_h(w')=\ell_h(s_{x+1}w')$.
    This means that the number of inversions in $v_1\, v_2\, \dots\, v_m$ equals that of inversions in $(s_{x+1}w')_1\, (s_{x+1}w')_2\, \dots \, (s_{x+1}w')_m$. Hence we have $\ell(v)=\ell(s_{x+1}w')=\ell(s_{x+1}\underline{w})+1$. 
    We can conclude that $v=s_{i, j}(s_{x+1}\underline{w})$ for a transposition $s_{i, j}$, and $(i,j)$ is uniquely determined; $(i, j)=(x, x+1)$ by comparing $v_m=(s_{x+1}w')_m=x$ and $(s_{x+1}\underline{w})_m=x+1$.  Since $\underline{w}=s_{x, x+2}w'$, we obtain $v=s_xs_{x+1}\underline{w}=s_{x+1}w'$.
    \textbf{(b)} If $y=x+2$, then 
    $v_m\, v_{m+1}\, \dots\, v_{m+n}=w'_m\, w'_{m+1} \cdots w'_{m+n}$. Since $\ell_h(v)=\ell_h(w')$, the numbers of inversions in $v_1\, v_2\, \dots\, v_m$ equals that of inversions in $w'_1\, w'_2\, \dots \, w'_m$, and we consequently have $\ell(v)=\ell(w')$.
    By Lemma~\ref{lem:partner}(2), we have $\ell(w')=\ell(s_{x+1}\underline{w})+2$, hence $\ell(v)=\ell(\underline{w})+2$ while $s_{x+1}\underline{w}< v$.
    We can conclude that $v=s_{x+1}s_{x}(s_{x+1}\underline{w})$  by comparing $(v_m, v_{m+1})$ with $((s_{x+1}\underline{w})_m, (s_{x+1}\underline{w})_{m+1})$.  Hence we obtain $v=s_{x+1}s_{x}s_{x+1}\underline{w}=w'$.
    \textbf{(c)} If $y\geq x+3$, then the values of $v$ and $s_{x+1}\underline{w}$ at positions $m, m+1, \dots, m+n$ are given as follows. 
\[\begin{array}{c|cccccccc}
\hline
                    & m & m+1 & m+2 &\cdots & i & i+1 & \cdots & m+n\\
\hline
  v                 & x & x+1  &x+2 & \cdots & y-1 & y+1 & \cdots & x+n+1\\
 s_{x+1}\underline{w} & x+1 & x+2 & x+3 & \cdots & y & y+1 & \cdots & x+n+1\\
\hline
\end{array}\]
Since $\ell_h(v)=\ell_h(w')$, $v_m<v_{m+1}<\cdots<v_{m+n}$, and $w'_m<w'_{m+1}<\cdots<w'_{m+n}$,
the numbers of inversions in $v_1\, v_2\, \dots\, v_m$ equals that of inversions in $w'_1\, w'_2\, \dots \, w'_m$, and we consequently have $\ell(v)=\ell(w')$.
By Lemma~\ref{lem:partner}(2), we have $\ell(w')=\ell(s_{x+1}\underline{w})+2$, and $\ell(v)=\ell(s_{x+1}\underline{w})+2$ while $s_{x+1}\underline{w}\leq v$.
We can conclude that $(s_{x+1}\underline{w})v^{-1}$ is a permutation of length $2$, which cannot hold as we can check from the table of values of $v$ and $s_{x+1}\underline{w}$.

   This completes the proof of the lemma.
\end{proof}

    \begin{lemma}\label{lem:sigma of partner} Let $h=h_{m,n}$ with $n>0$,
and $w\in \Gh$ be of type iii). Let $x\coloneqq x_w=x_{w'}$ and $p$ be the index with $\underline{w}_p=x$. Then $\sigma_{\underline{w}}^T(s_{x+1}w')=0$. Let $$S\coloneqq \{(i,j)\,|\, i<j\leq m, w'_i>w'_j,  (i,j)\neq (p, m) \}.$$
\begin{enumerate}[topsep=2pt, itemsep=0pt, parsep=0pt]
    \item The value $\sigma_{\underline{w}}^T(w')$ is given by \begin{align*}
    \sigma_{\underline{w}}^T(w')&= (t_{x+1}-t_{x})\prod_{(i,j)\in S} (t_{w'_j}-t_{w'_i}).
    \end{align*}
    \item When $v$ is either $w'$ or $s_{x+1}w'$, the value $\sigma_{s_{x+1}\underline{w}}(v)$ is given by 
    \begin{align*}
        \sigma_{s_{x+1}\underline{w}}^T(v)&= \prod_{(i,j)\in S} (t_{v_j}-t_{v_i}).
    \end{align*}
\end{enumerate}
Consequently, we get
\begin{equation*}
    (s_{x+1}\cdot\sigma_{\underline{w}}^T)(w')=0\text{ and } 
    (s_{x+1}\cdot\sigma_{\underline{w}}^T)(s_{x+1}w')=(t_{x+2}-t_{x})\prod_{(i,j)\in S} (t_{(s_{x+1}w')_j}-t_{(s_{x+1}w')_i}).
\end{equation*}
\end{lemma}
\begin{proof}
For convenience, we make a table of values of permutations we are interested in, where   $\widetilde{\underline{w}}$ is the $h$-admissible permutation such that $ \underline{w} \in P(\widetilde{\underline{w}})$, $y= (\widetilde{\underline{w}})_m$ and $\phi= \widetilde{\underline{w}} \,\underline{w}^{-1}$.
\[\begin{array}{c|ccccccccc}
\hline
                              &  p & m & m+1 & m+2 & \cdots & m+n\\
\hline
         w'              &   x+2  & x & x+1 & x+3 &  \cdots & x+n+1\\
  s_{x+1}w'              &   x+1  & x & x+2 & x+3 &  \cdots & x+n+1\\
 \underline{w}          &   x  & x+2 & x+1 & x+3 & \cdots & x+n+1\\
 \widetilde{\underline{w}}=\phi\underline{w} &y-1  &y  & 1  & 2 & \cdots & n\\
 \phi w'   &y  &y-1  & 1  & 2 & \cdots & n\\
\hline
\end{array}\]

\smallskip

Since $s_{x+1}w'\not\geq \underline{w}$, we get $\sigma_{\underline{w}}^T(s_{x+1}w')=0$ by Proposition~\ref{prop:computing the value}(1). In the following, we will use Proposition~\ref{prop:computing the value}(4) to compute the $\sigma_{\underline{w}}^T(w')$, $\sigma_{s_{x+1}\underline{w}}^T(w')$, and $\sigma_{s_{x+1}\underline{w}}^T(s_{x+1}w')$. Note that $s_{x+1}\underline{w}$ is $h$-admissible, but $\underline{w}$ is not. 

\begin{description}[leftmargin=0pt, labelsep=0.5em, itemindent=0pt, itemsep=0.5em]
\item[(1)] 
Let  $\widetilde{\underline{w}}\in\Gh$ be the $h$-admissible permutation such that $\underline{w} \in  P(\widetilde{\underline{w}})$. Then  $\phi\left(\supp(\sigma^T_{\underline{w}})\right)= [\widetilde{\underline{w}}, w_0]$  where $\phi= \widetilde{\underline{w}} \,\underline{w}^{-1}$. Moreover, $\Gamma_{\underline{w}, h}$ is isomorphic to $\Gamma_{\widetilde{\underline{w}}, h}$ through $\phi$, where $\alpha_{\underline{w}}(\{u, v \})=\alpha_{\widetilde{\underline{w}}}(\{\phi(u), \phi(v) \})$ for $\{u, v \}\in E_{\underline{w}}$.
 Note that $w'=s_{x, x+2}\underline{w}= \underline{w} s_{p,m}>_h \underline{w}$. Furthermore, the degree of $w'$ in  $\Gamma_{\underline{w}, h}$ is the same as the degree of $\underline{w}$, which is $\ell_h(\underline{w})$, because for any $i<j\leq h(i)$, $(i,j)\neq (p,m)$, we have $(\widetilde{\underline{w}})_i<(\widetilde{\underline{w}})_j$ if and only if $(\phi w')_i<(\phi w')_j$. Applying Proposition~\ref{prop:computing the value}(4), we compute the value $\sigma_{\underline{w}}^T(w')$:
$$\sigma_{\underline{w}}^T(w')=\phi^{-1}\left( \sigma_{\widetilde{\underline{w}}}^T(\phi(w'))\right)=\phi^{-1}\left[(t_1-t_{y-1})\prod_{(i,j)}\left(t_{(\phi w')_j} -t_{(\phi w')_i} \right)\right]\,,$$
where the product is over the $(i,j)\neq(p,m)$ such that $i<j\leq m$ and $(\phi w')_i>(\phi w')_j$. This shows (1).
\item[(2)] 
We first note that $s_{x+1}\underline{w}$, $s_{x+1}w'$, and $w'$ differ only at positions $p,m,m+1$, and that none of them has a descent at any of the positions $m,m+1,\dots,m+n-1$. Since $\{s_{x+1}w',s_{x+1}w's_{p,m}\}$ and $\{w',w's_{p,m}\}$ belong to  $E_{s_{x+1}\underline{w}}$, the following statements are equivalent for $u\in\{s_{x+1}\underline{w}, s_{x+1}w'\}$:
\begin{itemize}
    \item[-] $\{u,us_{i,j}\}\in E-E_{s_{x+1}\underline{w}}$.
    \item[-] $\{w',w's_{i,j}\}\in E-E_{s_{x+1}\underline{w}}$.
    \item[-] $i<j\leq m$, $(i,j)\neq (p,m)$, and $w'_i>w'_j$.
\end{itemize} Therefore, $s_{x+1}\underline{w}$, $s_{x+1}w'$, and $w'$ have the same degree in $\Gamma_{s_{x+1}\underline{w}}$. Applying Proposition~\ref{prop:computing the value}(4) to compute the values $\sigma_{s_{x+1}\underline{w}}^T(s_{x+1}w')$ and $\sigma_{s_{x+1}\underline{w}}^T(w')$, we get (2).
\end{description}

The last part follows from the definition of the dot action. This completes the proof.
\end{proof}

The next proposition gives a uniform formula, but its proof naturally splits into two cases according to whether $w$ is of type ii) or type iii). The proof for the case of type iii) relies on the two lemmas above.
\begin{proposition}\label{prop:partner1}
    Let $h=h_{m,n}$ with $n>0$, and  $w\in \Gh$ be an $h$-admissible permutation of type ii) or iii). Let $x\coloneqq x_w$. Then 
    $$\displaystyle s_{x+1} \cdot \sigma_{\underline{w}}=\sigma_{\underline{w}}+(\sigma_{w'}-\sigma_{s_{x+1} w'})+\sum_{v\ll w} c_v \sigma_v.$$
\end{proposition}
\begin{proof}
    We first consider when $w\in\Gh$ is a permutation of type ii) with $(j_w,x_w)=(j,1)$.
    Since $\underline{w}$ admits $(m+j-2)$-block decomposition; { see Table~\ref{tab:permutation-values},} if we apply Theorem~\ref{thm:separation} then  
    \[s_{2} \cdot \sigma_{\underline{w}}=(\sigma_{^{\da{m+j-2}}\underline{w}}:(s_2\cdot\sigma_{\underline{w}^{\da{m+j-2}}})).
    \] 
    Noting that $\underline{w}^{\da{m+j-2}}=s_2\in \Sn{n-j+2}$, we apply Proposition~\ref{prop:s_i}  with $h^{\da{m+j-2}}=(2, 3, \dots, n-j+2, n-j+2)$ to obtain 
    $$s_2\cdot\sigma_{s_2,h^{\da{m+j-2}}}=\sigma_{s_2,h^{\da{m+j-2}} }+\sum_{u\in \mathcal{A}_2} \left (\sigma_{u,h^{\da{m+j-2}}}-\sigma_{s_2u,h^{\da{m+j-2}}}\right)\,,$$
    where $\mathcal{A}_2=\{u \in \Sn{n-j+2}\,|\, s_2\leq u, \ell_{h^{\da{m+j-2}}}(u)=1, u\dasharrow s_2u \}$.

    Note that $u=3\,1\,2\,4\cdots\, (n-j+2)\in \mathcal{A}_2$ and \[(^{\da{m+j-2}}\underline{w}: u)=w',\qquad (^{\da{m+j-2}}\underline{w}: s_2u)=s_2w'.\] Hence, $u$ gives the term $\sigma_{w'}-\sigma_{s_2w'}$. It remains to show that the other permutations in \(\mathcal{A}_2\) correspond to permutations \(v\) satisfying \(v \ll w\).
    We thus consider the permutations $u\in \mathcal{A}_2$ such that  $u\neq 3\,1\,2\,4\cdots\, (n-j+2)$. Since $\ell_{h^{\da{m+j-2}}}(u)=1$, $u$ has one descent; that is, there is $a>0$ such that $u_1<u_2<\cdots <u_a$, $u_{a+1}<u_{a+2}<\cdots <u_{n-j+2}$ and $u_a>u_{a+1}$. Moreover, since $u\dasharrow s_2u$, $3\in \{u_1, \dots, u_a\}$ and $2\in \{u_{a+1},\dots, u_{n-j+2} \}$. There are two possibilities: (a) $u_1=1, u_{a+1}=2$, and (b) $u_{a+1}=1, u_{a+2}=2$. In case (a), $u_2=3$ and $a+1>3$. In case (b),  $a+1 > 2$ because we consider $u\neq 3\,1\,2\,4\cdots\, n-j+2$. Hence, in any case, the permutation $v=({}^{\da m+j-2}\underline{w},u)$ has $j_v=\{(m+j-2)+(a+1)\}-m>j$ and $x_v=1$, so $v\ll w$. Since $\widetilde{v}=\widetilde{s_2v}$, we have $(j_{s_2v}, x_{s_2v})=(j_v,x_v)$, so $s_2v\ll w$. This completes the proof when $w$ is of type ii).
 
    Now we consider when $w$ is of type iii) with $(j_w,x_w)=(0,x)$. Then  \( \underline{w} \ra s_{x+1}\underline{w} \) and \( s_{x+1}\underline{w}\in \Gh \).   By Proposition~\ref{prop:simple action}(3), we have 
\[s_{x+1} \cdot \sigma^T_{\underline{w}}=\sigma^T_{\underline{w}}+(t_{x+2}-t_{x+1}) \sigma^T_{s_{x+1}\underline{w}}+\sum_{u\in \mathcal A_{s_{x+1}, \underline{w}}} (\tau_{u}-\tau_{s_{x+1}u})\,,\] where
$ \mathcal A_{s_{x+1}, \underline{w}} \subset 
 \widetilde{\mathcal A}_{s_{x+1}, \underline{w}} =\{u\in [\underline{w}, w_0] \,|\, u \dra s_{x+1} u,\, \lh{u}\leq \lh{\underline{w}} \}$ by Lemma~\ref{lem:Atil}.

{ If $\sigma_v^T$ appears in $\sum_{u\in \mathcal A_{s_{x+1}, \underline{w}}}(\tau_{u}-\tau_{s_{x+1}u})$, then $\ell_h(v)=\ell_h(\underline{w})$ and by Lemma~\ref{lem:tau2} there exists $u\in \widetilde{\mathcal A}_{s_{x+1}, \underline{w}}$ such that $u\leq_h v$ or $s_{x+1}u\leq_h v$.} Since $u>\underline{w}$ and $s_{x+1}u>s_{x+1}\underline{w}$ for each $u\in \mathcal{A}_{s_{x+1},\underline{w}}$  by Lemma~\ref{lem:partner}(4), we get $v>\underline{w}$ or $v>s_{x+1}\underline{w}$. Hence, $v\ll w$ or $v\in \{w',s_{x+1}w'\}$ by Lemma~\ref{lem:type3_step1_2}. We therefore get
\[s_{x+1} \cdot \sigma^T_{\underline{w}}=\sigma^T_{\underline{w}}+(t_{x+2}-t_{x+1}) \sigma^T_{s_{x+1}\underline{w}}+a\sigma_{w'}^T+b\sigma_{s_{x+1}w'}^T+\sum_{v\ll w} c_v\sigma_v^T\,.\] Hence it suffices to prove that $a=1$ and $b=-1$. To this end, we prove the following two assertions. First, we show that for $\sigma_v$ appearing in
\(
\sum_{u\in \mathcal A_{s_{x+1}, \underline{w}}}
(\tau_u-\tau_{s_{x+1}u}),
\) if \(
w'\in \operatorname{supp}(\sigma_v^T) \) or \(
s_{x+1}w'\in \operatorname{supp}(\sigma_v^T)
\), then $v \in \{w',s_{x+1}w'\}$. Second, we show that the values of \[
s_{x+1}\cdot \sigma_{\underline{w}}^T
-\sigma_{\underline{w}}^T
-(t_{x+2}-t_{x+1})\sigma_{s_{x+1}\underline{w}}^T
\]
at $w'$ and $s_{x+1}w'$ are exactly
\[
\sigma_{w'}^T(w')
\quad\text{and}\quad
-\sigma_{s_{x+1}w'}^T(s_{x+1}w'),
\]
respectively. The first assertion implies $\sum_{v\ll w} c_v\sigma_v^T(w')=\sum_{v\ll w} c_v\sigma_v^T(s_{x+1}w')=0$. Combining this with the second assertion, we obtain $a=1$ and $b=-1$.
    
We begin by proving the first assertion. Suppose that $\sigma_v$ appears in \(
\sum_{u\in \mathcal A_{s_{x+1}, \underline{w}}}
(\tau_u-\tau_{s_{x+1}u}).
\) Then by Lemmas ~\ref{lem:Atil} and ~\ref{lem:tau2}, there exists a permutation $u\in\mathcal{A}_{s_{x+1},\underline{w}}\subset \widetilde{\mathcal{A}}_{s_{x+1},\underline{w}}$ such that either $u\leq_h v$ or $s_{x+1}u\leq_h v$. Note that $\underline{w}<u$ and $s_{x+1}\underline{w}<s_{x+1}u$ by Lemma~\ref{lem:partner}(4). We now prove the assertion by considering the four possible combinations of the following two alternatives: \[ u\leq_h v \quad\text{or}\quad s_{x+1}u\leq_h v, \] and \[ w'\in \operatorname{supp}(\sigma_v^T) \quad\text{or}\quad s_{x+1}w'\in \operatorname{supp}(\sigma_v^T). \]
(i) Assume $u\leq_h v$ and $w'\in \operatorname{supp}(\sigma_v^T)$. Then $\underline{w}<u\leq_h v$ implies $\ell(v)>\ell(\underline{w})$. Since {$w'\geq v$}, from Lemma~\ref{lem:partner}(2), we get $\ell(w')\geq \ell(v)>\ell(\underline{w})=\ell(w')-1$, which implies that $\ell(v)=\ell(w')$. Accordingly, $v=w'$. (ii) Assume $u\leq_h v$ and $s_{x+1}w'\in \operatorname{supp}(\sigma_v^T)$. Then $s_{x+1}w'\geq v$, and we have $\ell(s_{x+1}w')\geq \ell(v)>\ell(w')-1$, which contradicts $\ell(s_{x+1}w')=\ell(w')-1$. (iii) Assume $s_{x+1}u\leq_h v$ and $s_{x+1}{  w'}\in \operatorname{supp}(\sigma_v^T)$. Then we get
$$\ell(w')-1=\ell(s_{x+1}w')\geq \ell(v)>\ell(s_{x+1}\underline{w})=\ell(\underline{w})-1=\ell(w')-2,$$
which implies $\ell(s_{x+1}w')=\ell(v)$. Accordingly, $v=s_{x+1}w'$. (iv) Assume $s_{x+1}u\leq_h v$ and $w'\in \operatorname{supp}(\sigma_v^T)$. Then we have $s_{x+1}\underline{w}<s_{x+1}u\leq_h v \leq w'$. Since $x+1\in D_L(w')-D_L(s_{x+1}u)$, we have $u\leq w'$ by Proposition~\ref{prop:lifting}. Therefore, 
$$\ell(\underline{w})=\ell(s_{x+1}\underline{w})+1<\ell(s_{x+1}u)+1=\ell(u)\leq \ell(w')=\ell(\underline{w})+1.$$
Accordingly, $u=w'$, which implies $s_{x+1}w'\leq_h v\leq w'$. Therefore, $v\in\{w',s_{x+1}w'\}$.

It remains to prove the second assertion. We will use Lemma~\ref{lem:sigma of partner}. Let $S\coloneqq \{(i,j)\mid i<j\leq m, w'_i>w'_j,  (i, j)\neq (p,m) \}$. We first evaluate $s_{x+1}\cdot \sigma_{\underline{w}}^T
-\sigma_{\underline{w}}^T
-(t_{x+2}-t_{x+1})\sigma_{s_{x+1}\underline{w}}^T$ at $w'$: 
\begin{align*}
&(s_{x+1} \cdot \sigma^T_{\underline{w}})(w')-\sigma^T_{\underline{w}}(w')-(t_{x+2}-t_{x+1}) \sigma^T_{s_{x+1}\underline{w}}(w')\\
 &= 0-(t_{x+1}-t_{x})\prod_{(i,j)\in S} (t_{w'_j}-t_{w'_i})-(t_{x+2}-t_{x+1}) \prod_{(i,j)\in S} (t_{w'_j}-t_{w'_i})\\
 &=(t_{x}-t_{x+2})\prod_{(i,j)\in S} (t_{w'_j}-t_{w'_i})\\
 &=\sigma_{w'}^T(w'),
\end{align*}
where the last equality follows from Proposition~\ref{prop:computing the value}(2). Now we evaluate $s_{x+1}\cdot \sigma_{\underline{w}}^T
-\sigma_{\underline{w}}^T
-(t_{x+2}-t_{x+1})\sigma_{s_{x+1}\underline{w}}^T$ at $s_{x+1}w'$:
 \begin{align*} &(s_{x+1} \cdot \sigma^T_{\underline{w}})(s_{x+1}w')-\sigma^T_{\underline{w}}(s_{x+1}w')-(t_{x+2}-t_{x+1}) \sigma^T_{s_{x+1}\underline{w}}(s_{x+1}w')\\
 &= (t_{x+2}-t_{x})\prod_{(i,j)\in S} (t_{(s_{x+1}w')_j}-t_{(s_{x+1}w')_i})-0-(t_{x+2}-t_{x+1}) \prod_{(i,j)\in S} (t_{(s_{x+1}w')_j}-t_{(s_{x+1}w')_i}) \\
 &=(t_{x+1}-t_{x})\prod_{(i,j)\in S} (t_{(s_{x+1}w')_j}-t_{(s_{x+1}w')_i})\\
 &=-\sigma_{s_{x+1}w'}^T(s_{x+1}w'),
 \end{align*}
where the last equality follows from Proposition~\ref{prop:computing the value}(2).
\end{proof}

From the work we have done to understand the class $\sigma_{\underline{w}}$, we can prove an important result that will be used in proving Theorem~\ref{thm:main} in the final section.

\begin{proposition}\label{lem1:main}
    Let $w\in\Gh^d$ be an $h$-admissible permutation of type ii) or iii). If $\sigma_v\in M_{\ll w}$ for every $v\ll w$ and
    \begin{equation}\label{eq1:main}
        \widehat{\sigma}_w\equiv \sum_{u\in P(w)}c_u\sigma_u\pmod{M_{\ll w}} \text{ for }c_u>0,
    \end{equation} then $\displaystyle\sigma_w\in \sum_{u\in S_w}\mathbb{C}\Sn{k}(\widehat{\sigma}_u)$, where $S_w=\{w\}\cup\{u\in \Gh^d \mid u\ll w\}$. Accordingly, $\sigma_w\in  M$.
\end{proposition}

\begin{proof}
We assume $w$ is of type ii) or iii). By Proposition~\ref{prop:partner1}, we have:
 $$ s_{x+1} \cdot \sigma_{\underline{w}}=\sigma_{\underline{w}}+(\sigma_{w'}-\sigma_{s_{x+1} w'})+\sum_{v\ll w}c_v\sigma_{v}\,,$$
from which we have, { by the assumption of the lemma,}
$$\sigma_{w'}-\sigma_{s_{x+1} w'}\equiv s_{x+1} \cdot \sigma_{\underline{w}}-\sigma_{\underline{w}}\pmod{M_{\ll w}}\,. $$ 
Note that both $w'$ and $s_{x+1} w'$ are in $P(w)$ by Lemma~\ref{lem:partner}(1). We also note that both $\sigma_{\underline{w}}$ and $s_{x+1}\cdot\sigma_{\underline{w}}$ belong to $M_{\ll w}$ because $\underline{w}\ll w$ by Lemma~\ref{lem:partner}(6). Therefore
$$\sigma_{w'}-\sigma_{s_{x+1}{ w'}}\in M_{\ll w}$$
Due to Proposition~\ref{prop:P(w)_classes}, by letting $\Sn{k}$ act on $\sigma_{w'}-\sigma_{s_{x+1} w'}$, we can obtain $|P(w)|-1$ linearly independent classes $(\sigma_{u}-\sigma_{v})\in M_{\ll w}$ for $u, v\in P(w)$.
In addition, \eqref{eq1:main} tells us that $$\sum_{u\in P(w)} c_u \sigma_u \in \sum_{u\in S_{w}}\mathbb{C}\Sn{k}(\widehat{\sigma}_u) \subset  M,$$ where $c_u$ are all positive. Therefore, for any $u\in P(w)$, the class $\sigma_u$ can be obtained as a linear combination of elements in $\sum_{u\in S_{w}}\mathbb{C}\Sn{k}(\widehat{\sigma}_u)$.
\end{proof}

\section{Proof of Theorem~\ref{thm:main}}\label{sec:decomposition}

In this section, we prove Theorem~\ref{thm:main}. For lollipop Hessenberg functions, the elements
\(\widehat{\sigma}_{w}\) for $w\in\Gh^d$ make the desired permutation
module decomposition of \(H^{2d}(\Hess(S,h))\).

We let $h=h_{m, n}$ and  $k=m+n$. If $n=0$ then $\Gh=\Sn{k}$, $\widehat{\sigma}_u=k!\sigma_u$ and $\mathbb C\Sn{k}(\widehat{\sigma}_u)\cong M^{(k)}$ for each $u\in \Sn{k}$ by Remark~\ref{rmk:generator_flag} and Corollary~\ref{cor:dot action on flag}. Therefore, 
$$H^{2d}(\Hess(S,h))=\bigoplus_{u\in \Gh^d}\mathbb C\Sn{k}(\widehat{\sigma}_u)\cong \bigoplus_{u\in \Gh^d} M^{\bb(u)}=\bigoplus_{u\in \Gh^d} M^{(k)}$$ holds and we assume that $n>0$ throughout the section.

We recall from Definition~\ref{def:M} that, for each $0\leq d\leq d_h$ and $ w\in\Gh^d$,
\[M\coloneqq \sum_{u\in \Gh^d}\mathbb C\Sn{k}(\widehat{\sigma}_u) \quad\text{and}\quad M_{\ll w}\coloneqq \sum_{u\in \Gh^d, u\ll w}\mathbb{C} \Sn{k}(\widehat{\sigma}_u) \,,
\]
and $M_{\ll w}\subset M$ are $\mathbb C\Sn{k}$ submodules of  $H^{2d}(\Hess(S,h))$.

\begin{proposition}\label{prop:M}
If $H^{2d}(\Hess(S,h))\subset M$; that is $H^{2d}(\Hess(S,h))= M$, then 
\[ H^{2d}(\Hess(S,h))=\bigoplus_{u\in \Gh^d}\mathbb C\Sn{k}(\widehat{\sigma}_u)\cong \bigoplus_{u\in \Gh^d} M^{\bb(u)}\,.\]
In particular, for each $u\in\Gh^d$, $\Stab(\widehat{\sigma}_u)=\mathcal{S}_u$ and $\mathbb C\Sn{k}(\widehat{\sigma}_u)\cong M^{\bb(u)}$.
\end{proposition}
\begin{proof}
For each $u\in\Gh^d$, 
$\mathcal{S}_u$ is a subgroup of $\Stab(\widehat{\sigma}_u)$ by the definition of $\widehat{\sigma}_u$. Hence
$$ \dim_\mathbb{C} M \leq \sum_{u\in\Gh^d} |\Sn{k}/\Stab(\widehat{\sigma}_u)|\leq  \sum_{u\in\Gh^d} |\Sn{k}/\mathcal{S}_u|=\sum_{u\in\Gh^d} \dim_\mathbb{C} M^{\bb(u)}=  \dim_\mathbb{C}(H^{2d}(\Hess(S,h)))\,,$$
where the last equality is valid from
Theorem~\ref{thm:Acyclic}. Since $H^*(\Hess(S,h)))=M$ by the assumption, we can conclude that all the inequalities are actually the equalities, and $M$ is a direct sum of the $\mathbb C\Sn{k}(\widehat{\sigma}_u)$'s.
\end{proof}

By Proposition~\ref{prop:M}, Theorem~\ref{thm:main} follows once we show that \(H^{2d}(\Hess(S,h))=M\). To this end, we show that \(\sigma_w\in M\) for every \(w\in\Sn{k}\) with \(\ell_h(w)=d\).

Before we start to prove Theorem~\ref{thm:main}, we introduce some notation. 
\begin{definition} Let $h=h_{m, n}$ and $k=m+n$.
 For $w\in \Gh$, if $\bb(w)=(b_1, \dots, b_\ell)$
  then we use $^\da w$ and $w^\da$ for $^{\da i}w$ and $w^{\da i}$, respectively, where $i\coloneqq b_1+\cdots+b_{\ell-1}$. In particular, if $\ell=1$ then $^\da w$ is the empty word and $w^\da=w$. 
\end{definition}

Now we are ready to prove Theorem~\ref{thm:main}.

\smallskip

\noindent\underline{\bf Proof of Theorem~\ref{thm:main}}

We first note that, due to Proposition~\ref{prop:representative} and Proposition~\ref{prop:M}, it is enough to show the following statement:
\begin{quote}
    $(*)$\hspace{2cm} $\sigma_w\in M=\sum_{u\in \Gh^d}\mathbb C\Sn{k}(\widehat{\sigma}_u)$ for every $w\in \Gh^d$.
\end{quote}
We use the induction on $k$ to prove $(*)$.

\smallskip
When $k=2$, the only possible case is when $h=h_{1,1}=(2,2)$. In this case, $\widehat{\sigma}_w=2\sigma_w$ for all $w\in\Sn{k}$; therefore $(*)$ holds.

\smallskip
Let $k>2$ and we assume that $(*)$ is  true for all $2\leq k'<k$ and $h_{m',n'}$ with $m'+n'=k'$, $m'\geq 1$, and $n'\geq 0$.
Let $m\geq 1, n\geq 0$ be  integers such that $m+n=k$ and we consider the Hessenberg function $h=h_{m,n}$.
Let $w\in \Gh$ and let $\bb(w)=(b_1, \dots, b_\ell)$. Then $b_1\geq m$ by the definition of $\bb(w)$. 
We use the induction with respect to the partial order $\ll$ on $\Sn{k}$ given in Definition~\ref{def:order}. We recall that $\ll$ is given in terms of the pair $(j_w,x_w)$.

\smallskip

\noindent\underline{Initial Step:} Let $w\in\Gh$ be an $h$-admissible permutation with $(j_w,x_w)=(n,1)$.

If $\bb(w)=(k)$ then $w$ is of type i) and $\sigma_w=c\widehat{\sigma}_w$ for a positive constant $c$ by Lemma~\ref{lem:stabilizer} (1).

If $\bb(w)=(b_1, \dots, b_\ell)\neq (k)$, then $0<b_\ell\coloneqq k-i \leq n$. Hence $w^{\da } \in \mathcal{G}_{h^{\da}}$ is an $h^{\da }$-admissible permutation of type i) in $\Sn{k-i}$, where $h^{\da }=h_{1,k-i-1}=(2,3,\dots,k-i,k-i)$. Thus  $\widehat{\sigma}_{w^{\da }}=(k-i)!\sigma_{w^{\da }}$. Then from the induction hypothesis on $k$ and Theorem~\ref{thm:separation}, we have
\begin{equation}\label{eq:type1}
    \sigma_w\in \sum_{u\in \mathcal{G}_{^{\da} h}} \mathbb C\Sn{\{k-i+1,\dots,k \}}(\widehat\sigma_{u}:\sigma_{w^{\da }}),
\end{equation}
 and 
$$(\widehat\sigma_{u}:\sigma_{w^{\da }})=\sum_{v\in \mathcal{S}_u} \left((v\cdot\sigma_u):\sigma_{w^{\da }}\right)=\sum_{v\in \mathcal{S}_u}v\cdot(\sigma_u : \sigma_{w^{\da}})=\sum_{v\in \mathcal{S}_u} v\cdot\sigma_{(u:w^{\da })} $$
for each $u\in \mathcal{G}_{^{\da} h}\,$.
Here, we note that $(u : w^{\da})\in\Gh$ by Lemma~\ref{lem:block_generator}.
For  $u\in \mathcal{G}_{^{\da} h}$, if $u^\da$ is of type i), then $(u^\da:w^{\da })$ is also of type i), and then  we have 
$$\mathcal{S}_{(u:w^{\da })}=\mathcal{S}_{(^\da u: (u^\da: w^{\da }))}=\mathcal{S}_{^\da u}\times \mathcal{S}_{(u^\da: w^{\da })}=\mathcal{S}_{^\da u}\times \Sn{k'}.$$
Since $\mathcal{S}_u=\mathcal{S}_{^\da u}\times \mathcal{S}_{u^\da}$ and $ u^\da$ is of type i), if $u^\da$ is a permutation of $i'$ integers,
\begin{align*}
(\widehat{\sigma}_u:\sigma_{w^{\da}})&= \sum_{(v',v'')\in \mathcal{S}_{^\da u}\times \mathcal{S}_{u^\da}} \left(((v'\cdot\sigma_{^\da u}):(v''\cdot\sigma_{u^\da})):\sigma_{w^{\da }}\right)\\
&=(i')!\sum_{v'\in \mathcal{S}_{^\da u}} \left((v'\cdot\sigma_{^\da u}):\sigma_{(u^\da:w^{\da })}\right) \\
&=\frac{(i')!}{k'!}(\widehat{\sigma}_{^{\da}u}:\widehat{\sigma}_{(u^{\da}:w^{\da})})\\
&=\frac{(i')!}{(k')!}\widehat\sigma_{(u:w^{\da })}
\end{align*} 
by Lemma~\ref{lem:stabilizer}(1).
When $u^\da$ is of type ii), iii) or iv), since $\mathcal{S}_{(u:w^{\da })}=\mathcal{S}_u\times\Sn{k-i}$ and $w^{\da i}\in\mathcal{G}_{h^{\da i}}$ is of type i),  we can conclude that $$(\widehat{\sigma}_u:\sigma_{w^{\da}})=\sum_{v'\in \mathcal{S}_u} v'\cdot\sigma_{(u:w^{\da })}=\frac{1}{(k-i)!}\widehat\sigma_{(u:w^{\da })}.$$ Therefore, $(\widehat\sigma_{u}:\sigma_{w^{\da }})$ is a (nonzero) constant multiple of $\widehat\sigma_{(u:w^{\da })}$ for each $u\in \mathcal{G}_{^{\da} h}$. Then~\eqref{eq:type1} implies that 
$$ \sigma_w\in \sum_{u\in \mathcal{G}_{^{\da{}}h}} \mathbb C\Sn{\{k-i+1, \dots,n \}}(\widehat\sigma_{(u:w^{\da })})\subset \sum_{v\in\Gh,\, j_v=n}\mathbb{C}\Sn{k}(\widehat{\sigma}_v)\subset M\,.$$
Therefore, $(*)$ holds for $w\in\Gh$ with $(j_w,x_w)=(n,1)$.

\smallskip

\noindent\underline{Induction Step:} Let $w\in \Gh$ be an $h$-admissible permutation with $0\leq j_w<n$ and assume that $(*)$ holds for all $u\in\Gh$ with $u\ll w$.

We divide the argument into the following four cases in accordance with the order $\ll$:
\[
\begin{array}{ll}
\text{(1)} & 0<j_w<n,\quad x_w=1,\\[1mm]
\text{(2)} & j_w=0,\quad x_w=1,\\[1mm]
\text{(3)} & j_w=0,\quad 1<x_w<m,\\[1mm]
\text{(4)} & j_w=0,\quad x_w=m.
\end{array}
\]
In case (1) with $\bb(w)=(k)$, case (2), and case (3), we will prove the congruence
\begin{equation}\label{eq3:main}
\widehat{\sigma}_w \equiv \sum_{u\in P(w)} c_u\sigma_u \pmod{M_{\ll w}},
\qquad c_u>0.
\end{equation}
The desired assertion in these cases then follows from Proposition~\ref{lem1:main}.
The remaining cases will be proved by adapting the arguments used in the base case.

\smallskip

\noindent\underline{Case (1): $0<j_w<n,\quad x_w=1$}\\ 
 In this case, we show the following statements using decreasing induction on $j$:  
\begin{itemize}
 \item    If $w\in \Gh^d$ satisfies $\bb(w)= (k)$, then 
$\displaystyle \sigma_w\in \sum_{u\in S_w}\mathbb C\Sn{k}(\widehat{\sigma}_u)$, where $S_w=\{w\}\cup\{u\in \Gh^d \mid u\ll w\}$.
\item If $w\in \Gh^d$ satisfies $\bb(w)\neq (k)$, then $\displaystyle \sigma_w\in \sum_{u\in \Gh^d,\, j_u\geq j_w}\mathbb C\Sn{k}(\widehat{\sigma}_u)$.
\end{itemize}   
We note that when $j_w=n$ and $x_w=1$, the assertions are proved in the Initial Step. Let $w\in \Gh^d$ have $j_w<n$ with $x_w=1$, and we assume that the assertions are true for all $j_w<j'\leq n$. Let $j\coloneqq j_w$.

We first consider the case when $\bb(w)=(k)$; that is $w$ is of type ii).
By Lemma~\ref{lem:stabilizer}(2) 
and Proposition~\ref{prop:coset representative}, 
\begin{equation}\label{eq:lem2main}
\begin{split}
\widehat{\sigma}_w=\sum_{v\in \Sn{k}} v\cdot \sigma_w & =(n-j+1)!(m+j-1)!\sum_{v\in \widetilde{S}} v\cdot \sigma_w\\
 &=(n-j+1)!(m+j-1)!\left(\sum_{u\in P(w)} \sigma_u+\sum_{v\in \widetilde{S}-S} v\cdot \sigma_w\right)\,,
\end{split}
\end{equation} where 
$$\widetilde{S}\coloneqq \{w_{k-1}(l_{k-1}) w_{k-2}(l_{k-2})\cdots w_{n-j+1}(l_{n-j+1}) \, |\, n-j+1\geq l_{n-j+1}\geq l_{n-j+2}\geq \cdots \geq l_{k-1} \geq 0\},$$ 
$$S\coloneqq \{w_{k-1}(l_{k-1}) w_{k-2}(l_{k-2})\cdots w_{n-j+1}(l_{n-j+1}) \, |\, n-j+1> l_{n-j+1}\geq l_{n-j+2}\geq \cdots \geq l_{k-1} \geq 0\}\,.$$

Letting $w^*\coloneqq (w_{n-j+1}(n-j))w\in P(w)$, we have   $$w^*= w_1 \cdots w_{m+j-2}\, 2 \,1 \,3 \,4\,\dots (n-j+2) \text{ and } (w_{n-j+1}(n-j+1))\cdot\sigma_w=s_1\cdot \sigma_{w^*}\,.$$
If we let $i\coloneqq m+j-2$ then $w^*$ admits an $i$-block decomposition with $^{\da i} w^*=(w_1\cdots w_{m+j-2})$. Using Theorem~\ref{thm:separation} and Proposition~\ref{prop:s_i}, we get
 $$s_1\cdot \sigma_{w^*}=\left(\sigma_{^{\da i}w}:(\sigma_{s_1}+\sum_{u\in \mathcal{A}_1}\left(\sigma_{u}-s_1\cdot\sigma_{u} \right))\right)\,,$$
where $\mathcal{A}_1=\{ u\in\Sn{n-j+2} \,|\, s_1\leq u, \ell_h(u)=1, u \dasharrow s_1u \}$. Since each $u\in \mathcal{A}_1$ satisfies $j_u>2$, we can conclude that 
$$s_1\cdot \sigma_{w^*}= \sigma_{w^*}+\sum_{u\in \mathcal{A}_1} (\sigma_{( ^{\da i}w:u)}+s_1\cdot  \sigma_{( ^{\da i}w:u)})\,.$$
For each $u\in \mathcal{A}_1$, $j_{(^{\da i}w:u)}$ is strictly greater than $j_w$ and by the induction hypothesis on $j$, we know that $\sigma_{( ^{\da i}w:u)}\in M_{\ll w}$ and thus $s_1\cdot\sigma_{( ^{\da i}w:u)}\in M_{\ll w}$. 
This shows that
\begin{equation}\label{eq:not1}(w_{n-j+1}(n-j+1))\cdot\sigma_w=s_1\cdot \sigma_{w^*}\equiv \sigma_{w^*} \pmod{M_{\ll w}}.
\end{equation}
Therefore, by Proposition~\ref{lem2:main}, we can conclude that
for nonnegative integers $d_u$'s,
$$\sum_{v\in\widetilde{S}-S}{v\cdot\sigma_w}\equiv \sum_{u\in P(w)} d_u\sigma_u \pmod{M_{\ll w}}.$$
Combining this with~\eqref{eq:lem2main}, we obtain
\[
\widehat{\sigma}_w=c\left(\sum_{u\in P(w)}\sigma_u + \sum_{v\in \widetilde{S}-S}v\cdot\sigma_w\right)\equiv c\left(\sum_{u\in P(w)}\sigma_u + \sum_{u\in P(w)}d_u\sigma_u\right) \pmod{M_{\ll w}}.
\]
Since $c$ is positive and $d_u$'s are nonnegative, (\ref{eq3:main}) holds, and we have 
$\displaystyle \sigma_w\in \sum_{u\in S_w}\mathbb C\Sn{k}(\widehat{\sigma}_u)$, where $S_w=\{w\}\cup\{u\in \Gh^d \mid u\ll w\}$ by Proposition~\ref{lem1:main}.

\smallskip

Now we consider the case when $\bb(w)\neq (k)$.
In this case, $0<b_\ell\coloneqq k-i\leq n$, $h^{\da }=h_{1, k-i-1}=(2,3,\dots, k-i, k-i)$ and $w^{\da} \in \mathcal{G}_{h^{\da }}$ is an $h^{\da }$-admissible permutation of type ii) or $w^{\da}=id$.
By the induction hypothesis on $k$ and Theorem~\ref{thm:separation},
\begin{equation}\label{eq:sigma_w}
\begin{split}
 \sigma_w=(\sigma_{^{\da}w}: \sigma_{w^{\da}})&\in \left(\left(\sum_{u\in \mathcal{G}_{^{\da{}}h}} \mathbb C\Sn{\{i+1, \dots k \}}\widehat\sigma_{u}\right):\left(\sum_{v\in S_{w^\da}}\mathbb C\Sn{k-i}\widehat{\sigma}_{v}\right)\right)\\
 &\subset \sum_{(u,v)\in \mathcal{G}_{^{\da{}}h}\times S_{w^\da}}\mathbb C\Sn{k}(\widehat\sigma_{u}:\widehat{\sigma}_{v}) \,,
 \end{split}
 \end{equation}
where $S_{w^\da}=\{w^\da\}\cup \{u\in \mathcal{G}_{h^{\da}}\,|\, j_u>j_{w^\da}\}$.
For convenience, we use the following notation:
\[\quad M_{\geq j}\coloneqq \sum_{u\in \Gh,\,j_u\geq j}\mathbb{C}\Sn{k}(\widehat{\sigma}_u).\]
For $(u,v)\in \mathcal{G}_{^{\da}h}\times S_{w^\da}$, we will show that $(\widehat\sigma_{u}:\widehat{\sigma}_{v})\in M_{\geq j}$, which proves $\sigma_w\in M_{\geq j}$. We write $(\widehat\sigma_{u}:\widehat{\sigma}_{v})$ as follows:
$$(\widehat\sigma_{u}:\widehat{\sigma}_{v})= \sum_{(v', v'')\in \mathcal{S}_u\times \mathcal{S}_v } \left((v'\cdot\sigma_u):(v''\cdot\sigma_{v})\right)=\sum_{(v', v'')\in \mathcal{S}_u\times \mathcal{S}_v } (v'v'')\cdot\sigma_{(u:v)}\,.$$
(a) Suppose that $v\in S_{w^\da}\setminus\{w^\da \}$. Then $j_v>j_{w^\da}$, so $j_{(u:v)}>j_w$. Thus, by the induction hypothesis on $j$, we have $\sigma_{(u:v)}\in M_{\ll w}$ and $(\widehat\sigma_{u}:\widehat{\sigma}_{v})\in M_{\ll w}\subset M_{\geq j}$.\\
(b) Suppose that $v=w^\da\in S_{w^\da}$. Then there are two subcases:\\
- If $u^\da$ is not of type i), then $\mathcal{S}_{(u:v)}=\mathcal{S}_u\times \mathcal{S}_v$, which implies that $(\widehat\sigma_{u}:\widehat{\sigma}_{v})=\widehat\sigma_{(u:v)}\in M_{\geq j}$.\\
- If $u^\da$ is of type i), then  $(u^\da: v)$ is of type ii). When 
$u^\da=u$, since $j_{(u:v)}=j_w=j$, we obtain $\sigma_{(u:v)}\in M_{\geq j}$ from the proof of the case when $j_w=j$ and $w$ is of type ii) in the above. When $u^\da\neq u$, we know $\mathcal{S}_{(u:v)}=\mathcal{S}_{^\da u}\times \mathcal{S}_{(u^\da : v)}$ since the last part of $^\da u$ cannot be of type i). It follows from the induction hypothesis on $k$, we know that $\sigma_{(u^\da : v)}\in M_{\geq j_{(u^\da:v)}}$. Hence we get $\sigma_{(u:v)}=\sigma_{(^\da u:(u^\da:v))}\in M_{\geq j}$.

\smallskip

\noindent\underline{Case (2): $j_w=0$ and $x_w=1$}\\ This is the case when $w$ is of type iii) with $x_w=1$. 
We will show that $s_1\cdot\sigma_{w^*}\equiv \sigma_{w^*}\pmod{M_{\ll w}}$, where $w^*=(w_{n+1}(n)) w\in P(w)$. Then by  Propositions~\ref{lem2:main} and~\ref{lem1:main} we can conclude that $\displaystyle\sigma_w\in \sum_{u\in S_w}\mathbb{C}\Sn{k}(\widehat{\sigma}_u)$, where $S_w=\{w\}\cup\{u\in \Gh^d \mid u\ll w\}$ and thus $\sigma_w\in M$.

Since $w^*\rightarrow s_1w^*$ and $s_1w^*$ is an $h$-admissible permutation,  by Lemma~\ref{lem:Atil}
we have 
\begin{align}\label{eq:s_1 w*}
s_1\cdot\sigma^T_{w^*} = \sigma^T_{w^*}+(t_2-t_1)\sigma^T_{s_1w^*} + \sum_{v  \in \mathcal A_{s_1, w^*}} (\tau_{v}-\tau_{s_1 v})\,,
\end{align}
where \(\mathcal{A}_{s_1, w^*}\) is a subset of $\widetilde{\mathcal{A}}_{s_1, w^*}=\{v\in [w^*, w_0] \,|\, v \dra s_1 v,\, \lh{v}\leq \lh{w^*} \}$. \\
In the following, we will show that $\displaystyle\sum_{v  \in \mathcal A_{s_1, w^*}} (\tau_{v}-\tau_{s_1 v})=\sum_{u\ll w}c_u\sigma_u^T$ for some constants $c_u$.
Note here that $w^*\in P(w)$, 
and $u\ll w$ if and only if $u\ll w^*$.

First, we consider the case when $\sigma^T_u$ appears in $\sum_{v  \in \mathcal A_{s_1, w^*}} \tau_{v}$. In this case, there exists $v\in \mathcal A_{s_1, w^*}\subset \widetilde{\mathcal{A}}_{s_1, w^*}$ satisfying $v\leq_h u$ and hence $w^*\leq u$ by Lemma~\ref{lem:tau2}. We note that $v\in A_{s_1, w^*}$ satisfies $(v^{-1})_1>m$ because $\mathcal A_{s_1, w^*}\subset \widetilde{\mathcal{A}}_{s_1, w^*}$ and $v\dra s_1v$. Since $v\leq u$, we get $(u^{-1})_1\geq (v^{-1})_1>m$, which implies that $u$ is neither of type iii) nor iv), which shows that 
$\sum_{v  \in \mathcal A_{s_1, w^*}} \tau_{v}=\sum_{u\ll w}a_u\sigma_u^T$ for some constants $a_u$.

Next, we consider the case when $\sigma^T_u$ appears in $\sum_{v  \in \mathcal A_{s_1, w^*}} \tau_{s_1v}$. Then there exists $v\in \mathcal A_{s_1, w^*}\subset \widetilde{\mathcal{A}}_{s_1, w^*}$ satisfying $s_1v\leq_h u$ and $w^*\leq v$  by Lemma~\ref{lem:tau2}. Hence, $((s_1v)^{-1})_2=(v^{-1})_1 > m$. We note that, since $1\in D_L(w^*)\cap D_L(v)$, we have $s_1w^*\leq s_1v$ and $s_1w^*\leq u$ by Proposition~\ref{prop:lifting}. Hence if we let $p=((s_1w^*)^{-1})_1$ then  $(v^{-1})_2=((s_1v)^{-1})_1\geq p$ and $(u^{-1})_1\geq p$.  

Suppose that $u\not\ll w$ and $\sigma^T_u$ appears in $\sum_{v  \in \mathcal A_{s_1, w^*}} \tau_{s_1v}$. Since $w$ is of type iii) with $x_w=1$, $u$ is of type iii) or iv) with $x_u\geq 1$. By Lemma~\ref{lem:u}(1), we know that  $u_m=x_u$. If $x_u\geq 2$, then $1,2\in\{u_1,\dots,u_m\}$, which contradicts $s_1v\leq u$ for some $v  \in \mathcal A_{s_1, w^*}$. Hence, we have $x_u=1$ and $u_m=1<u_{m+1}<\cdots<u_{m+n}$. Since $s_1v\leq u$, we get $u_{m+1}=2$. Moreover, since $s_1w^*\leq u$, we have $u_{m+n}\leq n+2$. Therefore, there exists a unique integer $y$ such that $\{u_m,u_{m+1},\dots,u_{m+n}\}=[n+2]-\{y\}$ for $y>2$.
Furthermore, $\ell_h(w^*)=\ell_h(w)=\ell_h(u)$ implies that the number of inversions in $w^*_1 \cdots w^*_m$ and the number of inversions in $u_1\cdots u_m$ must be the same, which because $u_m=w^*_m=1$, also implies that the number of inversions in  $w^*_1 \cdots w^*_{m-1}$ and the number of inversions in $u_1\cdots u_{m-1}$ must be the same. 
Now, noting that the number of inversions in $w^*_1 \cdots w^*_{m-1}$ and the number of inversions in $(s_1w^*)_1 \cdots (s_1w^*)_{m-1}$ are the same, $s_1w^*\leq u$, and 
\begin{align*}
&\{(s_1w^*)_1, \cdots, (s_1w^*)_{m-1}\}=\{1\}\cup [n+3, { k}], \\
&\{u_1, \cdots, u_{m-1}\}=\{y\}\cup [n+3,{ k}] \text{ for }  3\leq y\leq n+2\,,
\end{align*}
we conclude that 
\[ u_j=w^*_j=w_j \text{ for }  j\in[m-1]-\{p\},  \text{ and } u_p=y\,.\]
We can also conclude that the corresponding $v\in A_{s_1, w^*}$ satisfying $s_1w^*\leq s_1v\leq_h u$ satisfies 
\[ v_j=u_j\text{ for }  j\in[m+n]-\{p, m, m+1\}, \text{ and } v_p=2, v_m=y, v_{m+1}=1  \,.\]
One can check that these $u$'s are all in $P(w)$. For each $u$ we let $y_u\coloneqq u_p$. 
The values of permutations of interests are listed in Table~\ref{tab:permutation-values2}, where $y \in [3, n+2]$. 

\begin{table}[ht]
\centering
\begin{tabular}{c|ccccccccc}
\hline
& \(p\) & \(m\) & \(m+1\) & \(\cdots\) & \(i-1\) & \(i\) & \(\cdots\) & \(m+n\)\\
\hline
\(w=\phi(w^*)\)
& \(n+2\) & \(1\) & \(2\) & \(\cdots\) & \(y-1\) & \(y\) & \(\cdots\) & \(n+1\)\\
\(w^*\)
& \(2\) & \(1\) & \(3\) & \(\cdots\) & \(y\) & \(y+1\) & \(\cdots\) & \(n+2\)\\
\(u\)
& \(y\) & \(1\) & \(2\) & \(\cdots\) & \(y-1\) & \(y+1\) & \(\cdots\) & \(n+2\)\\
\(v\)
& \(2\) & \(y\) & \(1\) & \(\cdots\) & \(y-1\) & \(y+1\) & \(\cdots\) & \(n+2\)\\
\(\widetilde{s_1v}=\psi(s_1v)\)
& \(n+1\) & \(n+2\) & \(1\) & \(\cdots\) & \(y-2\) & \(y-1\) & \(\cdots\) & \(n\)\\
\(\psi(u)\)
& \(n+2\) & \(n+1\) & \(1\) & \(\cdots\) & \(y-2\) & \(y-1\) & \(\cdots\) & \(n\)\\
\hline
\end{tabular}
\caption{Values of the permutations $w,w^*,u,v,\widetilde{s_1v}=\psi(s_1v),\psi(u)$}
\label{tab:permutation-values2}
\end{table}

We show that the coefficient of $\sigma_u^T$ in the expansion of $\tau_{s_1v}$ is zero; that is, $\sigma_u$ where $u$ is of type iii) with $x_u=1$ does not appear in the sum $\sum_{v'  \in \mathcal A_{s_1, w^*}} \tau_{s_1v'}$. To compute the coefficient of $\sigma_u^T$ in the expansion of $\tau_{s_1v}$
we first consider all $\sigma_\pi^T$ with $u\in \supp(\sigma^T_\pi)$, that could appear in the sum $\sum_{v'  \in \mathcal A_{s_1, w^*}} (\tau_{v'}-\tau_{s_1v'})$. Let $\sigma_\pi^T$ where $u\in \supp(\sigma^T_\pi)$ appear in the sum $\sum_{v'  \in \mathcal A_{s_1, w^*}} (\tau_{v'}-\tau_{s_1v'})$. 
If $\sigma_\pi$ appears in  $\sum_{v'  \in \mathcal A_{s_1, w^*}} \tau_{v'}$ then $w^*<v'\leq \pi\leq u$ for some $v'\in \mathcal A_{s_1, w^*}$, and $v'_m$ must be $1$ while $((v')^{-1})_2\leq m+1$. This is not possible since $v'\dra s_1 v'$. If, on the other hand, $\sigma_\pi$ appears in  $\sum_{v'  \in \mathcal A_{s_1, w^*}} \tau_{s_1v'}$ \, then  $s_1w^*<s_1v'\leq \pi\leq u$ for some $v'\in \mathcal A_{s_1, w^*}$, and we know from the argument above that $v'=v$ is the only possible permutation. We further can see that since $u=s_{1,y} (s_1v)$ and $u_i=(s_1v)_i>y$ for all $i\in [p+1, m-1]$ there is no $\pi$ such that $s_1v< \pi < u$.
This means that there are only two possibilities: $\pi=s_1v$ and $\pi=u$.
Let $\sigma_{s_1v}$ and $\sigma_u$ appear in $\sum_{v'  \in \mathcal A_{s_1, w^*}} \tau_{s_1v'}$ with coefficients $d_1$ and $d_2$, respectively; then 
$$s_1\cdot\sigma^T_{w^*}(u) = \sigma^T_{w^*}(u)+(t_2-t_1)\sigma^T_{s_1w^*}(u) + d_1 \sigma^T_{s_1v}(u)+d_2\sigma^T_u(u)\,.$$
Remember that $\widetilde{w^*}=w$; and let $\phi\coloneqq w(w^*)^{-1}$. Then 
$\phi(s_1u)=\phi(u)_p=y-1$ while $\phi(s_1u)=\phi(u)_i=w_i$ for $i\in [p-1]$.  
Thus 
$s_1\cdot\sigma^T_{w^*}(u) = \sigma^T_{w^*}(u)=0$.
To compute $\sigma^T_{s_1w^*}(u)$, noting that $s_1w^*\in \Gh^d$ and $s_1w^*\leq_h u$, we look at the vertex $u$ in the  graph $\Gamma_{s_1w^*, h}=(V, E)$: Since 
$$\prod_{\{u, u'\}\in E-E_u}\alpha(\{u, u'\})=  \prod_{(i,j)\in S'} (t_{u_j}-t_{u_i})\,,$$  where
$S'=\{(i, j)\,|\, 1\leq i<j\leq m, u_i>u_j,\text{ and }(i,j)\neq (p,m)\}$ and $|S'|=\ell_h(s_1w^*)$, we can conclude that $\sigma^T_{s_1w^*}(u)= \prod_{(i,j)\in S'} (u_j-u_i)$ by Proposition~\ref{prop:computing the value}.
To compute $\sigma^T_{s_1v}(u)$ let $\widetilde{s_1v}=\psi(s_1v )$, then 
$$\sigma^T_{s_1v}(u)=\psi^{-1}(\sigma^T_{\psi(s_1v)}(\psi(u))) = \psi^{-1}\left( \prod_{(i,j)\in S''} t_{\psi(u)_j}-t_{\psi(u)_i}\right)\,,   $$
where $S''=S'\cup\{(m,m+1)\}$.
We note that $\psi(u)_i>\psi(u)_j$ is equivalent to $u_i>u_j$ for $1\leq i<j\leq m$. Thus, we have  
\begin{align*} 
\sigma^T_{s_1v}(u)&=\psi^{-1}\left( \prod_{(i,j)\in S''} t_{\psi(u)_j}-t_{\psi(u)_i}\right)\\
& = (t_{\psi^{-1}(n+1)}-t_{\psi^{-1}(1)})\prod_{(i,j)\in S'} (t_{u_j}-t_{u_i})=(t_1-t_2)\prod_{(i,j)\in S'} (t_{u_j}-t_{u_i})\,.
\end{align*}
From the computations above, we can conclude that 
\begin{align*}
d_2\sigma^T_u(u)&=d_2(t_1-t_y)\prod_{(i,j)\in S'} (t_{u_j}-t_{u_i})\\
 &=s_1\cdot\sigma^T_{w^*}(u) - \sigma^T_{w^*}(u)-(t_2-t_1)\sigma^T_{s_1w^*}(u) - d_1\sigma^T_{s_1v}(u)\\
 &= 0-0-(t_2-t_1)\prod_{(i,j)\in S'} (t_{u_j}-t_{u_i})-d_1(t_1-t_2)\prod_{(i,j)\in S'} (t_{u_j}-t_{u_i})\,,
 \end{align*}
and we can conclude that $d_1=1$ and $d_2=0$. Therefore, $\sigma_u^T$ with $u\not\ll w$ does not appear in $\sum_{v  \in \mathcal A_{s_1, w^*}} (\tau_{s_1 v})$.

Therefore, if $\sigma^T_u$ appears in $\sum_{v  \in \mathcal A_{s_1, w^*}} (\tau_{v}-\tau_{s_1 v})$ then $u\ll w$. Accordingly, $s_1\cdot \sigma_{w^*}\equiv \sigma_{w^*}\pmod{M_{\ll w}},$ which implies that $\sigma_w\in {M}$ by Proposition~\ref{lem1:main}.

\smallskip

\noindent\underline{Case (3): $j_w=0$ and $1<x_w<m$}\\
This is the case when $w$ is of type iii) with $x_w>1$. Let $x=x_w$.
We remind that $w_m\, w_{m+1}\, \cdots \, w_{k}= x\, (x+1)\, \cdots \, (x+n)$, and $H\coloneqq\Sn{\{ x-1, x, \dots, x+n \}} \times \Sn{\{ x+n+1, \dots, k \}}$ is a subgroup of $\Stab(\sigma_w)$ by Lemma~\ref{lem:stabilizer}(3). We also note that $s_i\cdot \sigma_w\equiv\sigma_w\pmod{M_{\ll w}}$ for $i\in [x-2]$ 
from Lemma~\ref{lem: iii)-1}, and
we have
\[P(w)=\{w_{k-1}(l_{k-1}) w_{k-2}(l_{k-2})\cdots w_{n+x}(l_{n+x})w \, |\, n\geq l_{n+x}\geq l_{n+x+1}\geq \cdots \geq l_{k-1}\geq 0 \}\]
from Lemma~\ref{lem:P(w)}.  We will show that $\widehat{\sigma}_w \equiv \sum_{u\in P(w)} c_u\sigma_u \pmod{M_{\ll w}}$ holds for positive $c_u$'s and use Proposition~\ref{lem1:main} to finish the proof.

Letting $H=\Sn{\{1, \dots, x+n \}} \times \Sn{\{ x+n+1, \dots, k \}}$, we have 
\begin{equation}\label{eq:H}\widehat{\sigma}_w=\sum_{u\in \Sn{k}} u\cdot \sigma_w\equiv (n+x)!(m-x)!\sum_{\bar{u}\in \Sn{k}/H}u\cdot \sigma_w  \pmod{{ M_{\ll w}}},\, \end{equation}
and $R\coloneq\{w_{k-1}(l_{k-1}) \cdots w_{x+n}(l_{x+n})\,|\, n+x\geq l_{x+n}\geq \cdots\geq l_{k-1}\}$ is a set of coset representatives of $\Sn{k}/H$.
Thus we consider the sum $\sum_{g\in R} g\cdot \sigma_w$.

If we let $R'\coloneqq \{w_{k-1}(l_{k-1}) \cdots w_{x+n}(l_{x+n})\,|\, n\geq l_{x+n}\geq \cdots\geq l_{k-1}\}\subset R$, then by Lemma~\ref{lem:P(w)}, 
we have 
\begin{equation}\label{eq:R'2} 
\sum_{g'\in R'} g'\cdot \sigma_w=\sum_{v\in P(w)} \sigma_v\,,
\end{equation}
and $R-R'=\{w_{k-1}(l_{k-1}) \cdots w_{x+n}(l_{x+n})\,|\,\,n+x\geq l_{x+n}>n \text{ and } l_{x+n}\geq \cdots\geq l_{m+n-1}\}\,. $ \\
Therefore we need to compute $\sum_{g\in R-R'} g\cdot \sigma_w$ and we first consider the case when $l_{x+n}=n+1$.

Note that $w_{x+n}(n+1)=s_xs_{x+1}\cdots s_{x+n-1}s_{x+n}=s_x w_{x+n}(n)$ and  $w_{x+n}(n+1)\cdot \sigma_w=s_x\cdot \sigma_{w^*}$ where $w^*=w_{x+n}(n) w\in P(w)$ satisfies $$w_m^*\, w_{m+1}^*\,w_{m+2}^* \cdots w_{m+n}^* =  x\, (x+2)\,( x+3)\,\cdots\, (x+n+1). $$
Since $w^*\rightarrow s_xw^*$ and $s_xw^*\in \Gh^d$,  we have $s_x \cdot\sigma_{w^*} = \sigma_{w^*} + \sum_{u   \in \mathcal A_{s_x, w^*}} (\tau_u-\tau_{s_x u})$ where \(\mathcal{A}_{s_x, w^*}\) is a subset of $\widetilde{\mathcal{A}}_{s_x, w^*}=\{u\in [w^*, w_0] \,|\, u \dra s_x u,\, \lh{u}\leq \lh{w^*} \}$ by Lemma~\ref{lem:Atil}. We claim that if $\sigma_v$ appears in the sum  $\sum_{u  \in \mathcal A_{s_x, w^*}} (\tau_u-\tau_{s_x u})$ then $v\ll w$.

Suppose that there exists $v\not\ll w$ such that $\sigma_v$ appears in $\sum_{u  \in \mathcal A_{s_x, w^*}} (\tau_u-\tau_{s_x u})$. Then there exists $u\in\widetilde{\mathcal{A}}_{s_x, w^*}$ such that $u\leq_h v$ or $s_xu\leq_h v$. Since $v\not\ll w$, $v$ is of type iii) or iv) with $x_v\geq x$. 
Hence, we get $v_m=x_v<v_{m+1}<\cdots<v_k$
where we note that $v_m=x_v$ comes from Lemma~\ref{lem:u}(1).
On the other hand, since $u\dra s_xu$, we have $(u^{-1})_x>m$ and $(u^{-1})_{x+1}\leq m$, which implies $u\not<v$. Thus $s_xu\leq_h v$. Furthermore, since $x\in D_L(u)-D_L(s_xw^*)$ and $s_xw^*<u$, we get $w^*<s_xu\leq_h v$, so  $v_m\leq x$ and $v_k\leq x+n+1$, which implies that $x_w=v_m=x$ and $\{v_i\mid i=m+1,\dots,k\}\subset [x+1,x+n+1]$.
Note that from $s_xw^*<s_xu\leq_h v$, we get
\begin{align*}
    \{(s_xw^*)_i\mid i=m+1, \dots, k\}\up &=\{x+2, x+3, \dots, x+n+1\}\up\\
    &\geq \{(s_xu)_i\mid i=m+1, \dots, k\}\up\\
    &\geq \{v_i \mid i=m+1, \dots, k\}\up.
\end{align*}
Since $x+1\in \{(s_xu)_i\mid i=m+1, \dots, k\}\up $, there exists  $y\in [x+2, x+n+1] $ such that $\{(s_xu)_i\mid i=m+1, \dots, k\}=\{ x+1\}\cup ([x+2, x+n+1]-\{y\})$. Then $ \{(s_xw^*)_i\mid i=m, \dots, k\}\up=\{x+1, x+2, \dots, x+n+1\}\up\geq \{(s_xu)_i\mid i=m, \dots, k\}\up$ and $\{(s_xu)_i\mid i=m, \dots, k\}\up \geq \{v_i \mid i=m, \dots, k\}\up$, where $\{v_i \mid i=m, \dots, k\}\subset [x, x+n+1]$, which implies that $(s_xu)_m=y$ since $x\not\in \{(s_xu)_i\mid i=m, \dots, k\}$. We can conclude that $$\{(s_xu)_i\mid i=m, \dots, k\}=[x+1, x+n+1]=\{ (s_xw^*)_i\mid i=m, \dots, k\},$$ which contradicts $w^*<s_xu$ because $\{(w^*)_i\mid i=m,m+1,\dots,k\}$ contains $x$.  
Therefore, we obtain $s_x\cdot\sigma_{w^*}=\sigma_{w^*}+\sum_{v\ll w}c_v\sigma_v$, which implies that $$w_{x+n}(n+1)\cdot \sigma_w=s_x \cdot\sigma_{w^*} \equiv \sigma_{w^*}  \pmod{ M_{\ll w}}.$$  We then consider $w_{x+n+1}(l_{x+n+1})w_{x+n}(n+1)$.    Since 
\begin{align*}
w_{x+n+1}(l_{x+n+1})w_{x+n}(n+1)\cdot \sigma_w&\equiv w_{x+n+1}(l_{x+n+1})\cdot\sigma_{w^*} \pmod{M_{\ll w}}\\
&= w_{x+n+1}(l_{x+n+1})w_{x+n}(n)\cdot \sigma_{w}\,,
\end{align*}
if $l_{x+n+1}\neq n+1$ then $w_{x+n+1}(l_{x+n+1})w_{x+n}(n+1)\cdot \sigma_w\equiv\sigma_v\pmod{M_{\ll w}}  $ for some $v\in P(w)$. We thus let $l_{x+n+1}= n+1$; then 
\begin{align*}
w_{x+n+1}(n+1)w_{x+n}(n+1)\cdot \sigma_w &\equiv s_{x+1}s_{x+2}\cdots s_{x+n+1}w_{x+n}(n)\cdot \sigma_{w} \pmod{M_{\ll w}}\\
&= (s_{x+1}s_{x+2}\cdots s_{x+n+1})(s_{x+1}\cdots s_{x+n})\cdot \sigma_{w}\\
&=  (s_{x+2}\cdots s_{x+n+1})(s_{x+1}\cdots s_{x+n}s_{x+n+1})\cdot \sigma_{w}\\
&= (s_{x+2}\cdots s_{x+n+1})(s_{x+1}\cdots s_{x+n})\cdot \sigma_{w}\\
&= w_{x+n+1}(n) w_{x+n}(n)\cdot \sigma_{w}\,,
\end{align*}
where the third equality is due to Lemma~\ref{lem:wl} and the fourth equality is valid because $s_{x+n+1}$ stabilizes $\sigma_{w}$. From the above observation and Lemma~\ref{lem:P(w)}, we obtain that  
\begin{equation}\label{eq:iii)-1} w_{x+n+1}(n+1)w_{x+n}(n+1)\cdot \sigma_w\equiv\sigma_u \pmod{M_{\ll w}}  \text{ for some } u\in P(w)\,.\end{equation} 
We can follow the same argument to show that for $n+1\geq l_{n+x+1}\geq \cdots \geq l_{k-1}$, we have $$w_{k-1}(l_{k-1}) \cdots w_{x+n}(n+1)\cdot \sigma_w\equiv\sigma_v \pmod{M_{\ll w}}  \text{ for some } v\in P(w)\,.$$ 
Now we consider $l_{x+n}=n+2$ and continue to use $w^*=w_{x+n}(n) w\in P(w)$ as above. Then $w_{x+n}(n+2)\cdot \sigma_w=s_{x-1}w_{x+n}(n+1)\cdot \sigma_{w}\equiv s_{x-1}\cdot\sigma_{w^*} \pmod{M_{\ll w}} $. Since $s_{x-1}w^*\rightarrow w^*$, we have $s_{x-1}\cdot\sigma_{w^*}=\sigma_{w^*}$ that implies
$w_{x+n}(n+2)\cdot\sigma_w =(s_{x-1}s_x)\cdot\sigma_{w^*}\equiv\sigma_{w^*}\pmod{M_{\ll w}} $.
Since 
$$w_{x+n+1}(l_{x+n+1})w_{x+n}(n+2)\cdot \sigma_w\equiv w_{x+n+1}(l_{x+n+1})w_{x+n}(n)\cdot\sigma_{w}  \pmod{M_{\ll w}}\,,$$ we obtain from Lemma~\ref{lem:P(w)}  
and (\ref{eq:iii)-1}) that 
$w_{x+n+1}(l_{x+n+1})w_{x+n}(n+2)\cdot \sigma_w\equiv \sigma_u \pmod{M_{\ll w}}$ for some $u\in P(w)$ if $l_{x+n+1}\leq n+1$. If $l_{x+n+1}=n+2$ then 
\begin{align*}
w_{x+n+1}(n+2)w_{x+n}(n+2)\cdot \sigma_w &\equiv s_xs_{x+1}\cdots s_{x+n+1}w_{x+n}(n+1)\cdot \sigma_{w}   \pmod{M_{\ll w}}\\
&= (s_x s_{x+1}\cdots s_{x+n+1})(s_{x}\cdots s_{x+n})\cdot \sigma_{w}\\
&=  (s_{x+1}\cdots s_{x+n+1})(s_{x}\cdots s_{x+n}s_{x+n+1})\cdot \sigma_{w}\\
&= (s_{x+1}\cdots s_{x+n+1})(s_{x}\cdots s_{x+n})\cdot \sigma_{w}\\
&= w_{x+n+1}(n+1) w_{x+n}(n+1)\cdot \sigma_{w}\,,
\end{align*}
where the third equality is due to Lemma~\ref{lem:wl} and the fourth equality is valid because $s_{x+n+1}\in \Stab(\sigma_w)$.
Therefore, we have $w_{x+n+1}(n+2)w_{x+n}(n+2)\cdot \sigma_w\equiv \sigma_u  \pmod{M_{\ll w}} $ for some $u\in P(w)$. 
We can further show that for $n+2\geq l_{n+x+1}\geq \cdots \geq l_{m+n-1}$, we have $$w_{m+n-1}(l_{m+n-1}) \cdots w_{x+n}(n+2)\cdot \sigma_w\equiv \sigma_v \pmod{M_{\ll w}} \text{ for some } v\in P(w)\,.$$ 
By repeating the same calculation, we can show that
\begin{equation}\label{eq:n+a} 
w_{k-1}(l_{k-1}) \cdots w_{x+n}(n+a)\cdot \sigma_w\equiv \sigma_v   \pmod{M_{\ll w}} \text{ for some } v\in P(w)\,,
\end{equation}
for any $1\leq a \leq x$ and $n+a\geq l_{n+x+1}\geq \cdots \geq l_{k-1}$.
Consequently, together with (\ref{eq:H}) and (\ref{eq:R'2}), we obtain
\begin{equation*}
\widehat{\sigma}_w=(n+x)!(m-x)!\sum_{g\in R} g\cdot \sigma_w\equiv \sum_{u\in P(w)} c_u \sigma_u   \pmod{M_{\ll w}} \text{ for positive integers $c_u$'s}\,.
\end{equation*}
Therefore, $\sigma_w\in M$ by Proposition~\ref{lem1:main}.

\smallskip

\noindent\underline{Case (4): $j_w=0$ and $x_w=m$.}\\ This is the case when $w$ is of type iv). We recall from Lemma~\ref{lem:stabilizer}(4) that $\Sn{\{m-1,m,\dots,k\}}$ is a subgroup of $\Stab(\sigma_w)$; from Lemma~\ref{lem: iii)-1}, we have $s_i\cdot\sigma_w\equiv\sigma_w\pmod{M}$
for each $i\in [m-2]$. Hence, we have
$$\widehat{\sigma}_w=\sum_{v\in\Sn{k}}v\cdot\sigma_w\equiv k!\sigma_w\pmod{M},$$
which implies that $\sigma_w\in M$.

\begin{flushright} {\bf End of Proof of Theorem~\ref{thm:main}.}
\end{flushright}

\section*{Acknowledgements}
	The authors are grateful to JiSun Huh for sharing insights and for helpful discussions.
			

\end{document}